\documentclass[11pt,activeacute,amsfonts]{amsart}
\usepackage{a4wide}
\usepackage[english]{babel}
\usepackage{inputenc}
\usepackage{mathabx}
\usepackage{amsmath,amsthm,amsfonts,amssymb,amscd}
\usepackage{amssymb,amsfonts}
\usepackage[all,arc]{xy}
\usepackage{enumerate}
\usepackage{mathrsfs}
\usepackage{marginnote}
\usepackage{mathtools}
\usepackage{amsthm,thmtools,xcolor}

\usepackage[all]{xy}
\usepackage{tikz-cd}
\usetikzlibrary{cd}
\usepackage[square,sort,comma,numbers,nonamebreak]{natbib}

\theoremstyle{plain}
\newtheorem{thm}{Theorem}[section]

\newtheorem{lem}{Lemma}[section]

\newtheorem{lemx}{Lemma}

\theoremstyle{definition}

\theoremstyle{remark}
\newtheorem{rem}{\textit{Remark}}[section]

\usepackage{color}

\usepackage[colorlinks = true,
linkcolor = blue,
urlcolor  = blue,
citecolor = blue,
anchorcolor = blue]{hyperref}
\usepackage{hyperref}
\makeatletter
\let\c@equation\c@thm
\makeatother
\numberwithin{equation}{section}

\DeclareMathOperator{\supp}{\mathrm{supp}}
\newcommand\underrel[3][]{\mathrel{\mathop{#3}\limits_{%
			\ifx c#1\relax\mathclap{#2}\else#2\fi}}}
\title[Zakharov-Kuznetsov decay]{On local energy decay for large solutions of the Zakharov-Kuznetsov equation}

\author[A. M\'endez]{Argenis J. Mendez$^{*}$}
\address{Center for mathematical Modeling, Universidad de Chile, Santiago de Chile.}
\email{amendez@dim.uchile.cl}
\thanks{$^{*}$ A.M. work was partially supported by CMM Conicyt PIA AFB170001.}

\author[C. Mu\~noz]{Claudio Mu\~noz$^{\lor}$}
\address{CNRS and Departamento de Ingenier\'{\i}a Matem\'atica and Centro
de Modelamiento Matem\'atico (UMI 2807 CNRS), Universidad de Chile, Casilla
170 Correo 3, Santiago, Chile.}
\email{cmunoz@dim.uchile.cl}
\thanks{$^{\lor}$ C. M. work was funded in part by Chilean research grants ANID FONDECYT 1191412, project France-Chile
ECOS-Sud C18E06, MathAmSud EEQUADDII 20-MATH-04, and CMM Conicyt PIA AFB170001.}

\author[F. Poblete]{Felipe Poblete$^{\ocirc}$}
\address{Instituto de Ciencias F\'isicas y Matem\'aticas,
Facultad de Ciencias, Universidad Austral de Chile,
    Valdivia, Chile.}
\email{felipe.poblete@uach.cl}
\thanks{$^{\ocirc}$ F. P. work  is partially supported by  ANID projects FONDECYT/Iniciaci\'on 11181263 and FONDECYT/Regular 1170466.}

\author[J. C. Pozo]{Juan Carlos Pozo$^{\land}$}
\address{Departamento de Matem\'aticas, Facultad de Ciencias, Universidad de Chile.}
\email{jpozo@ug.uchile.cl}
\thanks{$^{\land}$ J.C.P. work was partially supported by Chilean research grant FONDECYT 1181084.}


\date{\today}
\commby{Argenis Mendez}
\begin{document}

\begin{abstract}
We consider the Zakharov-Kutznesov (ZK) equation posed in $\mathbb R^d$, with $d=2$ and $3$. Both equations are globally well-posed in $L^2(\mathbb R^d)$. In this paper, we prove local energy decay of global solutions: if $u(t)$ is a solution to ZK with data in $L^2(\mathbb R^d)$, then
\[
\liminf_{t\rightarrow \infty}\int_{\Omega_d(t)}u^{2}({\bf x},t)\mathrm{d}{\bf x}=0,
\]
for suitable regions of space $\Omega_d(t)\subseteq \mathbb R^d$ around the origin, growing unbounded in time, not containing the soliton region. We also prove local decay for $H^1(\mathbb R^d)$ solutions. As a byproduct, our results extend decay properties for KdV and quartic KdV equations proved by Gustavo Ponce and the second author. Sequential rates of decay and other strong decay results are also provided as well.
\end{abstract}

\maketitle

\section{Introduction and Main Results}

In this paper we consider the Zakharov-Kutznesov equation (ZK),
\begin{equation}\label{ZK:Eq}
\partial_t u+\partial_x\Delta u+u\partial_x u=0,  \tag{ZK}
\end{equation}
where $u=u({\bf x},t)\in\mathbb{R}$, $t\in\mathbb{R}$ and ${\bf x}\in\mathbb{R}^d$, with $d=2,3$. In Physics, \eqref{ZK:Eq} arises as an asymptotic model of wave propagation in a magnetized plasma,  see e.g. \cite{Sulem:Sulem:1999,Bittencourt:1986}. It was originally proposed by Zakharov and Kuznetsov in \cite{ZK} for $d = 3$, see also \cite{Laedke:Spatschek:1982} for a formal derivation. Additionally, equation \eqref{ZK:Eq}  is a natural multi-dimensional generalization of the Korteweg-de Vries (KdV) equation (which is the case $d=1$). In particular, it has solitary wave solutions, or solitons (see Subsection \ref{literatura}).

\medskip

Very recently it was proved that the two and three dimensional ZK equations are globally well-posed in $L^2$ and $H^1$ \cite{Kinoshita:Arxiv:2019,Herr:Kinoshita:Arxiv:2020}. This global behavior can be seen as a consequence of the fact that \eqref{ZK:Eq} enjoys conservation of mass and energy:
\[
\int_{\mathbb R^d} u^2({\bf x} ,t)d{\bf x}=\hbox{const.}, \qquad \frac12 \int_{\mathbb R^d} |\nabla u|^2({\bf x} ,t)d{\bf x} - \frac13 \int_{\mathbb R^d} u^3({\bf x} ,t)d{\bf x}=\hbox{const.}
\]
In this paper, our main goal is to study the asymptotic behavior of these global ZK solutions, under minimal assumptions (essentially, data only in $L^2(\mathbb R^d)$ or $H^1(\mathbb R^d)$). In particular, we shall describe the dynamics in local regions of space where solitons are absent. No scattering seems to be available for 2D and 3D \eqref{ZK:Eq}, not even in the small data case, except if the nonlinearity is sufficiently large \cite{MR2950463}.

\medskip

First, we consider the two dimensional case. We always assume $t\geq 2$. Let $\Omega(t)$ denote the following rectangular box (see Fig. \ref{fig:1})
\begin{equation}\label{brconditions}
	\Omega(t):=  \left\{(x,y)\in\mathbb{R}^{2}\, ~ \Big|\, ~  |x| < t^b \,\wedge |y|< t^{br} \right\}, \quad  \frac13<r<  3,    \quad 0<b< \frac{2}{3+r}.
\end{equation}
(Note that $b<\frac35$ and $br<1$.) Under rough data assumptions, we show decay along a sequence of times.

\begin{thm}[$L^2$-decay in 2d]\label{Thmdim2L2}
	Suppose that $u_0\in L^2(\mathbb R^2)$ and let $u=u(x,y,t)$ be the bounded in time solution  to 2D  \eqref{ZK:Eq} such that $u\in C\left(\mathbb{R}: L^2(\mathbb{R}^{2})\right)$. Then
	\begin{equation}\label{cero2D}
	\liminf_{t\rightarrow \infty}\int_{\Omega(t)}u^{2}(x,y,t)\mathrm{d}x\,\mathrm{d}y=0.
	\end{equation}
	Moreover, there exist constant $ C_0>0$ and an increasing sequence of times $t_n\to +\infty$ such that
	\begin{equation}\label{secuencia2D}
	\int_{\Omega(t_n)}u^{2}(x,y,t_n)\mathrm{d}x\,\mathrm{d}y\leq \frac{C_0}{\ln^{\frac{1}{b}-1} (t_n) }.
	\end{equation}
\end{thm}
The previous result holds for arbitrarily large data in $L^2$, despite the fact that 2D ZK is scattering critical (the standard scattering trick is $u\partial_xu \sim \frac1t u$, see Faminskii \cite{Faminskii:1995} for required linear decay estimates). We also present in \eqref{secuencia2D} a mild decay rate valid along a sequence of times growing to infinity. No $L^\infty$ decay seems reasonable here because the $H^s$ regularity needed is at least $s>1$. Since $\Omega(t)$ grows with time, it contains any compact region in $\mathbb R^2$, but it does not contain the soliton region $x\sim t$. However, combining \eqref{cero2D} with the asymptotic stability of the 2D ZK soliton proved in \cite{Cote:Munoz:Pilod:Simpson:2016} (see also Subsection \ref{literatura}), a better description of the soliton dynamics is obtained. In that sense, Theorem \ref{Thmdim2L2} and the results below can be understood as one step forward the validity of the soliton resolution conjecture for \eqref{ZK:Eq}. We also state in Lemmas \ref{lemdecay} and \ref{lemmedida} some interesting consequences of \eqref{cero2D}, which we believe are of independent interest.

\medskip

\begin{figure}[h!]
\begin{center}
\begin{tikzpicture}[scale=0.7]
\filldraw[thick, color=lightgray!30] (0,0) -- (4,0) -- (4,4) -- (0,4) -- (0,0);
\draw[thick, dashed] (0,0) -- (4,0) -- (4,4) -- (0,4) -- (0,0);
\draw[thick,dashed] (5,-1) -- (5,5);
\draw[->] (-1,2) -- (6,2) node[below] {$x$};
\draw[->] (2,-1) -- (2,5) node[right] {$y$};
\node at (5.7,2.5){$x\sim t$};
\node at (2.5,3.7){$t^{br}$};
\node at (3.7,2.5){$t^b$};
\node at (1,1){$\Omega(t)$};
\end{tikzpicture}
\qquad
\begin{tikzpicture}[scale=0.7]
\filldraw[thick, color=lightgray!30] (1,1) -- (4,1) -- (4,4) -- (1,4) -- (1,1);
\draw[thick, dashed] (1,1) -- (4,1) -- (4,4) -- (1,4) -- (1,1);
\draw[thick,dashed] (5,-1) -- (5,5);
\draw[->] (-1,0) -- (6,0) node[below] {$x$};
\draw[->] (0,-1) -- (0,5) node[right] {$y$};
\node at (5.8,2.3){$x\sim t$};
\node at (2.5,3.7){$t^{br}$};
\node at (3.7,2.5){$t^b$};
\node at (2,2){$\Omega(t)$};
\node at (2.5,0){$\bullet$};
\node at (2.5,0.5){$t^m$};
\node at (0,2.5){$\bullet$};
\node at (0.5,2.5){$t^n$};
\end{tikzpicture}
\end{center}
\caption{(\emph{Left}). Schematic figure depicting the set $\Omega(t)$, in the centered case, as defined in \eqref{brconditions}. Recall that $\frac13<r<3$, $0\leq b<\frac{2}{3+r}$ and $0\leq br<\frac{2r}{3+r}$. This set corresponds to the region of the plane where Theorem  \ref{Thmdim2L2} holds. Some important limiting cases are $r\approx \frac13$, for which $b<\frac35$ and $br<\frac15$; and $r\approx 3$, for which $b<\frac13$ and $br<1$. Here $x\sim t$ represents the soliton region. Theorem \ref{Thmdim2H1} requires $1<r<3$. (\emph{Right}). Schematic representation of $\Omega(t)$ in the non-centered case, see Remark \ref{NON} for the values of $m$ and $n$.}\label{fig:1}
\end{figure}
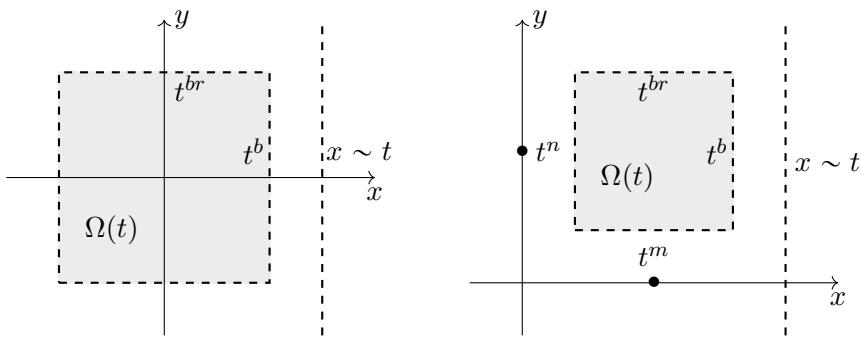

\begin{rem}\label{NON}
Theorem \ref{Thmdim2L2} can be extended to the non-centered case with some minor modifications. Indeed, \eqref{cero2D} still holds if $\Omega(t)$ is given by the expression
\begin{equation}\label{brconditions_non}
\begin{aligned}
	\Omega(t):= &~{}\Omega_{\delta_{1},\delta_{2}}(t):=  \left\{(x,y)\in\mathbb{R}^{2}\, ~ \Big|\, ~  |x \pm t^m |< t^{b} \,\wedge |y \pm t^n |< t^{br} \right\},\\
	&~{}   \frac13<r<  3,    \quad 0<b< \frac{2}{3+r}, \quad 0\leq m< 1-\frac12 b(1+r), \quad 0\leq n< 1-\frac12 b(3-r). 
	   \end{aligned}
\end{equation}
Note that when $r\approx 3$, one has the maximum value for $n$, which is $\approx 1$. At the same time, $br \approx 1$, which makes sense with the fact that one cannot go further proving decay in the $y$ variable, not more than the maximum value of the centered case. However, by making the rectangle smaller if necessary, one can go further in the $x$ variable: take $b$ small; since $m\approx 1$, one can reach the soliton limit $x\sim t$, but the size of the decay window must be small. Clearly one improves the region of decay obtained in the centered case, which is $b<\frac35$; see \eqref{brconditions}. See Subsection \ref{Non} for the proofs.
\end{rem}

\begin{rem}
The area of the region $\Omega(t)$ is not preserved with respect to variations of the parameters $b$ and $r$. The supremum value of the area is obtained in the limit $r= 3$ and $b=\frac13$, which is $t^{\frac43}$. 
\end{rem}

\medskip

Obtaining the remaining $\limsup$ property is left here as an open question, even in the small data case. The related problem about the maximum size of $\Omega(t)$ is very relevant here, since the global $L^2$ norm is always conserved, and positive for nontrivial solutions. One could guess that for $\Omega(t)$ large enough,
\[
\limsup_{t\rightarrow \infty}\int_{\Omega(t)}u^{2}(x,y,t)\mathrm{d}x\,\mathrm{d}y>0,
\]
and therefore a smaller $\Omega(t)$ than in our results is probably needed. In this direction, if $u(t)$ has better decay properties, such as being in $L^\infty ([0,\infty),L^1(\mathbb R^2))$, then the zero $\limsup$ part can be recovered following \cite{MR3936126}. However, having such strong decay is extremely far from being known in the ZK case, except if the solution is a soliton.

\medskip

Coming back to \eqref{cero2D}, some key results in the dispersive PDE literature have been established primarily via a sequence of times. We mention the work by Duyckaerts, Kenig, Jia and Merle \cite{MR3678502} for the proof of the soliton resolution conjecture in the focusing, energy critical wave equation. Unlike the wave equation, our problem is energy subcritical in nature, and of infinite speed of propagation. In particular, an infinite number of solitons could emerge from large $L^2$ data (see \cite{Valet} for the existence of multi-solitons). The conclusion stated in \eqref{cero2D} was used by Tao in \cite{MR2091393} as a condition to prove a weak form of soliton resolution for cubic focusing NLS in 3D under bounded finite $H^1$ norm. Note that the problem considered by Tao is mass supercritical, but energy subcritical, and these restrictions are key  whenever scattering is treated. The $L^2$ subcritical condition on \eqref{ZK:Eq}, and more importantly the scattering critical condition, make things definitely more subtle. See also \cite{MR1616917,MR1047566} for early but fundamental results involving sequential in time convergence.

\medskip

Recovering the decay of the gradient of $u$ in the case of $H^1$ data requires a slight modification of the parameters in \eqref{brconditions}.
\begin{thm}[$H^1$-decay in 2d]\label{Thmdim2H1}
	Suppose additionally that $u_0\in H^1(\mathbb R^2)$, and let $u=u(x,y,t)$ be the solution  to 2D ZK \eqref{ZK:Eq} such that $u\in C\left(\mathbb{R}: H^1(\mathbb{R}^{2})\right)$.
	If now $1<r<3$ in \eqref{brconditions}, one has
	\begin{equation*}
	\liminf_{t\rightarrow \infty}\int_{\Omega(t)} (u^2+ |\nabla u|^{2})(x,y,t)\mathrm{d}x\,\mathrm{d}y=0.
	\end{equation*}
	A similar decay rate as in \eqref{secuencia2D} also holds in this case.
\end{thm}

One could guess that by employing the Gr\"unrock-Herr's dilation/rotation trick \cite{Grunrock:Herr:2015} on the ZK variables, the extra condition $r>1$ may be lifted, but as of today it is not clear to us that  such an improvement is possible.

\medskip

The techniques required for the proof of Theorems \ref{Thmdim2L2} and \ref{Thmdim2H1} are reminiscent of the works by Ponce and the second author \cite{MR3936126,MR4021089} in the case of 1D KdV and Benjamin-Ono equations (see also \cite{MR3960140,ACKM} for applications to other 1D models). Here we deal with the ZK model, which contains additional difficulties because of the higher dimension. Additionally, in this paper we lift the $L^1$ conditions posed in \cite{MR3936126,MR4021089} and consider data only in the energy space $L^2$ or $H^1$. The rates of decay that we obtain (see e.g. \eqref{secuencia2D}) are clearly weaker than the ones obtained by assuming much more regularity and decay on the initial data, but in the vastly energy space, it is hard to think about a possible universal rate of decay. Finally, the proof works equally for quadratic and quartic KdV in 1D as well, with some minor modifications, giving relative improvements to the results stated in \cite{MR3936126} (the $L^\infty_t L^1_x$ condition on the solution is lifted, at the expense of only having liminf in the decay property).

\medskip

Indeed, consider quadratic and quartic KdV equations
\begin{equation}\label{gKdV}
\partial_t u + \partial_x (\partial_x^2 u + u^p)=0, \quad p=2,4, \quad u=u(x,t)\in\mathbb R, \quad t,x\in\mathbb R.
\end{equation}
The IVP for these problems is very well-known, global solutions are known for $L^2$ and $H^1$ data, see \cite{MR3308874} for instance. Define
\begin{equation}\label{final_gKdV}
\Omega(t):=\left\{ x\in\mathbb R ~ :~ |x \pm t^n | < t^b \right\},  \quad 0<b< \frac{p}{2p-1}, \quad 0\leq n< 1-\frac{b}2.
\end{equation}
(The $\pm$ signs are considered at the same time.) For this set, we have the following large data sequential decay.
\begin{thm}[Decay in gKdV]\label{Thmdim2L2_KDV}
	Suppose that $u_0\in L^2(\mathbb R)$ if $p=2$, and $u_0\in H^1(\mathbb R)$ if $p=4$. Let $u=u(x,t)$ be the solution to  \eqref{gKdV}. Then
	\begin{equation}\label{cero2D_KdV}
	\liminf_{t\rightarrow \infty}\int_{\Omega(t)}u^{p}(x,t)\mathrm{d}x=0.
	\end{equation}
\end{thm}
A similar sequential rate of decay as in \eqref{secuencia2D} can be obtained as well. Note that by making $b$ smaller if necessary, one can almost reach the soliton region $x\sim t$. From the proof itself, if the data is in $H^1$, \eqref{cero2D_KdV} also holds for nonintegrable perturbations of KdV, of the form $u^2+ o(u^2)$, following  \cite{MR3936126}. The proof for cubic KdV ($p=3$, the so-called mKdV) does not work for obvious reasons: there exist periodic-in-time solutions around zero, spatially localized in the Schwartz class, called breathers, which do not decay. See \cite{MR3116324} and references therein for more details on that important case. Finally, for $p=4$, thanks to the Cauchy-Schwarz inequality and \eqref{cero2D_KdV}, there is decay of the $L^2$ norm along sequential times, on any fixed compact set of space. See \cite{MR3936126} and references therein for a detailed description of the gKdV Cauchy problem and the corresponding scattering results. 

\medskip

Let us explain in more detail the idea of proof in Theorems \ref{Thmdim2L2} and \ref{Thmdim2H1}. Given an $L^2$ solution $u(x,y,t)$ of \eqref{ZK:Eq}, we introduce an $L^1$ virial-type functional of the form\footnote{Recall that $\int_{x,y} u$ is formally conserved, and scattering critical for the \eqref{ZK:Eq} scaling $\lambda^2 u( \lambda x,\lambda y,\lambda^2 t )$.}
\begin{equation}\label{Virial}
\Xi(t)=\frac{1}{\eta(t)}\int_{\mathbb{R}^{2}}u(x,y,t) \psi \left(\frac{x}{\lambda_{1}(t)}\right)\phi \left(\frac{x}{\lambda_{1}^q(t)}\right)\phi \left(\frac{y}{\lambda_{2}(t)}\right)\,\mathrm{d}x\mathrm{d}y, \quad q>1,
\end{equation}
in the spirit of Bona-Souganidis-Strauss \cite{MR897729} and Martel-Merle \cite{MR1896235}. Here $\psi$ denotes a smooth increasing and bounded function (e.g. $\tanh$), and $\phi$ is a very localized function (assume $\phi=\hbox{sech}$). The time-dependent parameters $\lambda_1(t)$, $\lambda_2(t)$ and $\eta(t)$ are key for the proof, and will be chosen following special requirements, related to the structure of the spatial region $\Omega(t)$ described in \eqref{brconditions}, among other not less important conditions. There are some differences between $\Xi(t)$ here and the same functional introduced in \cite{MR3936126}, the most important being the double localization in $x,y$ via the function $\phi$, and the introduction of the compensation function $\eta(t)$ (first introduced in \cite{MR4021089}). These procedures make possible to give a meaning for $\Xi(t)$ even for $L^2$ data, but introduce plenty of new error terms that must be controlled with care. This is done by using appropriate choices for $q>1$ and $\lambda_2(t)$ in terms of $\eta(t)$ and $\lambda_1(t)$. Once it is proved that $\Xi(t)$ makes sense, we show that the local $L^2$ integral bound
\[
\int_{\{t\gg1\}} \frac{1}{t\ln t}	
\left(\int_{\Omega(t)} u^{2}(x,y,t) \,\mathrm{d}x\mathrm{d}y\right)\mathrm{d}t \leq C_0 <\infty,
\]
is valid no matter the size of $u$. This last bound is essentially the statement in Theorem \ref{Thmdim2L2},  and is proved analyzing the dynamics of $\Xi(t)$ in the long time regime, in the same spirit as previous work by Martel and Merle \cite{MR1896235}, and more recent works \cite{MR3936126,MR4021089,MR3630087}. Theorem \ref{Thmdim2H1} is not different in nature, but follows the more standard use of Kato smoothing estimates, and the previously proved Theorem \ref{Thmdim2L2}. The more restrictive restriction $r>1$ appears from some interactions between mixed derivatives that require stronger control than other less complicated terms.

\medskip

The techniques used to prove Theorems \ref{Thmdim2L2} and \ref{Thmdim2H1} are sufficiently versatile to provide, as far as we understand, a first proof of decay in the  ZK 3D case. Consider now the region (see Fig. \ref{fig:2})
\begin{equation}\label{brconditions3D}
\begin{aligned}
	& \Omega(t):=  \left\{(x,y,z)\in\mathbb{R}^{3}\, ~ \Big|\, ~ |x| < t^{b}~  \wedge ~   |y|< t^{br_1} ~ \wedge ~  |z|< t^{br_2}\right\},\\
	  & b>0,\quad r_1,r_2>1, \quad   r_1+r_2<3, \quad r_1+1 <3r_2, \quad r_2+1 <3r_1,\quad  b < \frac{2}{3+r_1+r_2}. 
	   \end{aligned}
\end{equation}
See Lemma \ref{easy_P} for a detailed geometric description of this set. Two particularly important cases are $r_1\approx r_2\approx 1$, for which $b\approx \frac25$, $br_1\approx \frac25$ and $br_2\approx \frac25$; and $r_1\approx 1, r_2\approx 2$ (and symmetric case), for which $b\approx \frac13$, $br_1\approx \frac13$ and $br_2\approx \frac23$. For this region we immediately go to the $H^1$ case, proving the following result.
\begin{thm}[Local decay in the 3D case]\label{Thmdim3H1}
	Suppose  $u_0\in H^1(\mathbb R^3)$, and let $u=u(x,y,z,t)$ be the solution  to \eqref{ZK:Eq} in 3D such that $u\in C\left(\mathbb{R}: H^1(\mathbb{R}^{3})\right)$.
	Then
	\begin{equation}\label{cero3d}
	\liminf_{t\rightarrow \infty}\int_{\Omega(t)} \left( u^2 + |\nabla u|^{2} \right)(x,y,z,t)\mathrm{d}x\,\mathrm{d}y\,\mathrm{d}z=0.
	\end{equation}
	A similar decay rate as in \eqref{secuencia2D} also holds in this case.
\end{thm}

\begin{rem}
If only $L^2$ decay is  considered, the conditions $r_1,r_2>1$ in \eqref{brconditions3D} can be extended to $r_1,r_2>\frac12$.
\end{rem}

\begin{rem}
The volume of the region $\Omega(t)$ is not preserved with respect to variations of the parameters $b$ and $r_1$ and $r_2$. The supremum value of the area is obtained in the limit $r_1=r_2= \frac32$ and $b=\frac13$, which is $t^{\frac43}$. 
\end{rem}

The proof of Theorem \ref{Thmdim3H1} is similar to proof for the 2D case, with care needed to take into account the bigger number of error terms appearing in the dynamics, as well as some new technical estimates for cubic terms. However, the key of the argument is contained in the 2D case.

\medskip

\begin{figure}[h!]
\begin{center}
\begin{tikzpicture}[scale=1.8]
\pgfmathsetmacro{\cubex}{2}
\pgfmathsetmacro{\cubey}{1}
\pgfmathsetmacro{\cubez}{1}
\draw[black,dashed,fill=lightgray!30] (0,0,0) -- ++(-\cubex,0,0) -- ++(0,-\cubey,0) -- ++(\cubex,0,0) -- cycle;
\draw[black,dashed,fill=lightgray!30] (0,0,0) -- ++(0,0,-\cubez) -- ++(0,-\cubey,0) -- ++(0,0,\cubez) -- cycle;
\draw[black,dashed,fill=lightgray!30] (0,0,0) -- ++(-\cubex,0,0) -- ++(0,0,-\cubez) -- ++(\cubex,0,0) -- cycle;
\draw[->] (-1,-0.5,0) -- (-1,-0.5,1.5) node[below] {$x$};
\draw[->] (0,-0.5,-0.5) -- (1,-0.5,-0.5) node[below] {$y$};
\draw[->] (-1,0,-0.5) -- (-1,0.5,-0.5) node[above] {$z$};
\node at (-0.7,0.05,-0.5){$t^{br_2}$};
\node at (0.35,-0.4,0.3){$t^{br_1}$};
\node at (-0.7,-0.4,0.5){$t^b$};
\node at (-1.5,-0.4,0){$\Omega(t)$};
\end{tikzpicture}
\end{center}
\caption{Schematic figure depicting the set $\Omega(t)\subseteq \mathbb R^3$ for Theorem \ref{Thmdim3H1}, in the centered case, as defined in \eqref{brconditions3D}.}\label{fig:2}
\end{figure}
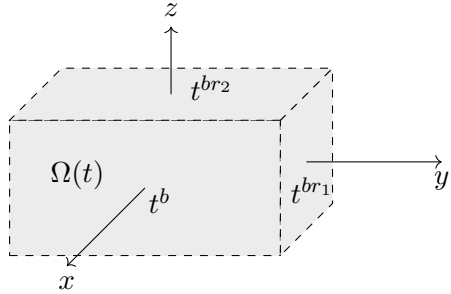

\medskip

Is it possible to obtain strong decay with data only in the energy space? This question is not easy at all, essentially because the dynamics in the large data case may be extremely complex. However, one can show strong $L^2$ decay for $H^1$ data in some particular regions of the space, characterized for being too far from the previously considered regions, and the soliton region. Recall that ${\bf x} =(x,y)=(x_1,x_2)$ in the 2D case, and ${\bf x}=(x,y,z) =(x_1,x_2,x_3)$ in the 3D one. For any $p\geq 1$, $\epsilon>0$ and $t\geq 2$, consider the region (see Fig. \ref{fig:3})
\begin{equation}\label{brconditionsFAR}
	\Omega_j(t):= \left\{{\bf x } \in\mathbb{R}^{d}\, ~ \Big|\, ~ | x_j |\sim  t^{p} \ln^{1+\epsilon} t\right\}, \quad j=1,\ldots, d.
\end{equation}
Recall that $p\geq 1$ is arbitrary. Here $a\sim b$ means that there exist $C_0,c_0>0$ independent of $a$ and $b$ such that $c_0 b \leq a\leq C_0b$. Our last result is the following strong decay property.
\begin{thm}\label{Thmdim23L2}
	Suppose  $u_0\in H^1(\mathbb R^d)$, $d=2,3$ and let $u=u({\bf x},t)$ be the solution to \eqref{ZK:Eq} such that $u\in C\left(\mathbb{R}: H^1(\mathbb{R}^{d})\right)$.
	Then
	\begin{equation}\label{Strong}
	\lim_{t\rightarrow \infty}\int_{\Omega_{j}(t)}  u^2 ({\bf x},t)\mathrm{d}{\bf x} =0, \quad j=1,\ldots, d.
	\end{equation}
\end{thm}
To prove this result we employ a new virial estimate first introduced in \cite{MPS}, which is extended here to the $d$-dimensional case (in particular, the transversal directions $y$ and $z$ work very well in terms of decay properties). The fact that we obtain the strong limit here is due to an improved virial estimate that profits of a key sign condition and not only forced decay estimates; however, it only works for long distances.

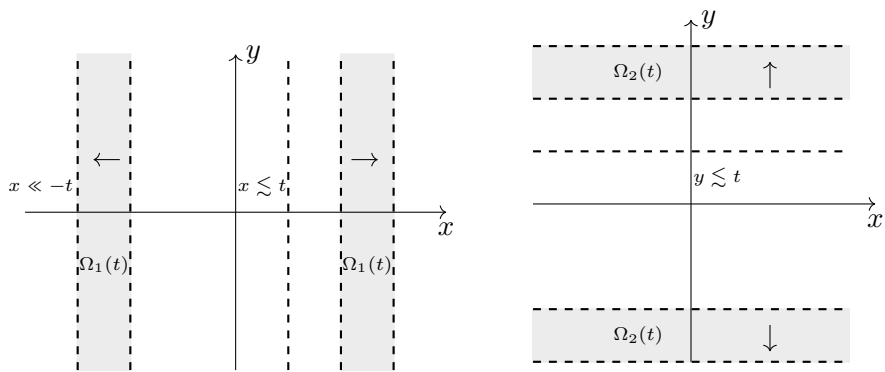
\begin{figure}[h!]
\begin{center}
\begin{tikzpicture}[scale=0.7]
\filldraw[thick, color=lightgray!30] (4,-1) -- (5,-1) -- (5,5) -- (4,5) -- (4,-1);
\filldraw[thick, color=lightgray!30] (0,-1) -- (-1,-1) -- (-1,5) -- (0,5) -- (0,-1);
\draw[thick,dashed] (3,-1) -- (3,5);
\draw[thick,dashed] (5,-1) -- (5,5);
\draw[thick,dashed] (4,-1) -- (4,5);
\draw[thick,dashed] (-1,-1) -- (-1,5);
\draw[thick,dashed] (0,-1) -- (0,5);
\draw[->] (-2,2) -- (6,2) node[below] {$x$};
\draw[->] (2,-1) -- (2,5) node[right] {$y$};
\draw[->] (-0.2,3) -- (-0.7,3) node[below] {};
\draw[->] (4.2,3) -- (4.7,3) node[below] {};
\node at (2.5,2.5){\tiny $x\lesssim t$};
\node at (-1.7,2.5){\tiny $x\ll -t$};
\node at (-0.5,1){\tiny $\Omega_1(t)$};
\node at (4.5,1){\tiny $\Omega_1(t)$};
\end{tikzpicture}
\qquad
\begin{tikzpicture}[scale=0.7]
\filldraw[thick, color=lightgray!30] (-1,4) -- (-1,5) -- (5,5) -- (5,4) -- (-1,4);
\filldraw[thick, color=lightgray!30] (-1,0) -- (-1,-1) -- (5,-1) -- (5,0) -- (-1,0);
\draw[thick,dashed] (-1,3) -- (5,3);
\draw[thick,dashed] (-1,5) -- (5,5);
\draw[thick,dashed] (-1,4) -- (5,4);
\draw[thick,dashed] (-1,-1) -- (5,-1);
\draw[thick,dashed] (-1,0) -- (5,0);
\draw[->] (-1,2) -- (5.5,2) node[below] {$x$};
\draw[->] (2,-1) -- (2,5.5) node[right] {$y$};
\node at (2.5,2.5){\tiny $y\lesssim t$};
\draw[->] (3.5,4.2) -- (3.5,4.7) node[below] {};
\draw[->] (3.5,-0.3) -- (3.5,-0.8) node[below] {};
\node at (1,-0.5){\tiny $\Omega_2(t)$};
\node at (1,4.5){\tiny $\Omega_2(t)$};
\end{tikzpicture}
\end{center}
\caption{(\emph{Left}). Schematic figure depicting the set $\Omega_1(t)$ in Theorem \ref{Thmdim23L2} in the 2D case. (\emph{Right}). The set $\Omega_2(t)$, case 2D.}\label{fig:3}
\end{figure}

\medskip

\subsection{A brief description of the ZK literature}\label{literatura}

In this subsection we briefly describe key previous results for \eqref{ZK:Eq} posed in the $\mathbb R^d$ setting. Originally derived by Zakharov and Kuznetsov \cite{ZK}, the mathematical study of the ZK equation has attracted the attention of many authors in past years. Unlike KdV, ZK is not integrable. It was rigorously derived from the Euler-Poisson system  with magnetic field as a long-wave and small-amplitude limit, see \cite[Section 10.3.2.6]{Lannes:Linares:Saut:2013}.

\medskip

Faminskii \cite{Faminskii:1995} showed local well-posedness (LWP) in $H^s(\mathbb{R}^2)$ for $s\ge 1$. After him, many researchers have contributed to the low regularity LWP theory. We mention the works of Linares and Pastor \cite{Linares:Pastor:2009}, Molinet and Pilod \cite{Molinet:Pilod:2015}, and Gr\"unrock and Herr \cite{Grunrock:Herr:2015}, who showed LWP at regularity $s>\frac12$. Very recently, Kinoshita \cite{Kinoshita:Arxiv:2019} has proved local-wellposedness for $s>-1/4$. This is best possible range. See also \cite{Linares:Pastor:2011} for the the proof of LWP in the case of a 2D modified ZK equation. Uniqueness results vs. spatial decay have been recently proved in \cite{MR3946612}, and propagation of regularity along regions of space has been considered in \cite{MR3842873}.

\medskip

Concerning the 3D case, Linares and Saut \cite{MR2486590}, Molinet and Pilod \cite{Molinet:Pilod:2015}, and Ribaud-Vento \cite{Ribaud:Vento:2012} proved local and global well-posedness (GWP) in $H^s(\mathbb{R}^3)$ for $s>1$. Herr and Kinoshita \cite{Herr:Kinoshita:Arxiv:2020} showed LWP for $s>-\frac{1}{2}$, and GWP in the energy space. Moreover, Herr and Kinoshita have proved that LWP holds in  $H^s(\mathbb{R}^d)$, with $d\ge 3$ and $s>\frac{d-4}{2}$. This last information, and the fact that \eqref{ZK:Eq} is $L^2$-critical in dimension 4 (possibly having blow-up solutions as well), has stopped us to get decay results in 4D.

\medskip

\noindent
{\bf Solitons.} Similar to the one dimensional KdV equation, \eqref{ZK:Eq} possesses soliton solutions of the form
\[
u({\bf x},t)= Q_c(x-ct,x'), \quad c>0, \quad x'\in \mathbb R^{d-1}.
\]
Here $Q_c= c Q(\sqrt{c} {\bf x})$ and $Q$ is the $H^1(\mathbb R^d)$ radial solution of the elliptic PDE
\[
\Delta Q -Q +Q^2=0, \quad Q>0, \qquad d\leq 5.
\]
Unlike KdV, no explicit formula is known for ZK solitons. However, for any $R>0$,
\[
\int_{|x-ct| \leq R, ~|x'|\leq R} Q_c^2 (x-ct,x') d{\bf x} \geq c^{2-\frac{d}2} c_0(R)>0,
\]
revealing that Theorem \ref{Thmdim2L2} cannot hold in the vicinity of solitons. However, Theorem \ref{Thmdim2L2} is still valid in $\Omega(t)$ even if the initial data contains infinitely many solitons adding finite $L^2$ norm. Anne de Bouard \cite{MR1378834} showed that $L^2$ subcritical ZK solitons are orbitally stable in $H^1$, and supercritical ones are unstable. The asymptotic stability of the solution has been proven in \cite{Cote:Munoz:Pilod:Simpson:2016} in the 2D case, and recently in \cite{Farah1} in 3D. Both works are nontrivial extensions of the foundational works by Martel and Merle \cite{MR1826966,MR1753061} concerning the one dimensional KdV case. Well-decoupled multi-solitons were proved stable in 2D, see \cite{Cote:Munoz:Pilod:Simpson:2016}. The modified ZK equation (cubic nonlinearity in \eqref{ZK:Eq}) is $L^2$ critical, and recently finite or infinite time blow up solutions were constructed around the solitary wave \cite{Farah2}, in close relation with a similar result obtained by Merle \cite{MR1824989} for the $L^2$-critical, quintic generalized KdV. Finally, see the recent work \cite{Valet} for the construction of multi-soliton like solutions for 2D and 3D ZK.



\subsection*{Organization of this paper} This paper is organized as follows. Section \ref{Sect:2} contains basic tools needed for the proofs in remaining Sections. Section \ref{Sect:3} and \ref{Sect:4} contain the proofs of Theorems \ref{Thmdim2L2} and \ref{Thmdim2H1}, respectively. Section \ref{Sect:5} is devoted to the proof of Theorem \ref{Thmdim3H1}, and Section \ref{Sect:6} deals with the proof of Theorem \ref{Thmdim23L2}. Finally, Theorem \ref{Thmdim2L2_KDV} is proved in Section \ref{A}.

\subsection*{Acknowledgments} We thank Didier Pilod, Gustavo Ponce and Jean-Claude Saut for comments on a first version of this draft.

\bigskip

\section{Preliminaries}\label{Sect:2}

\medskip

The purpose of this section is to gather all the necessary auxiliary results that will be needed in forthcoming sections. We start by describing the weighted functions used to define our local norms.

\subsection{Weighted functions}\label{Preli1}
Let    $\phi$ be a smooth {\bf even} and  positive function  such that
\medskip
\begin{itemize}
	\item[(i)] $\phi'\leq 0$ \quad on\quad  $\mathbb{R}^{+}=[0,\infty),$
	
	\item[(ii)]  $\phi|_{[0,1]}=1,$\quad  $\phi(x)=e^{-x}$\quad  on\quad  $[2,\infty),\quad e^{-x}\leq \phi(x)\leq 3e^{-x}$\quad  on\quad  $\mathbb{R}^{+}.$
	
	\item[(iii)] The derivatives of $\phi$ satisfy:
\begin{equation*}
|\phi'(x)|\leq c \phi(x)\qquad\mbox{and}\qquad |\phi''(x)|\leq c\phi(x),
\end{equation*}
for some positive constant $c$.
\end{itemize}
\medskip
Let $\psi(x)=\int_{0}^{x} \phi(s)\,\mathrm{d}s.$ Then $\psi$ is an odd function  such that  $\psi(x)=x$ on $[-1,1]$ and $|\psi(x)|\leq 3.$

\medskip

For $\sigma$ a parameter, we set
\begin{equation}\label{psi_phi}
\psi_{\sigma}(x)= \sigma\psi\left(\frac{x}{\sigma}\right)\quad \mbox{so that }\quad \psi_{\sigma}'(x)=\phi\left(\frac{x}{\sigma}\right)=:\phi_{\sigma}(x)
\end{equation}
and
\begin{equation}\label{sigmabound}
\begin{aligned}
\psi_{\sigma}(x)= & x\quad \mbox{on}\quad [-\sigma,\sigma],\\
 |\psi_{\sigma}(x)|\leq 3\sigma,\qquad & e^{-\frac{|x|}{\sigma}}\leq \phi_{\sigma}(x)\leq 3 e^{-\frac{|x|}{\sigma}}\quad \mbox{on}\quad \mathbb{R}.
\end{aligned}
\end{equation}
Also as part of our analysis we require to   define   functions $\lambda_{1},\lambda_{2}$ and $\eta$ that will be described later  in a more detailed manner, so that for the moment we will assume that such functions are smooth enough for $t\gg1$.\footnote{Through all the document we will use the  notation $t\gg1$ that for us  mean   $t\geq 10.$}





\subsection{Compactly supported weights}\label{2.3}

In this paragraph we consider weights needed for the proof of Theorem \ref{Thmdim23L2}, see Section \ref{Sect:6} for more details. We will consider the following  a function $\chi\in C^{\infty}(\mathbb{R})$  such that:
\begin{equation}\label{xi_especial}
0\leq \chi\leq 1,
\end{equation}
\begin{equation*}
\chi(x)=
\begin{cases}
1 & x\leq -1 \\
0 & x\geq  0,
\end{cases}
\end{equation*}
with $\supp(\chi)\subset (-\infty,0]$ and $\chi'(x)\leq 0$ for all $x\in\mathbb{R}.$ Also, for $x\in [-3/4,-1/4],$ the function $\chi'$ satisfies the inequality
\begin{equation}\label{bound_below}
-\chi'(x)\geq c_0 1_{[-\frac34,-\frac14]}(x)\quad\mbox{  for all}\quad  x\in\mathbb{R},
\end{equation}
where $c_0$ is a universal, positive constant.

\bigskip

\section{$L^2$ decay in 2D. Proof of Theorem \ref{Thmdim2L2}}\label{Sect:3}

Recall the region $\Omega(t)$ introduced in \eqref{brconditions}. The purpose of this section is to first show the following auxiliary result.
\begin{lem}\label{lem:d2L2}
Assume that $u_0\in L^2(\mathbb R^2)$. Let $u\in (C\cap L^\infty)(\mathbb R : L^2(\mathbb R^2))$ be the corresponding unique solution of \eqref{ZK:Eq} with initial data $u(t=0)=u_0$. Then, there exists a constant $C_0>0$ such that
\begin{equation}\label{IntegraL2}
\begin{split}
&\int_{\{t\gg1\}} \frac{1}{t\ln t}	
\left(\int_{\Omega(t)} u^{2}(x,y,t) \,\mathrm{d}x\mathrm{d}y\right)\mathrm{d}t \leq C_0 <\infty.
\end{split}
\end{equation}
\end{lem}
The proof of \eqref{IntegraL2} is the key element to conclude Theorem \ref{Thmdim2L2}, but its proof is technical; it requires the introduction of the modified virial functional \eqref{Virial}. 

\medskip

Assuming Lemma \ref{lem:d2L2}, we can easily prove Theorem \ref{Thmdim2L2}, following the lines in \cite{MR4021089}.

\subsection{End of proof of Theorem \ref{Thmdim2L2}} Since the function $\frac{1}{t\ln t}\notin L^{1}\left(\left\{t\gg 1\right\}\right)$ we can ensure that there exist a sequence of positive time $(t_{n})_{n}\uparrow \infty$ as $n$ goes to infinity, such that
\begin{equation*}
\lim_{n\uparrow \infty}\int_{\Omega_{\delta_{1},\delta_{2}}(t_{n})} u^{2}(x,y,t_{n}) \,\mathrm{d}x\mathrm{d}y=0.
\end{equation*}
This convergence of this sequence shows that 0 is an accumulation point, the least one because of nonnegativity. This proves the liminf and concludes  the proof of Theorem \ref{Thmdim2L2}.

\medskip
  Now we prove \eqref{secuencia2D}. It requires a simple result, Lemma \ref{lemdecay}, which is stated and proved below. One of the conclusions of Lemma \ref{lemdecay} is that for each $t\gg 1$
\begin{equation}\label{decay}
  \int_{t}^\infty \frac{1}{s\ln s}	
\left(\int_{\Omega(s)} u^{2}(x,y,s) \,\mathrm{d}x\mathrm{d}y\right)\mathrm{d}s\le C_0\int_{t}^{\infty}\frac{ds}{s\ln^{1/b}(s)}.
\end{equation}
For $t=t_0\gg 1$ in \eqref{decay}, there exists $t_1\ge t_0$ such that
$$\frac{1}{t_1\ln t_1}	
\left(\int_{\Omega(t_1)} u^{2}(x,y,t_1) \,\mathrm{d}x\mathrm{d}y\right)\le \frac{1}{t_1\ln^{1/b}(t_1)}.
$$
Otherwise we arrive to a contradiction with \eqref{decay}. Let $t_1^*=t_1+1$, by a similar argument as above we obtain $t_2\ge t_1^*>t_1$ satisfying
$$\frac{1}{t_2\ln t_2}	
\left(\int_{\Omega(t_2)} u^{2}(x,y,t_2) \,\mathrm{d}x\mathrm{d}y\right)\le \frac{1}{t_2\ln^{1/b}(t_2)}.
$$
Recursively, one can ensure the existence of a sequence of positive times $(t_{n})_{n}\uparrow \infty$ as $n$ goes to infinity, satisfying \eqref{secuencia2D}.

\medskip

The rest of this section will be devoted to the proof of Lemma \ref{lem:d2L2}.

\subsection{Setting} Recall the weighted functions $\psi_\sigma$ and $\phi_\sigma$ defined in \eqref{psi_phi}, the parameters $(b,r)$ in \eqref{brconditions}, and $\delta_1,\delta_2>0$.

\medskip

In the next for each  $b,r>0$  satisfying \eqref{brconditions} we denote by   $q\in(1,2)$  a number  such that
\begin{equation}\label{bqconditions_new}
  b\le\frac{2}{2+q+r}<\frac{2}{3+r},\quad \frac{1}{3}<r<3.
\end{equation}
For $u$  a solution of  the Zakharov-Kuznetsov equation \eqref{ZK:Eq}, we set the functional
\begin{equation}\label{XI2D}
\Xi(t):=\frac{1}{\eta(t)}\int_{\mathbb{R}^{2}}u(x,y,t)\psi_{\sigma}\left(\frac{x}{\lambda_{1}(t)}\right)\phi_{\delta_{1}}\left(\frac{x}{\lambda_{1}^q(t)}\right)\phi_{\delta_{2}}\left(\frac{y}{\lambda_{2}(t)}\right)\,\mathrm{d}x\mathrm{d}y,
\end{equation}
where $\lambda_{1},\lambda_{2}$ and $\eta$ are functions depending on $t$. We consider
\begin{equation}\label{defns}
\lambda_{1}(t)=\frac{t^{b}}{\ln t}\quad\mbox{and}\quad \eta(t)=t^{p}\ln ^{2} t,
\end{equation}
where $p$ is a  positive constant satisfying the constraint
\begin{equation}\label{eq4}
p+b=1.
\end{equation}
We also consider
\begin{equation}\label{eq5}
\lambda_{2}(t)=\lambda_{1}^{r}(t) \qquad\mbox{where}\quad r>0.
\end{equation}
Then,
\begin{equation}\label{Ap7}
\frac{\lambda_{1}'(t)}{\lambda_{1}(t)}\sim \frac{\eta'(t)}{\eta(t)}\sim \frac{1}{t}\qquad\mbox{for}\quad t\gg1.
\end{equation}
Also,
\begin{equation}\label{Ap8}
\lambda_{1}'(t)=\frac{1}{t^{1-b}\ln t}\left(\frac{b\ln t-1}{\ln t}\right)\quad \mbox{and}\quad \lambda_{1}(t)\eta(t)=t\ln t.
\end{equation}
%
%
%
%
%

\begin{lem}\label{bounded2d}
For $u\in L^{2}(\mathbb{R}^{2})$, the functional $\Xi$ is  well defined and bounded in time.
\end{lem}
\begin{proof}
By Cauchy-Schwarz inequality we obtain
\begin{equation}\label{eq1}
\begin{split}
|\Xi(t)|&\leq \frac{1}{\eta(t)}\|u(t)\|_{L^{2}_{x,y}}\left\|\psi_{\sigma}\right\|_{L^{\infty}_{x}}\left\|\phi_{\delta_{1}}\left(\frac{\cdot}{\lambda_{1}^q(t)}\right)\phi_{\delta_{2}}\left(\frac{\cdot}{\lambda_{2}(t)}\right)\right\|_{L^{2}_{x,y}}\\
&=\frac{\left(\lambda_{1}^q(t)\lambda_{2}(t)\right)^{1/2}}{\eta(t)}\|u_{0}\|_{L^{2}_{x,y}}\left\|\psi_{\sigma}\right\|_{L^{\infty}_{x}}\left\|\phi_{\delta_{2}}\right\|_{L^{2}_{y}}\left\|\phi_{\delta_{1}}\right\|_{L^{2}_{x}}\\
&\lesssim_{\delta_{1},\delta_{2},\sigma} \frac{1}{(\ln(t))^{(4+q+r)/2}}\frac{\|u_{0}\|_{L^{2}_{x,y}}}{t^{(2-b(2+q+r))/2}}.
\end{split}
\end{equation}
Since \eqref{bqconditions_new} is satisfied we have
\[
\sup_{t\gg 1} |\Xi(t)| \leq C_0 <\infty,
\]
which finishes the proof.
\end{proof}

\subsection{Dynamics for $\Xi(t)$}  In what follows, we compute and estimate the dynamics of $\Xi(t)$ in the long time regime.	
	

\begin{lem}\label{le:dT}
For any $t\geq 10$, one has the bound
\begin{equation}\label{dT_ppal}
\begin{split}
\frac{1}{\lambda_{1}(t)\eta(t)}\int_{\mathbb{R}^{2}} u^{2} \psi_{\sigma}'\left(\frac{x}{\lambda_{1}(t)}\right)\phi_{\delta_{1}}\left(\frac{x}{\lambda_{1}^q(t)}\right)\phi_{\delta_{2}}\left(\frac{y}{\lambda_{1}^r(t)}\right)\,\mathrm{d}x\mathrm{d}y &\le\frac{d\Xi}{dt}(t) + \Xi_{int}(t),
\end{split}
\end{equation}
where $\Xi_{int}(t)$ are terms that belong to $L^1([10,\infty))$.
\end{lem}

Assuming this estimate, \eqref{defns}, \eqref{eq4} and Lemma \ref{bounded2d} imply Lemma \ref{lem:d2L2}, after noticing that the set $\Omega(t)$ defined in \eqref{brconditions} is nothing but a set where
\[
\begin{aligned}
& \frac{1}{\lambda_{1}(t)\eta(t)}\int_{\Omega(t)} u^{2}\mathrm{d}x\mathrm{d}y \\
& \qquad \le \frac{1}{\lambda_{1}(t)\eta(t)}\int_{\mathbb{R}^{2}} u^{2} \psi_{\sigma}'\left(\frac{x}{\lambda_{1}(t)}\right)\phi_{\delta_{1}}\left(\frac{x}{\lambda_{1}^q(t)}\right)\phi_{\delta_{2}}\left(\frac{y}{\lambda_{1}^r(t)}\right)\,\mathrm{d}x\mathrm{d}y,
\end{aligned}
\]
provided $q>1$ is chosen sufficiently close to 1 in \eqref{bqconditions_new}, and the log terms in \eqref{defns} are discarded after making the parameter $b$ slightly smaller if necessary. The rest of the Section will be devoted to the proof of Lemma \ref{le:dT}.

\medskip

\noindent
{\bf Proof of Lemma \ref{le:dT}}. We have
\begin{equation}\label{separation1}
\begin{split}
\frac{\mathrm{d}}{\mathrm{d}t}\Xi(t)&=\frac{1}{\eta(t)}\int_{\mathbb{R}^{2}}\partial_{t}\left(u\psi_{\sigma}\left(\frac{x}{\lambda_{1}(t)}\right)\phi_{\delta_{1}}\left(\frac{x}{\lambda_{1}^q(t)}\right)\phi_{\delta_{2}}\left(\frac{y}{\lambda_{2}(t)}\right)\right)\,\mathrm{d}x\mathrm{d}y\\
&\quad -\frac{\eta'(t)}{\eta^{2}(t)}\int_{\mathbb{R}^{2}}u\psi_{\sigma}\left(\frac{x}{\lambda_{1}(t)}\right)\phi_{\delta_{1}}\left(\frac{x}{\lambda_{1}^q(t)}\right)\phi_{\delta_{2}}\left(\frac{y}{\lambda_{2}(t)}\right)\,\mathrm{d}x\mathrm{d}y\\
&=: \Xi_{1}(t)+ \Xi_{2}(t).
\end{split}
\end{equation}
First, we bound  $\Xi_{2},$ that in virtue of \eqref{eq1} the same analysis applied there  yields
\begin{equation}\label{22d}
\begin{split}
|\Xi_{2}(t)|&\leq\left|\frac{\eta'(t)}{\eta^{2}(t)}\int_{\mathbb{R}^{2}}u\psi_{\sigma}\left(\frac{x}{\lambda_{1}(t)}\right)\phi_{\delta_{1}}\left(\frac{x}{\lambda^q_{1}(t)}\right)\phi_{\delta_{2}}\left(\frac{y}{\lambda_{2}(t)}\right)\,\mathrm{d}x\mathrm{d}y\right|\\
&\lesssim_{\sigma,\delta_{1},\delta_{2}}\|u_{0}\|_{L^{2}}\frac{\left(\lambda^q_{1}(t)\lambda_{2}(t)\right)^{1/2}\eta'(t)}{\eta^{2}(t)}\\
&\lesssim \frac{1}{t}\frac{1}{\eta(t)}\frac{1}{(\lambda_1(t))^{-(q+r)/2}}=\frac{1}{t^{2-b(1+(q+r)/2)}\ln^{2+(q+r)/2}(t)}.
\end{split}
\end{equation}
From \eqref{bqconditions_new} we have $b\le\frac{2}{q+r+2}$ then $2-b\left(1+\frac{(q+r)}{2}\right)\ge 1$. Thus,  $\Xi_{2}\in L^1(\{t\gg1\})$.

\medskip

Unlike $\Xi_{2}$, to bound  $\Xi_{1}$ it is required   to take into consideration the dispersive part associated to the ZK equation as well as the non-linear interaction. More precisely, we shall decompose such term    as follows:
\begin{equation}\label{separation_Xi}
\begin{split}
\Xi_{1}(t)&=\frac{1}{\eta(t)}\int_{\mathbb{R}^{2}}\partial_{t}u\psi_{\sigma}\left(\frac{x}{\lambda_{1}(t)}\right)\phi_{\delta_{1}}\left(\frac{x}{\lambda_{1}^q(t)}\right)\phi_{\delta_{2}}\left(\frac{y}{\lambda_{2}(t)}\right)\,\mathrm{d}x\mathrm{d}y\\
&\quad -\frac{\lambda_{1}'(t)}{\lambda_{1}(t)\eta(t)}\int_{\mathbb{R}^{2}}u \psi_{\sigma}'\left(\frac{x}{\lambda_{1}(t)}\right)\left(\frac{x}{\lambda_{1}(t)}\right)\phi_{\delta_{1}}\left(\frac{x}{\lambda_{1}^q(t)}\right)\phi_{\delta_{2}}\left(\frac{y}{\lambda_{2}(t)}\right)\,\mathrm{d}x\mathrm{d}y\\
&\quad -q\frac{\lambda_{1}'(t)}{\lambda_{1}(t)\eta(t)}\int_{\mathbb{R}^{2}}u \psi_{\sigma}\left(\frac{x}{\lambda_{1}(t)}\right)\left(\frac{x}{\lambda_{1}^q(t)}\right)\phi_{\delta_{1}}'\left(\frac{x}{\lambda_{1}^q(t)}\right)\phi_{\delta_{2}}\left(\frac{y}{\lambda_{2}(t)}\right)\,\mathrm{d}x\mathrm{d}y\\
&\quad -\frac{\lambda_{2}'(t)}{\lambda_{2}(t)\eta(t)}\int_{\mathbb{R}^{2}}u \psi_{\sigma}\left(\frac{x}{\lambda_{1}(t)}\right)\phi_{\delta_{1}}\left(\frac{x}{\lambda_{1}^q(t)}\right)\left(\frac{y}{\lambda_{2}(t)}\right)\phi_{\delta_{2}}'\left(\frac{y}{\lambda_{2}(t)}\right)\,\mathrm{d}x\mathrm{d}y\\
&=: \Xi_{1,1}(t)+\Xi_{1,2}(t)+\Xi_{1,3}(t)+\Xi_{1,4}(t).
\end{split}
\end{equation}
Concerning to $\Xi_{1,1}$ we have by \eqref{ZK:Eq}  and integration by  parts
\begin{equation}\label{separation_Xi1}
\begin{split}
\Xi_{1,1}(t)&=-\frac{1}{\eta(t)}\int_{\mathbb{R}^{2}}\partial_{x}\left(\Delta u+\frac{u^{2}}{2}\right)\psi_{\sigma}\left(\frac{x}{\lambda_{1}(t)}\right)\phi_{\delta_{1}}\left(\frac{x}{\lambda^q_{1}(t)}\right)\phi_{\delta_{2}}\left(\frac{y}{\lambda_{2}(t)}\right)\,\mathrm{d}x\mathrm{d}y\\
&=\frac{1}{\eta(t)\lambda_{1}(t)}\int_{\mathbb{R}^{2}} \Delta u\psi_{\sigma}'\left(\frac{x}{\lambda_{1}(t)}\right)\phi_{\delta_{1}}\left(\frac{x}{\lambda^q_{1}(t)}\right)\phi_{\delta_{2}}\left(\frac{y}{\lambda_{2}(t)}\right)\,\mathrm{d}x\mathrm{d}y\\
&\quad +\frac{1}{\eta(t)\lambda^q_{1}(t)}\int_{\mathbb{R}^{2}} \Delta u\psi_{\sigma}\left(\frac{x}{\lambda_{1}(t)}\right)\phi_{\delta_{1}}'\left(\frac{x}{\lambda^q_{1}(t)}\right)\phi_{\delta_{2}}\left(\frac{y}{\lambda_{2}(t)}\right)\,\mathrm{d}x\mathrm{d}y\\
&\quad +\frac{1}{2\eta(t)\lambda_{1}(t)}\int_{\mathbb{R}^{2}}  u^{2}\psi_{\sigma}'\left(\frac{x}{\lambda_{1}(t)}\right)\phi_{\delta_{1}}\left(\frac{x}{\lambda_{1}^q(t)}\right)\phi_{\delta_{2}}\left(\frac{y}{\lambda_{2}(t)}\right)\,\mathrm{d}x\mathrm{d}y\\
&\quad +\frac{1}{2\eta(t)\lambda^q_{1}(t)}\int_{\mathbb{R}^{2}}  u^{2}\psi_{\sigma}\left(\frac{x}{\lambda_{1}(t)}\right)\phi_{\delta_{1}}'\left(\frac{x}{\lambda^q_{1}(t)}\right)\phi_{\delta_{2}}\left(\frac{y}{\lambda_{2}(t)}\right)\,\mathrm{d}x\mathrm{d}y\\
&=: \Xi_{1,1,1}(t)+\Xi_{1,1,2}(t)+\Xi_{1,1,3}(t)+\Xi_{1,1,4}(t).
\end{split}
\end{equation}
For  $\Xi_{1,1,1}$ we have  after combining  integration by parts
\begin{equation*}
\begin{split}
\Xi_{1,1,1}(t)&=\frac{1}{\eta(t)\lambda_{1}^{3}(t)}\int_{\mathbb{R}^{2}}u\psi'''_{\sigma}
\left(\frac{x}{\lambda_{1}(t)}\right)\phi_{\delta_{1}}\left(\frac{x}{\lambda^q_{1}(t)}\right)\phi_{\delta_{2}}\left(\frac{y}{\lambda_{2}(t)}\right)\,\mathrm{d}x\mathrm{d}y\\
&\quad +\frac{2}{\eta(t)\lambda_{1}^{2+q}(t)}\int_{\mathbb{R}^{2}}u\psi''_{\sigma}\left(\frac{x}{\lambda_{1}(t)}\right)\phi_{\delta_{1}}'\left(\frac{x}{\lambda^q_{1}(t)}\right)\phi_{\delta_{2}}\left(\frac{y}{\lambda_{2}(t)}\right)\mathrm{d}x\mathrm{d}y\\
&\quad +\frac{1}{\eta(t)\lambda_{1}^{1+2q}(t)}\int_{\mathbb{R}^{2}}u\psi'_{\sigma}\left(\frac{x}{\lambda_{1}(t)}\right)\phi_{\delta_{1}}''\left(\frac{x}{\lambda^q_{1}(t)}\right)\phi_{\delta_{2}}\left(\frac{y}{\lambda_{2}(t)}\right)\mathrm{d}x\mathrm{d}y\\
&\quad +\frac{1}{\eta(t)\lambda_{1}(t)\lambda_{2}^{2}(t)}\int_{\mathbb{R}^{2}}u\psi_{\sigma}^\prime\left(\frac{x}{\lambda_{1}(t)}\right)\phi_{\delta_{1}}\left(\frac{x}{\lambda^q_{1}(t)}\right)\phi_{\delta_{2}}''\left(\frac{y}{\lambda_{2}(t)}\right)\mathrm{d}x\mathrm{d}y.
\end{split}
\end{equation*}
First, we bound each term using Cauchy-Schwarz inequality, as follows:
\[
  \begin{aligned}
    |\Xi_{1,1,1}(t)|\le& \frac{(\lambda_1(t)\lambda_2(t))^{1/2}}{\eta(t)\lambda_1^3(t)}\|u_0\|_{L^2_{x,y}}\|\psi^{\prime\prime\prime}\|_{L^2_x}\|\phi_{\delta_1}\|_{L^\infty_{x}}\|\phi_{\delta_2}\|_{L^2_{y}}\\
    &+\frac{(\lambda_1(t)\lambda_2(t))^{1/2}}{\eta(t)\lambda_1^{2+q}(t)}\|u_0\|_{L^2_{x,y}}\|\psi^{\prime\prime}\|_{L^2_x}\|\phi^\prime_{\delta_1}\|_{L^\infty_{x}}\|\phi_{\delta_2}\|_{L^2_{y}}\\
    &+\frac{(\lambda_1(t)\lambda_2(t))^{1/2}}{\eta(t)\lambda_1^{1+2q}(t)}\|u_0\|_{L^2_{x,y}}\|\psi^{\prime}\|_{L^2_x}\|\phi^{\prime\prime}_{\delta_1}\|_{L^\infty_{x}}\|\phi_{\delta_2}\|_{L^2_{y}}\\
    &+\frac{(\lambda_1(t)\lambda_2(t))^{1/2}}{\eta(t)\lambda_1(t)\lambda_2^2(t)}\|u_0\|_{L^2_{x,y}}\|\psi^{\prime}\|_{L^2_x}\|\phi_{\delta_1}\|_{L^\infty_{x}}\|\phi^{\prime\prime}_{\delta_2}\|_{L^2_{y}}.
  \end{aligned}
\]
Consequently,
\begin{equation}\label{111}
  \begin{aligned}
    |\Xi_{1,1,1}(t)|    \lesssim&~ {} \frac{1}{\eta(t)(\lambda_1(t))^{5/2-r/2}}+\frac{1}{\eta(t)(\lambda_1(t))^{3/2+q-r/2}}\\
    &+\frac{1}{\eta(t)(\lambda_1(t))^{1/2+2q-r/2}}+\frac{1}{\eta(t)(\lambda_1(t))^{1/2+3r/2}}\\
    =&\frac{1}{t^{5/2-r/2}\ln^{r/2-1/2}(t)}+\frac{1}{t^{1+b(1/2+q-r/2)}\ln^{1/2-q+r/2}(t)}\\
    &+\frac{1}{t^{1+b(-1/2+2q-r/2)}\ln^{3/2-2q+r/2}(t)}+\frac{1}{t^{1+b(-1/2+3r/2)}\ln^{3/2-q-3r/2}(t)}.
  \end{aligned}
\end{equation}
We claim $\Xi_{1,1,1}\in L^1(\{t\gg1\})$. Indeed, from \eqref{bqconditions_new} we have $4q-1>1+2q>3>r>\frac 13$. Thus
\begin{align*}
  3&>r, &  1+2q&>r,\\
  4q-1&>r,&  r&>\frac13,
\end{align*}
or equivalently
\begin{align*}
  5/2-\frac r2&>1,& 1 / 2+q-\frac r2&>0,\\
  -1 / 2+2 q-\frac r2&>0, &   -1 / 2+ \frac 32r &>0.
\end{align*}
Next, applying integration by parts,
\[
\begin{aligned}
 \Xi_{1,1,2}(t) = &~{} \frac{1}{\eta(t)\lambda_{1}^q(t)}\int_{\mathbb{R}^{2}} \Delta u\psi_{\sigma}\left(\frac{x}{\lambda_{1}(t)}\right)\phi_{\delta_{1}}'\left(\frac{x}{\lambda_{1}^q(t)}\right)\phi_{\delta_{2}}\left(\frac{y}{\lambda_{2}(t)}\right)\,\mathrm{d}x\mathrm{d}y\\
 =&~{} \frac{1}{\eta(t)\lambda_{1}^{2+q}(t)}\int_{\mathbb{R}^{2}}u\psi_{\sigma}''\left(\frac{x}{\lambda_{1}(t)}\right)\phi_{\delta_{1}}'\left(\frac{x}{\lambda_{1}^q(t)}\right)\phi_{\delta_{2}}\left(\frac{y}{\lambda_{2}(t)}\right)\,\mathrm{d}x\mathrm{d}y\\
&~{}   +\frac{2}{\eta(t)\lambda_{1}^{1+2q}(t)}\int_{\mathbb{R}^{2}}u\psi_{\sigma}'\left(\frac{x}{\lambda_{1}(t)}\right)\phi_{\delta_{1}}''\left(\frac{x}{\lambda_{1}^q(t)}\right)\phi_{\delta_{2}}\left(\frac{y}{\lambda_{2}(t)}\right)\,\mathrm{d}x\mathrm{d}y\\
&~{} +\frac{1}{\eta(t)\lambda_{1}^{3q}(t)}\int_{\mathbb{R}^{2}}u\psi_{\sigma}\left(\frac{x}{\lambda_{1}(t)}\right)\phi_{\delta_{1}}'''\left(\frac{x}{\lambda_{1}^q(t)}\right)\phi_{\delta_{2}}\left(\frac{y}{\lambda_{2}(t)}\right)\,\mathrm{d}x\mathrm{d}y\\
&~{}  +\frac{1}{\eta(t)\lambda_{1}^q(t)\lambda_{2}^{2}(t)}\int_{\mathbb{R}^{2}}u \psi_{\sigma}\left(\frac{x}{\lambda_{1}(t)}\right)\phi_{\delta_{1}}'\left(\frac{x}{\lambda_{1}^q(t)}\right)\phi_{\delta_{2}}''\left(\frac{y}{\lambda_{2}(t)}\right)\mathrm{d}x\mathrm{d}y,
\end{aligned}
\]
and the Cauchy-Schwarz inequality yields
\[
\begin{aligned}
 &   |\Xi_{1,1,2}(t)| \\
 & \quad \leq \frac{(\lambda_1(t)\lambda_2(t))^{1/2}}{\eta(t)\lambda_1^{2+q}(t)}\|u_0\|_{L^2_{x,y}}\|\psi^{\prime\prime\prime}\|_{L^2_x}\|\phi_{\delta_1}\|_{L_x^\infty}\|\phi_{\delta_2}\|_{L^2_{y}} +\frac{(\lambda_1(t)\lambda_2(t))^{1/2}}{\eta(t)\lambda_1^{1+2q}(t)}\|u_0\|_{L^2_{x,y}}\|\psi^{\prime\prime}\|_{L^2_x}\|\phi^\prime_{\delta_1}\|_{L_x^\infty}\|\phi_{\delta_2}\|_{L^2_{y}}\\
    &\qquad +\frac{(\lambda_1^q(t)\lambda_2(t))^{1/2}}{\eta(t)\lambda_1^{3q}(t)}\|u_0\|_{L^2_{x,y}}\|\psi\|_{L_x^\infty}\|\phi^{\prime\prime\prime}_{\delta_1}\|_{L^2_{x}}\|\phi_{\delta_2}\|_{L^2_{y}} +\frac{(\lambda_1^q(t)\lambda_2(t))^{1/2}}{\eta(t)\lambda_1^q(t)\lambda_2^2(t)}\|u_0\|_{L^2_{x,y}}\|\psi\|_{L_x^\infty}\|\phi_{\delta_1}^\prime\|_{L^2_{x}}\|\phi^{\prime\prime}_{\delta_2}\|_{L^2_{y}}.
\end{aligned}
\]
Hence,
\begin{equation}\label{112}
  \begin{aligned}
    |\Xi_{1,1,2}(t)|     \lesssim& {}~ \frac{1}{\eta(t)(\lambda_1(t))^{3/2+q-r/2}}+\frac{1}{\eta(t)(\lambda_1(t))^{1/2+2q-r/2}}\\
    &+\frac{1}{\eta(t)(\lambda_1(t))^{5q/2-r/2}}+\frac{1}{\eta(t)(\lambda_1(t))^{q/2+3r/2}}\\
        =&\frac{1}{t^{1+b(1/2+q-r/2)}\ln^{1/2-q+r/2}(t)}+\frac{1}{t^{1+b(-1/2+2q-r/2)}\ln^{3/2-2q-r/2}(t)}\\
    &+\frac{1}{t^{1+b(-1+5/2q-r/2)}\ln^{2-5/2q+r/2}(t)}+\frac{1}{t^{1+b(-1+q/2+3r/2)}\ln^{2-q/2-3r/2}(t)}.
  \end{aligned}
\end{equation}
From \eqref{bqconditions_new} we have $5q-2>4q-1>1+2q>3>r>1/3>(2-q)/3$. Thus,
\begin{align*}
  1+2q&>r, &  4q-1&>r,\\
  5q-2&>r,&  r&>\frac13 (2-q).
\end{align*}
or equivalently
\begin{align*}
 \frac12+q-\frac r2&>0,& -1/2+2q-\frac r2&>0,\\
  -1+\frac52q- \frac r2 &>0, &  -1 + \frac q2+\frac32 r&>0.
\end{align*}
 Hence $\Xi_{1,1,2}\in L^1(\{t\gg1\})$.

\medskip

We emphasize that the term $\Xi_{1,1,3}$ in \eqref{separation_Xi1}
\begin{equation}\label{Xi113}
\Xi_{1,1,3}(t)= \frac{1}{2\eta(t)\lambda_{1}(t)}\int_{\mathbb{R}^{2}}  u^{2}\psi_{\sigma}'\left(\frac{x}{\lambda_{1}(t)}\right)\phi_{\delta_{1}}\left(\frac{x}{\lambda_{1}^q(t)}\right)\phi_{\delta_{2}}\left(\frac{y}{\lambda_{2}(t)}\right)\,\mathrm{d}x\mathrm{d}y,
\end{equation}
is the term to be estimated after integrating in time, leading to the left hand side in \eqref{dT_ppal}. Therefore, it will remain unchanged nearly until the end of the proof.

\medskip

The therm $\Xi_{1,1,4}$ in \eqref{separation_Xi1} satisfies de following estimate
\begin{equation}\label{11}
\begin{split}
|\Xi_{1,1,4}(t)|&\leq\left|\frac{1}{2\eta(t)\lambda_{1}^q(t)}\int_{\mathbb{R}^{2}}  u^{2}\psi_{\sigma}\left(\frac{x}{\lambda_{1}(t)}\right)\phi_{\delta_{1}}'\left(\frac{x}{\lambda_{1}^q(t)}\right)\phi_{\delta_{2}}\left(\frac{y}{\lambda_{2}(t)}\right)\,\mathrm{d}x\mathrm{d}y\right|\\
&\le\frac{1}{2\eta(t)\lambda_1^q(t)}\|u_0\|^2_{L^2_{x,y}}\|\psi\|_{L_x^\infty}\|\phi^{\prime}_{\delta_1}\|_{L_x^\infty}\|\phi_{\delta_2}\|_{L_y^\infty}.\\
&\lesssim\frac{1}{t^{1+b(q-1)}\ln^{2-q}(t)}
\end{split}
\end{equation}
Since \eqref{bqconditions_new} are satisfied, we obtain $q>1$ (note that $b>0$ is needed here). Thus $\Xi_{1,1,4}\in L^1(\{t\gg1\})$. This last estimate ends the study of the term $\Xi_{1,1}$ in \eqref{separation_Xi}.

\medskip

Now, we focus our attention in the remaining terms in \eqref{separation_Xi}. First, by means of Young's inequality,  we have for   $\epsilon>0,$
\begin{equation*}
\begin{split}
\left| \Xi_{1,2}(t)  \right| &= \left| \frac{\lambda_{1}'(t)}{\lambda_{1}(t)\eta(t)}\int_{\mathbb{R}^{2}}u \psi_{\sigma}'\left(\frac{x}{\lambda_{1}(t)}\right)\left(\frac{x}{\lambda_{1}(t)}\right)\phi_{\delta_{1}}\left(\frac{x}{\lambda_{1}^q(t)}\right)\phi_{\delta_{2}}\left(\frac{y}{\lambda_{2}(t)}\right)\,\mathrm{d}x\mathrm{d}y \right| \\
&\leq \frac{1}{4\epsilon}\left|\frac{\lambda_{1}'(t)}{\lambda_{1}(t)\eta(t)}\right|\int_{\mathbb{R}^{2}}u^{2}\psi_{\sigma}'\left(\frac{x}{\lambda_{1}(t)}\right)
\phi_{\delta_{1}}\left(\frac{x}{\lambda_{1}^q(t)}\right)\phi_{\delta_{2}}\left(\frac{y}{\lambda_{2}(t)}\right)\,\mathrm{d}x\mathrm{d}y\\
&\quad +\epsilon\left|\frac{\lambda_{1}'(t)}{\lambda_{1}(t)\eta(t)}\right|\int_{\mathbb{R}^{2}}
\psi_{\sigma}'\left(\frac{x}{\lambda_{1}(t)}\right)\left(\frac{x}{\lambda_{1}(t)}\right)^{2}\phi_{\delta_{1}}\left(\frac{x}{\lambda_{1}^q(t)}\right)\phi_{\delta_{2}}\left(\frac{y}{\lambda_{2}(t)}\right)\,\mathrm{d}x\mathrm{d}y\\
&=\frac{1}{4\epsilon}\left|\frac{\lambda_{1}'(t)}{\lambda_{1}(t)\eta(t)}\right|\int_{\mathbb{R}^{2}}u^{2}\psi_{\sigma}'\left(\frac{x}{\lambda_{1}(t)}\right)
\phi_{\delta_{1}}\left(\frac{x}{\lambda_{1}^q(t)}\right)\phi_{\delta_{2}}\left(\frac{y}{\lambda_{2}(t)}\right)\,\mathrm{d}x\mathrm{d}y\\
&\quad+\epsilon\left|\frac{\lambda_{1}'(t)\lambda_{2}(t)}{\eta(t)}\right|\|(\cdot)^2\psi^\prime_{\sigma}\|_{L^1_x}\left\|\phi_{\delta_{1}}(\cdot)\right\|_{L_x^\infty}\|\phi_{\delta_{2}}\|_{L^{1}_{y}},\\
\end{split}
\end{equation*}
so that, taking $\epsilon=\lambda'(t)>0$ for $t\gg1;$ it is clear that
\begin{equation*}
\begin{split}
\left|\Xi_{1,2}(t)\right|&\le \frac{1}{4\lambda_{1}(t)\eta(t)}\int_{\mathbb{R}^{2}}u^{2}\psi_{\sigma}'\left(\frac{x}{\lambda_{1}(t)}\right)
\phi_{\delta_{1}}\left(\frac{x}{\lambda_{1}^q(t)}\right)\phi_{\delta_{2}}\left(\frac{y}{\lambda_{2}(t)}\right)\,\mathrm{d}x\mathrm{d}y\\
&\quad +\frac{\left(\lambda_{1}'(t)\right)^{2}}{\lambda_1^2(t)}\frac{\lambda_{2}(t)}{\eta(t)(\lambda_1(t))^{-2}}\|(\cdot)^2\psi_{\sigma}\|_{L^1_x}\left\|\phi_{\delta_{1}}\right\|_{L_x^\infty}\|\phi_{\delta_{2}}\|_{L^{1}_{y}}\\
&=:\frac{1}{2}\Xi_{1,1,3}(t)+\Xi^*_{1,2}(t).
\end{split}
\end{equation*}
Note that the first term in the r.h.s is the quantity to be estimated (see \eqref{Xi113}), unlike the remaining term $\Xi^*_{1,2}$ which
 satisfies
\begin{equation}\label{12}
0\leq \Xi^*_{1,2}(t)\lesssim\frac{\left(\lambda_{1}'(t)\right)^{2}}{\lambda_1^2(t)}\frac{\lambda_2(t)}{\eta(t)(\lambda_1(t))^{-2-r}}\lesssim \frac{1}{t^2}\frac{1}{\eta(t)(\lambda_1(t))^{-2-r}}=\frac{1}{t^{3-b(3+r)}\ln^{4+r}(t)}.
\end{equation}
The  term $\Xi^*_{1,2}$ belongs in $L^1(\{t \gg 1\})$, since   \eqref{bqconditions_new} implies that $b<\frac{2}{r+3}$ or equivalent $3-b(3+r)>1$. This ends the estimate of $\Xi_{1,2}(t)$.

\medskip

Now we consider the term $\Xi_{1,3}(t)$. 
\begin{equation}\label{132d}
\begin{split}
\left| \Xi_{1,3}(t) \right|   &= q\left|\frac{\lambda_{1}'(t)}{\lambda_{1}(t)\eta(t)}\int_{\mathbb{R}^{2}}u \psi_{\sigma}\left(\frac{x}{\lambda_{1}(t)}\right)\left(\frac{x}{\lambda_{1}^q(t)}\right)\phi_{\delta_{1}}'\left(\frac{x}{\lambda^q_{1}(t)}\right)\phi_{\delta_{2}}\left(\frac{y}{\lambda_{2}(t)}\right)\,\mathrm{d}x\mathrm{d}y \right| \\
&\le q\left|\frac{\lambda_{1}'(t)\lambda_1^{q/2}(t)\lambda_2^{1/2}(t)}{\lambda_{1}(t)\eta(t)}\right|\|u\|_{L^2_{x,y}}\|\phi^{\prime}_{\delta_1}\|_{L^2_{x}}\|\phi_{\delta_2}\|_{L^2_{y}}\\
&\lesssim \frac{1}{t^{2-b(+q/2+r/2)}\ln^{2+q/2+r/2}(t)}.
\end{split}
\end{equation}

 By \eqref{bqconditions_new} we have $b\leq\frac{2}{2+q+r}$, consequently  $2-b(1+q/2+r/2) \ge 1$. Thus, $\Xi_{1,3}\in L^1(\{t\gg1\})$.

 As before,
\begin{equation}\label{eq32d}
\begin{aligned}
  \left|\Xi_{1,4}(t)\right|&\le\left|\frac{\lambda_{2}^{\prime}(t)}{\lambda_{2}(t) \eta(t)}\right| \int_{\mathbb{R}^{2}} |u| \psi_{\sigma}\left(\frac{x}{\lambda_{1}(t)}\right) \phi_{\delta_{1}}\left(\frac{x}{\lambda_{1}^q(t)}\right)\left|\frac{y}{\lambda_2(t)}\right| \left|\phi_{\delta_{2}}^{\prime}\left(\frac{y}{\lambda_{2}(t)}\right)\right| \mathrm{d} x \mathrm{d} y\\
 &\le \left|\frac{\lambda_{1}'(t)\lambda_1^{q/2}(t)\lambda_2^{1/2}(t)}{\lambda_{2}(t)\eta(t)}\right|\|u\|_{L^2_{x,y}}\|\phi_{\delta_1}\|_{L^2_{x}}\|(\cdot)\phi^{\prime}_{\delta_2}\|_{L^2_{y}}\\
&\lesssim \frac{1}{t^{2-b(+q/2+r/2)}\ln^{2+q/2+r/2}(t)}.
\end{aligned}
\end{equation}
Hence we obtain that $\Xi_{1,4}\in L^1(\{t\gg 1\})$.

\medskip

Now combining \eqref{separation1}, \eqref{separation_Xi}, \eqref{separation_Xi1} and \eqref{12} we have
\begin{equation*}
\begin{split}
  0\le\Xi_{1,1,3}(t)&=\frac{d\Xi}{dt}(t)-\Xi_{1,1,1}(t)-\Xi_{1,1,2}(t)-\Xi_{1,1,4}(t)\\
  &\quad-\Xi_{1,2}(t)-\Xi_{1,3}(t)-\Xi_{1,4}(t)-\Xi_{2}(t)\\
  &\le\frac{d\Xi}{dt}(t)-\Xi_{1,1,1}(t)-\Xi_{1,1,2}(t)-\Xi_{1,1,4}(t)\\
  &\quad+\frac{1}{2}\Xi_{1,1,3}(t)+\Xi^*_{1,2}(t)-\Xi_{1,3}(t)-\Xi_{1,4}(t)-\Xi_{2}(t).
\end{split}
\end{equation*}
Thus,
\begin{equation}\label{compar}
\begin{split}
  \Xi_{1,1,3}(t)&\le\frac{d\Xi}{dt}(t)-\Xi_{2}(t)-\Xi_{1,1,1}(t)-\Xi_{1,1,2}(t)-\Xi_{1,1,4}(t)\\
  &\quad+\Xi^*_{1,2}(t) -\Xi_{1,3}(t)-\Xi_{1,4}(t).
\end{split}
\end{equation}
Note that all the terms on the right above, including $\frac{d\Xi}{dt}(t)$, lie in $L^1(\{t\gg 1\})$, by  \eqref{22d}, \eqref{111}, \eqref{112}, \eqref{11}, \eqref{12}, \eqref{132d} and \eqref{eq32d}.  In consequence, we conclude \eqref{dT_ppal}.

\subsection{The non-centered case}\label{Non} Here we explain how Theorem \ref{Thmdim2L2}, can be extended to the non-centered case, as explained in Remark \ref{NON}. The proof is simple and require some minor modifications of the proof in the centered case. First of all, we consider this time the functional (compare with \eqref{XI2D})
\begin{equation}\label{XI2D_new}
\Xi(t):=\frac{1}{\eta(t)}\int_{\mathbb{R}^{2}}u(x,y,t)\psi_{\sigma}\left(\frac{x+\rho_1(t)}{\lambda_{1}(t)}\right)\phi_{\delta_{1}}\left(\frac{x+\rho_1(t)}{\lambda_{1}^q(t)}\right)\phi_{\delta_{2}}\left(\frac{y+\rho_2(t)}{\lambda_{2}(t)}\right)\,\mathrm{d}x\mathrm{d}y,
\end{equation}
 where $\rho_1(t):= \pm t^m, \rho_2(t):=\pm t^n$ are the new time-dependent parameters. It is enough to consider the positive sign in $\rho_i(t)$. Only three essential new estimates appear when computing the time derivative of $\Xi(t)$. These are
 \[
\Xi_{a}(t):= \frac{\rho_1'(t)}{\eta(t)\lambda_1(t)}\int_{\mathbb{R}^{2}}u(x,y,t)\psi_{\sigma}'\left(\frac{x+\rho_1(t)}{\lambda_{1}(t)}\right)\phi_{\delta_{1}}\left(\frac{x+\rho_1(t)}{\lambda_{1}^q(t)}\right)\phi_{\delta_{2}}\left(\frac{y+\rho_2(t)}{\lambda_{2}(t)}\right)\,\mathrm{d}x\mathrm{d}y,
 \]
\[
\Xi_{b}(t):= \frac{\rho_1'(t)}{\eta(t)\lambda_1^q (t)}\int_{\mathbb{R}^{2}}u(x,y,t)\psi_{\sigma}\left(\frac{x+\rho_1(t)}{\lambda_{1}(t)}\right)\phi_{\delta_{1}}'\left(\frac{x+\rho_1(t)}{\lambda_{1}^q(t)}\right)\phi_{\delta_{2}}\left(\frac{y+\rho_2(t)}{\lambda_{2}(t)}\right)\,\mathrm{d}x\mathrm{d}y,
 \]
and
\[
\Xi_{c}(t):= \frac{\rho_2'(t)}{\eta(t)\lambda_2(t)}\int_{\mathbb{R}^{2}}u(x,y,t)\psi_{\sigma}\left(\frac{x+\rho_1(t)}{\lambda_{1}(t)}\right)\phi_{\delta_{1}}\left(\frac{x+\rho_1(t)}{\lambda_{1}^q(t)}\right)\phi_{\delta_{2}}'\left(\frac{y+\rho_2(t)}{\lambda_{2}(t)}\right)\,\mathrm{d}x\mathrm{d}y.
 \]
These three quantities are estimated as follows:
 \[
| \Xi_{a}(t)| \lesssim  \frac{|\rho_1'(t)| \lambda_2^{1/2}(t)}{\eta(t)\lambda_1^{1/2}(t)} \lesssim \frac{1}{t^{2-\frac b2(r+1) -m} \ln^{\frac32 +\frac{r}2} t} \in L^1_t([10,\infty))
 \]
 provided $0\leq m\leq 1-\frac12 b(1+r)$. Also, since $q>1$ and $b>0$,
 \[
| \Xi_{b}(t)| \lesssim  \frac{|\rho_1'(t)| \lambda_2^{1/2}(t)}{\eta(t)\lambda_1^{q/2}(t)} \ll  \frac{|\rho_1'(t)| \lambda_2^{1/2}(t)}{\eta(t)\lambda_1^{1/2}(t)} \in L^1_t([10,\infty)); 
 \]
 and finally,
\[
| \Xi_{c}(t)| \lesssim  \frac{|\rho_2'(t)| \lambda_1^{q/2}(t)}{\eta(t)\lambda_2^{1/2}(t)} \lesssim \frac{1}{t^{2 -b +\frac b2(r-q) -n} \ln^{2 -\frac{(r-q)}2 } t} \in L^1_t([10,\infty)), 
 \]
provided $0\leq n\leq 1-\frac12 b(2+q-r)$. Adding the conditions on $m$ and $n$ in \eqref{brconditions_non}, taking $q\approx 1$, one concludes.

\medskip

\subsection{On the liminf condition and some consequences}  The purpose of this subsection is to further explain the consequences of the zero liminf result sated in \eqref{cero2D}. It turns out that some interesting secuencial  theoretical implications can be obtained from such a simple result.

\begin{lem}\label{lemdecay}Let $b$ described in \eqref{bqconditions_new}. Then  for some $C_0>0$ we have
\begin{equation}\label{decay1}
  \int_{t}^{\infty} \frac{1}{s\ln s}	
\left(\int_{\Omega(s)} u^{2}(x,y,s) \,\mathrm{d}x\mathrm{d}y\right)\mathrm{d}s\le C_0\int_{t}^{\infty}\frac{1}{s\ln^{1/b}(s)}ds, \quad t\gg 1,
\end{equation}
and
\begin{equation}\label{decay2}
  \int_{a}^{b} \frac{1}{s\ln s}	
\left(\int_{\Omega(s)} u^{2}(x,y,s) \,\mathrm{d}x\mathrm{d}y\right)\mathrm{d}s\le \frac{C_0}{\ln^{1/b}(a)}+C_0\int_{a}^{b}\frac{1}{s\ln^{1/b}(s)}ds\quad a,b \gg1.
\end{equation}
\end{lem}

\begin{proof}
We adopt the same notation given in the proof of Lemma \ref{bounded2d}. After comparing the terms in \eqref{compar} we set $q$ in \eqref{bqconditions_new} by $q=\frac 2b-r-2$, for $0<b<\frac{2}{3+r}$.  From  \eqref{22d}, \eqref{111}, \eqref{112}, \eqref{11}, \eqref{12}, \eqref{132d} and \eqref{eq32d} we have $|\Xi_{1,1,1}|,$ $ |\Xi^{\ast}_{1,2}|,$ $ |\Xi_{1,1,4}|,$ $|\Xi_{1,1,2}|\lesssim |\Xi_{1,1,4}|$ and $|\Xi_{2}|,$ $|\Xi_{1,4}|$, $|\Xi_{1,3}|\lesssim \frac{1}{(\cdot)\ln^{2+(q+r)/2}(\cdot)}=\frac{1}{(\cdot)\ln^{1/b}(\cdot)}$. Consequently by \eqref{compar} we have

\begin{equation*}
  \Xi_{1,1,3}(t)\le \frac{d\Xi}{dt}(t)+\frac{C_1}{t\ln^{\frac{1}{b}}(t)},\quad t\gg 1,
\end{equation*}
for some $C_1>0$.
Integrating on time and taking into consideration \eqref{eq1} we get for $t\gg 1$
\begin{equation*}
  \begin{split}
    \int_{t}^{\infty}\Xi_{1,1,3}(s)ds&\le \int_{t}^{\infty}\frac{d\Xi}{ds}(s)ds+C_1\int_{t}^{\infty}\frac{1}{s\ln^{1/b}(s)}ds\\
    &\le -\Xi(t)+C_1\int_{t}^{\infty}\frac{1}{s\ln^{1/b}(s)}ds\\
    &\le \frac{C_2}{\ln^{1/b}(t)}+C_1\int_{t}^{\infty}\frac{1}{s\ln^{1/b}(s)}ds\\
    &\le \frac{C_2}{b}\int_{t}^{\infty}\frac{1}{s\ln^{1/b+1}(s)}ds+C_1\int_{t}^{\infty}\frac{1}{s\ln^{1/b}(s)}ds\le C_0\int_{t}^{\infty}\frac{1}{s\ln^{1/b}(s)}ds,
  \end{split}
\end{equation*}
where $C_0=\frac{C_2}{b}+C_1$ and $C_2=\|u_{0}\|_{L^{2}_{x,y}}\left\|\psi_{\sigma}\right\|_{L^{\infty}_{x}}\left\|\phi_{\delta_{2}}\right\|_{L^{2}_{y}}\left\|\phi_{\delta_{1}}\right\|_{L^{2}_{x}}$. Thus \eqref{decay1} is satisfied. Finally by a similar argument we conclude \eqref{decay2}.
\end{proof}

\begin{lem}\label{lemmedida} Let $\varepsilon>0$  and $E_{\varepsilon}:=\left\{s\in \mathbb{R}: \int_{\Omega(s)} u^{2}(x,y,s) \,\mathrm{d}x\mathrm{d}y>\varepsilon\right\}$ and $b$ described in \eqref{bqconditions_new}. Then  the following conditions are satisfied:
\begin{enumerate}
  \item $E_\varepsilon=\bigsqcup_{n\in \mathbb{N}} (a_n,b_n)$ (disjoint union).
  \item   If $a_n\gg 1$ then we have
    \begin{equation}\label{3p32}
      b_n<  a_n^{\exp\left({\frac{2C_0}{\varepsilon\ln^{1/b-1}(a_n)}}\right)} .
    \end{equation}
\end{enumerate}
\end{lem}

\begin{proof}
Item (1) is a consequence of that the set $E_{\varepsilon}$ is open on $\mathbb{R}$. First we note from \eqref{bqconditions_new} that $b<2/5$. Let $F(s): =\int_{\Omega(s)} u^{2}(x,y,s) \,\mathrm{d}x\mathrm{d}y$ and  $b_n>a_n\gg 1$. Since $F(s)>\varepsilon$ for $s\in (a_n,b_n)$ and  \eqref{decay2} is satisfied we obtain
\begin{equation*}
  \int_{a_n}^{b_n}\frac{\varepsilon}{s\ln(s)}ds < \int_{a_n}^{b_n}\frac{F(s)}{s\ln(s)}ds\le \frac{C_0}{\ln^{1/b}(a_n)}+C_0\int_{a_n}^{b_n}\frac{1}{s\ln^{1/b}(s)}ds,
\end{equation*}
then we have
\begin{equation*}
         \varepsilon(\ln(\ln(b_n))-\ln(\ln(a_n))< \frac{C_0}{\ln^{1/b}(a_n)}+\frac{C_0}{1/b-1}\left(\frac{1}{\ln^{1/b-1}(a_n)}-\frac{1}{\ln^{1/b-1}(b_n)}\right).
       \end{equation*}
The last inequality implies
\begin{equation*}\begin{split}
         \varepsilon(\ln(\ln(b_n))&< \varepsilon\ln(\ln(a_n))+ \frac{C_0}{\ln^{1/b}(a_n)}+\frac{C_0}{1/b-1}\frac{1}{\ln^{1/b-1}(a_n)}\\
         &< \varepsilon\ln(\ln(a_n))+ \frac{C_0}{\ln^{1/b-1}(a_n)}+\frac{C_0}{\ln^{1/b-1}(a_n)},
         \end{split}
       \end{equation*}
consequently \eqref{3p32} is satisfied.
\end{proof}

\begin{rem}[An explicit construction of the times $t_n$]
Item (2) of the Lemma \ref{lemmedida} is useful  to specify the times   of a sequence  $(t_{n})_{n}\uparrow \infty$ as $n$ goes to infinity  such that
\begin{equation*}
\lim_{n\uparrow \infty}\int_{\Omega_{\delta_{1},\delta_{2}}(t_{n})} u^{2}(x,y,t_{n}) \,\mathrm{d}x\mathrm{d}y=0.
\end{equation*}
In fact let $F(s): =\int_{\Omega_{\delta_{1},\delta_{2}}(s)} u^{2}(x,y,s) \,\mathrm{d}x\mathrm{d}y$, $t_0\gg 1$ and $\varepsilon_1=F(t_0)/2$. From \eqref{3p32} we infer  $F(t_1)\le \varepsilon_1$, if $t_1=t_0^{\exp\left({\frac{2C_0}{\varepsilon\ln^{1/b-1}(t_0)}}\right)}$. Similar as above  for $\varepsilon_2=F(t_1)/2$ we have $F(t_1)\le \varepsilon_2$ if $t_2=t_1^{\exp\left({\frac{2C_0}{\varepsilon\ln^{1/b-1}(t_1)}}\right)}$. Recursively the sequence of time $(t_n)$ described by \begin{equation*}t_{n+1}=t_n^{\exp\left({\frac{2C_0}{\varepsilon\ln^{1/b-1}(t_n)}}\right)},\end{equation*} is such that
\begin{equation*}
  \int_{\Omega_{\delta_{1},\delta_{2}}(t_{n})} u^{2}(x,y,t_{n}) \,\mathrm{d}x\mathrm{d}y\le F(t_0)/2^n\to 0 ,
\end{equation*}
as $n$ goes to infinity.
\end{rem}
\section{ Asymptotic behavior of solutions in $\dot{H}^{1}(\mathbb{R}^{2})$}\label{Sect:4}

The purpose of this section is to prove Theorem \ref{Thmdim2H1}. We follow similar ideas as in previous section, with two additional ingredients. First, we use Theorem \ref{Thmdim2L2}, more precisely, Lemma \ref{lem:d2L2}. Second, we use some technical estimates for nonlinear terms first introduced by Kenig and Martel \cite{MR2590690} in the case of the Benjamin-Ono equation.

\medskip

Recall the region $\Omega(t)$ introduced in \eqref{brconditions}, with the additional assumption $1<r<3$. Theorem \ref{Thmdim2H1} follows from the following result.

\begin{lem}\label{lem:d2H1}
Assume now that $u_0\in H^1(\mathbb R^2)$. Let $u\in (C\cap L^\infty)(\mathbb R : H^1(\mathbb R^2))$ be the corresponding unique solution of \eqref{ZK:Eq} with initial data $u(t=0)=u_0$. Then, there exists a constant $C_1>0$ such that
\begin{equation}\label{IntegraH1}
\int_{\{t\gg1\}} \frac{1}{t\ln t}	
\left(\int_{\Omega(t)} |\nabla u|^{2}(x,y,t) \,\mathrm{d}x\mathrm{d}y\right)\mathrm{d}t \leq C_1 <\infty.
\end{equation}
\end{lem}
Recall $\psi$ and $\phi$ defined in Subsection \ref{Preli1}. In what follows, $\sigma'$ is a positive constant to be determined later. To prove \eqref{IntegraH1}, we consider the functional
\begin{equation}\label{defQd2}
\mathcal{Q}(t):=\frac{1}{\eta(t)}\int_{\mathbb{R}^{2}} u^{2}(x,y,t)\psi_{\sigma'}\left(\frac{x}{\lambda_{1}(t)}\right)\phi_{\delta_{2}}\left(\frac{y}{\lambda_{2}(t)}\right)\mathrm{d}x\,\mathrm{d}y,
\end{equation}
that is clearly well defined for   solutions of the IVP  \eqref{ZK:Eq}.

\medskip

Next, we compute:
\begin{equation}\label{main2}
\begin{split}
\frac{\mathrm{d}}{\mathrm{d}t}\mathcal{Q}(t)=& \underbrace{\frac{2}{\eta(t)}\int_{\mathbb{R}^{2}}u\partial_{t}u\psi_{\sigma'}\left(\frac{x}{\lambda_{1}(t)}\right)\phi_{\delta_{2}}\left(\frac{y}{\lambda_{2}(t)}\right)\,\mathrm{d}x\,\mathrm{d}y}_{A_{1}(t)}\\
&\underbrace{-\frac{\lambda_{2}'(t)}{\lambda_{2}(t)\eta(t)}\int_{\mathbb{R}^{2}}u^{2}\psi_{\sigma'}\left(\frac{x}{\lambda_{1}(t)}\right)\phi_{\delta_{2}}'\left(\frac{y}{\lambda_{2}(t)}\right)\left(\frac{y}{\lambda_{2}(t)}\right)\,\mathrm{d}x\,\mathrm{d}y}_{A_{2}(t)}\\
&\underbrace{-\frac{\lambda_{1}'(t)}{
		\lambda_{1}(t)\eta(t)}\int_{\mathbb{R}^{2}}u^{2}\phi_{\sigma'}\left(\frac{x}{\lambda_{1}(t)}\right)\left(\frac{x}{\lambda_{1}(t)}\right)\phi_{\delta_{2}}\left(\frac{y}{\lambda_{2}(t)}\right)\,\mathrm{d}x\,\mathrm{d}y}_{A_{3}(t)}\\
&\underbrace{-\frac{\eta'(t)}{\eta^{2}(t)}\int_{\mathbb{R}^{2}}u^{2}\psi_{\sigma'}\left(\frac{x}{\lambda_{1}(t)}\right)\phi_{\delta_{2}}\left(\frac{y}{\lambda_{2}(t)}\right)\,\mathrm{d}x\,\mathrm{d}y}_{A_{4}(t)}.
\end{split}
\end{equation}
Now we bound each of the terms above. In the first place we have, after applying integration by parts,  that
\begin{equation}\label{divisionA11A12}
\begin{split}
A_{1}(t)&=\frac{2}{\eta(t)}\int_{\mathbb{R}^{2}}\partial_{x}u\left(\Delta u+\frac{u^{2}}{2}\right)\psi_{\sigma'}\left(\frac{x}{\lambda_{1}(t)}\right)\phi_{\delta_{2}}\left(\frac{y}{\lambda_{2}(t)}\right)\,\mathrm{d}x\,\mathrm{d}y\\
&\quad +\frac{2}{\lambda_{1}(t)\eta(t)}\int_{\mathbb{R}^{2}}u\left(\Delta u+\frac{u^{2}}{2}\right)\phi_{\sigma'}\left(\frac{x}{\lambda_{1}(t)}\right)\phi_{\delta_{2}}\left(\frac{y}{\lambda_{2}(t)}\right)\,\mathrm{d}x\,\mathrm{d}y\\
&=A_{1,1}(t)+A_{1,2}(t).
\end{split}
\end{equation}
In this sense, we have that
\begin{equation}\label{divisionA11}
\begin{split}
A_{1,1}(t)&=-\frac{1}{\lambda_{1}(t)\eta(t)}\int_{\mathbb{R}^{2}}\left(\partial_{x}u\right)^{2}\phi_{\sigma'}\left(\frac{x}{\lambda_{1}(t)}\right)\phi_{\delta_{2}}\left(\frac{y}{\lambda_{2}(t)}\right)\,\mathrm{d}x\,\mathrm{d}y\\
&\quad +\frac{1}{\lambda_{1}(t)\eta(t)}\int_{\mathbb{R}^{2}}\left(\partial_{y}u\right)^{2}\phi_{\sigma'}\left(\frac{x}{\lambda_{1}(t)}\right)\phi_{\delta_{2}}\left(\frac{y}{\lambda_{2}(t)}\right)\,\mathrm{d}x\,\mathrm{d}y\\
&\quad -\frac{2}{\lambda_{2}(t)\eta(t)}\int_{\mathbb{R}^{2}}\partial_{x}u\partial_{y}u\psi_{\sigma'}\left(\frac{x}{\lambda_{1}(t)}\right)\phi_{\delta_{2}}'\left(\frac{y}{\lambda_{2}(t)}\right)\,\mathrm{d}x\,\mathrm{d}y\\
&\quad -\frac{1}{3\eta(t)\lambda_{1}(t)}\int_{\mathbb{R}^{2}}u^{3}\phi_{\sigma'}\left(\frac{x}{\lambda_{1}(t)}\right)\phi_{\delta_{2}}\left(\frac{y}{\lambda_{2}(t)}\right)\,\mathrm{d}x\,\mathrm{d}y\\
&=A_{1,1,1}(t)+A_{1,1,2}(t)+A_{1,1,3}(t)+A_{1,1,4}(t).
\end{split}
\end{equation}
The terms $A_{1,1,1}$ and $A_{1,1,2}$ are essentially some portion of the quantities appearing in \eqref{IntegraH1}. The others come from $A_{1,2}(t)$. Instead, the terms $A_{1,1,3}$ and $A_{1,1,4}$ need to be estimated.

\medskip

First, we handle $A_{1,1,3}$.  We obtain
\begin{equation*}
\int_{\{t\gg1\}}|A_{1,1,3}(t)|\,\mathrm{d}t\lesssim_{\delta_{2}}\|u\|_{L^{\infty}_{t}H^{1}}^{2}\int_{\{t\gg1\}}\frac{\mathrm{d}t}{\eta(t)\lambda_{2}(t)}<\infty,
\end{equation*}
whenever $r>1$ (see \eqref{defns} and \eqref{eq5}). This is the extra condition needed for the proof of Theorem \ref{Thmdim2H1}.

\medskip

Next, we focus our attention into $A_{1,1,4}(t)$.  Here  we consider a smooth cut-off function $\chi:\mathbb{R}\longrightarrow\mathbb{R}$  such that
\begin{equation}\label{chi_chi}
\hbox{$\chi\equiv 1$ on $[0,1]$,\quad $0\leq \chi \leq 1$ ~ and ~ $\chi\equiv0 $ on $(-\infty,-1]\cup[2,\infty).$}
\end{equation}
For $n\in\mathbb{Z},$ we set $\chi_{n}(x):=\chi(x-n)$, such that $\chi_n \equiv 1$ in $[n,n+1]$.
Similarly, we define for $m\in\mathbb{Z}$ the function  $\chi_{m}(y):=\chi(y-m).$

\medskip

First, notice that  by the Gagliardo-Nirenberg-Sobolev inequality
\[
\begin{split}
&\int_{\mathbb{R}^{2}}|u|^{3}\phi_{\sigma'}\left(\frac{x}{\lambda_{1}(t)}\right)\phi_{\delta_{2}}\left(\frac{y}{\lambda_{2}(t)}\right)\,\mathrm{d}x\,\mathrm{d}y\\
&= \sum_{m\in\mathbb{Z}}\sum_{n\in\mathbb{Z}}\int_{m}^{m+1}\int_{n}^{n+1}|u|^{3}\phi_{\sigma'}\left(\frac{x}{\lambda_{1}(t)}\right)\phi_{\delta_{2}}\left(\frac{y}{\lambda_{2}(t)}\right)\,\mathrm{d}x\,\mathrm{d}y\\
&\leq\sum_{m\in\mathbb{Z}}\sum_{n\in\mathbb{Z}}\int_{\mathbb{R}}\int_{\mathbb{R}}\left(|u|\chi_{m}\chi_{n}\right)^{3}\phi_{\sigma'}\left(\frac{x}{\lambda_{1}(t)}\right)\phi_{\delta_{2}}\left(\frac{y}{\lambda_{2}(t)}\right)\,\mathrm{d}x\,\mathrm{d}y\\
&\leq\sum_{m\in\mathbb{Z}}\sum_{n\in\mathbb{Z}}\|u\chi_{m}\chi_{n}\|_{L^{3}_{xy}}^{3}\left(\max_{x\in[n,n+1]}\phi_{\sigma'}\left(\frac{x}{\lambda_{1}(t)}\right)\right)\left(\max_{y\in[m,m+1]}\phi_{\delta_{2}}\left(\frac{y}{\lambda_{2}(t)}\right)\right)\\
& \lesssim \sum_{m\in\mathbb{Z}}\sum_{n\in\mathbb{Z}}\|\nabla\left(u\chi_{m}\chi_{n}\right)\|_{L^{2}_{xy}}\|u\chi_{m}\chi_{n}\|_{L^{2}_{xy}}^{2}\left(\max_{x\in[n,n+1]}\phi_{\sigma'}\left(\frac{x}{\lambda_{1}(t)}\right)\right)\left(\max_{y\in[m,m+1]}\phi_{\delta_{2}}\left(\frac{y}{\lambda_{2}(t)}\right)\right).
\end{split}
\]
Nevertheless, by hypothesis
\begin{equation*}
\begin{split}
\left\|\nabla\left(u\chi_{n}\chi_{m}\right)\right\|_{L^{2}_{xy}}&=\left\|\chi_{n}\chi_{m}\nabla u+u\nabla(\chi_{n}\chi_{m})\right\|_{L^{2}_{xy}}\\
&\lesssim \|u(t)\|_{H^{1}(\mathbb{R}^{2})} \leq c\|u\|_{L_{t}^{\infty}H^{1}}<\infty.
\end{split}
\end{equation*}
Therefore, we obtain the simpler bound
\begin{equation}\label{cubic}
\begin{split}
&\int_{\mathbb{R}^{2}}|u|^{3}\phi_{\sigma'}\left(\frac{x}{\lambda_{1}(t)}\right)\phi_{\delta_{2}}\left(\frac{y}{\lambda_{2}(t)}\right)\,\mathrm{d}x\,\mathrm{d}y\\
& \lesssim \sum_{m\in\mathbb{Z}}\sum_{n\in\mathbb{Z}}  \|u\chi_{m}\chi_{n}\|_{L^{2}_{xy}}^{2}\left(\max_{x\in[n,n+1]}\phi_{\sigma'}\left(\frac{x}{\lambda_{1}(t)}\right)\right)\left(\max_{y\in[m,m+1]}\phi_{\delta_{2}}\left(\frac{y}{\lambda_{2}(t)}\right)\right).
\end{split}
\end{equation}
Also, from \eqref{sigmabound},
\begin{equation*}
e^{-\frac{|x|}{\sigma'\lambda_{1}(t)}}\leq \phi_{\sigma'}\left(\frac{x}{\lambda_{1}(t)}\right)\leq 3e^{-\frac{|x|}{\sigma'\lambda_{1}(t)}}.
\end{equation*}
Then, we consider the  cases described below.
\begin{itemize}
	\item Case $x>0:$
	In this case we have that
	\begin{equation}\label{e1.1}
	\begin{split}
	\max_{x\in[n,n+1]}\phi_{\sigma'}\left(\frac{x}{\lambda_{1}(t)}\right) \leq 	3\max_{x\in[n,n+1]}e^{-\frac{|x|}{\sigma'\lambda_{1}(t)}}=3e^{-\frac{n}{\sigma'\lambda_{1}(t)}}.
	\end{split}
	\end{equation}
	Instead, the minimum value at the same interval is given by
	\begin{equation}\label{e1.2}
	\begin{split}
	\min_{x\in[n,n+1]}\phi_{\sigma'}\left(\frac{x}{\lambda_{1}(t)}\right)\geq 	\min_{x\in[n,n+1]}e^{-\frac{|x|}{\sigma'\lambda_{1}(t)}}=e^{-\frac{1}{\sigma'\lambda_{1}(t)}}e^{-\frac{n}{\sigma'\lambda_{1}(t)}},
	\end{split}
	\end{equation}
	so that, after combining both inequalities above we get
	\begin{equation}\label{e1.3}
	\max_{x\in[n,n+1]}\phi_{\sigma'}\left(\frac{x}{\lambda_{1}(t)}\right)\leq 3e^{\frac{1}{\sigma'\lambda_{1}(t)}}\min_{x\in[n,n+1]}\phi_{\sigma'}\left(\frac{x}{\lambda_{1}(t)}\right),\quad\mbox{whenever}\quad x>0.
	\end{equation}
	\medskip
	\item Case $x\leq 0:$ By parity, we obtain a similar bound as before.
\end{itemize}
Finally, combining both cases and using that $t\gg 1$, we get that for a universal constant $C$,
	\begin{equation}\label{e1.7}
	\max_{x\in[n,n+1]}\phi_{\sigma'}\left(\frac{x}{\lambda_{1}(t)}\right)\leq  C
	 \min_{x\in[n,n+1]}\phi_{\sigma'}\left(\frac{x}{\lambda_{1}(t)}\right),\quad\mbox{for all } \quad x\in\mathbb{R}.
	\end{equation}
A similar analysis  yields
\[
\max_{y\in[m,m+1]}\phi_{\delta_{2}}\left(\frac{y}{\lambda_{2}(t)}\right)\leq C 
\min_{y\in[m,m+1]}\phi_{\delta_{2}}\left(\frac{y}{\lambda_{2}(t)}\right),\quad\mbox{for all } \quad y\in\mathbb{R}.
\]
Next, we incorporate the two last estimates above into the  original one  \eqref{cubic}. Combined with the Monotone Converge Theorem, we get
\begin{equation*}
\begin{split}
&\int_{\mathbb{R}^{2}}|u|^{3}\phi_{\sigma'}\left(\frac{x}{\lambda_{1}(t)}\right)\phi_{\delta_{2}}\left(\frac{y}{\lambda_{2}(t)}\right)\,\mathrm{d}x\,\mathrm{d}y\\
& \lesssim \sum_{m\in\mathbb{Z}}\sum_{n\in\mathbb{Z}}\|u\chi_{m}\chi_{n}\|_{L^{2}_{xy}}^{2}\left(\max_{x\in[n,n+1]}\phi_{\sigma'}\left(\frac{x}{\lambda_{1}(t)}\right)\right)\left(\max_{y\in[m,m+1]}\phi_{\delta_{2}}\left(\frac{y}{\lambda_{2}(t)}\right)\right)\\
&\lesssim \sum_{m\in\mathbb{Z}}\sum_{n\in\mathbb{Z}}\left(\int_{\mathbb{R}^{2}}u^{2}\chi_{m}^{2}\chi_{n}^{2}\phi_{\sigma'}\left(\frac{x}{\lambda_{1}(t)}\right)\phi_{\delta_{2}}\left(\frac{y}{\lambda_{2}(t)}\right)\,\mathrm{d}x\,\mathrm{d}y\right)\\
&\lesssim \int_{\mathbb{R}^{2}}u^{2}\phi_{\sigma'}\left(\frac{x}{\lambda_{1}(t)}\right)\phi_{\delta_{2}}\left(\frac{y}{\lambda_{2}(t)}\right)\,\mathrm{d}x\,\mathrm{d}y.
\end{split}
\end{equation*}
(Note that the implicit constants in the inequalities above does not depend on the variable $m$ nor $n.$)

\medskip

In summary, we have proved that
\[
\begin{split}
|A_{1,1,4}(t)| &\lesssim \frac{1}{\eta(t)\lambda_{1}(t)}\int_{\mathbb{R}^{2}}|u|^{3}\phi_{\sigma'}\left(\frac{x}{\lambda_{1}(t)}\right)\phi_{\delta_{2}}\left(\frac{y}{\lambda_{2}(t)}\right)\,\mathrm{d}x\,\mathrm{d}y\\
&\lesssim \frac{1}{\eta(t)\lambda_{1}(t)}\ \int_{\mathbb{R}^{2}}u^{2}\phi_{\sigma'}\left(\frac{x}{\lambda_{1}(t)}\right)\phi_{\delta_{2}}\left(\frac{y}{\lambda_{2}(t)}\right)\,\mathrm{d}x\,\mathrm{d}y.
\end{split}
\]
Nevertheless,  by \eqref{dT_ppal}, we have  for  $\sigma'>0$ satisfying
\[
\frac{1}{\sigma}+\frac{1}{\delta_{1}}\leq \frac{1}{\sigma'},
\]
one has
\[
\int_{\{t\gg1\}}\left(\frac{1}{\eta(t)\lambda_{1}(t)}\ \int_{\mathbb{R}^{2}}u^{2}\phi_{\sigma'}\left(\frac{x}{\lambda_{1}(t)}\right)\phi_{\delta_{2}}\left(\frac{y}{\lambda_{2}(t)}\right)\,\mathrm{d}x\,\mathrm{d}y\right)\,\mathrm{d}t <\infty,
\]
 which clearly implies that $A_{1,1,4}\in L^{1}\left(\{t\gg 1\}\right).$ Summarizing, $A_{1,1}$ defined in \eqref{divisionA11A12} and \eqref{divisionA11} satisfies
\[
A_{1,1}(t) = A_{1,1,1}(t)  +A_{1,1,2}(t)  + A_{int}(t), \qquad A_{int}(t)\in L^{1}\left(\{t\gg 1\}\right).
\]

\medskip

Next, recall $A_{1,2}(t)$ from \eqref{divisionA11A12}. We have
\begin{equation*}
\begin{split}
A_{1,2}(t)&=-\frac{2}{ \lambda_{1}(t)\eta(t)}\int_{\mathbb{R}^{2}}\left(\partial_{x}u\right)^{2}\phi_{\sigma'}\left(\frac{x}{\lambda_{1}(t)}\right)\phi_{\delta_{2}}\left(\frac{y}{\lambda_{2}(t)}\right)\,\mathrm{d}x\,\mathrm{d}y\\
&\quad +\frac{1}{ \lambda_{1}^{3}(t)\eta(t)}\int_{\mathbb{R}^{2}}u^{2}\phi_{\sigma'}''\left(\frac{x}{\lambda_{1}(t)}\right)\phi_{\delta_{2}}\left(\frac{y}{\lambda_{2}(t)}\right)\,\mathrm{d}x\,\mathrm{d}y\\
&\quad -\frac{2}{ \lambda_{1}(t)\eta(t)}\int_{\mathbb{R}^{2}}\left(\partial_{y}u\right)^{2}\phi_{\sigma'}\left(\frac{x}{\lambda_{1}(t)}\right)\phi_{\delta_{2}}\left(\frac{y}{\lambda_{2}(t)}\right)\,\mathrm{d}x\,\mathrm{d}y\\
&\quad +\frac{2}{ \lambda_{1}(t)\eta(t)\lambda_{2}^{2}(t)}\int_{\mathbb{R}^{2}}u^{2}\phi_{\sigma'}\left(\frac{x}{\lambda_{1}(t)}\right)\phi_{\delta_{2}}''\left(\frac{y}{\lambda_{2}(t)}\right)\,\mathrm{d}x\,\mathrm{d}y\\
&\quad +\frac{1}{ \lambda_{1}(t)\eta(t)}\int_{\mathbb{R}^{2}}u^{3}\phi_{\sigma'}\left(\frac{x}{\lambda_{1}(t)}\right)\phi_{\delta_{2}}\left(\frac{y}{\lambda_{2}(t)}\right)\,\mathrm{d}x\,\mathrm{d}y\\
&=A_{1,2,1}(t)+A_{1,2,2}(t)+A_{1,2,3}(t)+A_{1,2,4}(t)+A_{1,2,5}(t).
\end{split}
\end{equation*}
We only will be focus on the terms  $A_{1,2,2}$ and $A_{1,2,4},$ since  $A_{1,2,1}$  is the quantity to be estimated after integrating in time in \eqref{IntegraH1}, the same as $A_{1,2,3}.$ Note additionally that the bad sign term $(\partial_y u)^2$ in \eqref{divisionA11} is solved by adding the term $A_{1,2,3}$. Finally, for $A_{1,2,5}$  we  described above how to obtain upper bounds that implies $A_{1,2,5}\in L^{1}\left(\{t\gg 1\}\right),$ and we omit its proof here.

\medskip

Thus, since from \eqref{eq4}  $3b+p>1$,
\begin{equation*}
\begin{split}
|A_{1,1,2}(t)| &\lesssim_{\|u_{0}\|_{L^{2}_{xy}}} \frac{1}{\lambda_{1}^{3}	(t)\eta(t)}\in L^{1}\left(\{t\gg 1\}\right).
\end{split}
\end{equation*}
For $A_{1,2,4}$  we have that $p+ b+ 2br >1$ and
\begin{equation*}
|A_{1,2,4}(t)| \lesssim_{\|u_{0}\|_{L^{2}_{xy}}}\frac{1}{\eta(t)\lambda_{1}(t)\lambda_{2}^{2}(t)}\in L^{1}\left(\{t\gg 1\}\right).
\end{equation*}
We partially conclude from \eqref{main2} and the previous computations that
\begin{equation}\label{main2_final}
\begin{aligned}
\frac{\mathrm{d}}{\mathrm{d}t}\mathcal{Q}(t)  = &~{}  -\frac{3}{ \lambda_{1}(t)\eta(t)}\int_{\mathbb{R}^{2}}\left(\partial_{x}u\right)^{2}\phi_{\sigma'}\left(\frac{x}{\lambda_{1}(t)}\right)\phi_{\delta_{2}}\left(\frac{y}{\lambda_{2}(t)}\right)\,\mathrm{d}x\,\mathrm{d}y \\
&   -\frac{1}{ \lambda_{1}(t)\eta(t)}\int_{\mathbb{R}^{2}}\left(\partial_{y}u\right)^{2}\phi_{\sigma'}\left(\frac{x}{\lambda_{1}(t)}\right)\phi_{\delta_{2}}\left(\frac{y}{\lambda_{2}(t)}\right)\,\mathrm{d}x\,\mathrm{d}y \\
&  + A_{int}(t) +A_2(t) +A_3(t) +A_4(t),
\end{aligned}
\end{equation}
with $A_{int}(t)\in L^{1}\left(\{t\gg 1\}\right)$. Finally, we  consider the remainders terms in \eqref{main2}. First,
\[
|A_{2}(t)|\lesssim_{\|u_{0}\|_{L^{2}_{xy}}}\frac{\lambda_{2}'(t)}{\eta(t)\lambda_{2}(t)}\in L^{1}\left(\{t\gg 1\}\right),
\]
since $p>0.$ The term $A_3(t)$ is completely similar. Finally,
\begin{equation*}
|A_{4}(t)|\lesssim_{\|u_{0}\|_{L^{2}_{xy}}} \frac{\eta'(t)}{\eta^{2}(t)}\in L^{1}\left(\{t\gg 1\}\right),
\end{equation*}
since $p>0.$ Gathering these estimates in \eqref{main2_final}, we conclude \eqref{IntegraH1} in the same vein as in the proof of Theorem \ref{Thmdim2L2}.

%


\bigskip

\section{The 3D case. Proof of Theorem \ref{Thmdim3H1}}\label{Sect:5}

\medskip

The proof in the 3D case follow similar lines as the one in 2D, but it is clearly more cumbersome. 

\subsection{$L^1$ virial and $L^2$ local decay} As we did in the previous case,  we start by defining the functional that will provide the mass behavior associated with the solutions to \eqref{ZK:Eq}  in the case $d=3$. More precisely, we set
\begin{equation}\label{e1}
\Xi(t)=\frac{1}{\eta(t)}\int_{\mathbb{R}^{3}}u(x,y,z,t)\psi_{\sigma}\left(\frac{x}{\lambda_{1}(t)}\right)\phi_{\delta_{1}}\left(\frac{x}{\lambda_{2}(t)}\right)\phi_{\delta_{2}}\left(\frac{y}{\lambda_{3}(t)}\right)\phi_{\delta_{3}}\left(\frac{z}{\lambda_{4}(t)}\right)\,\mathrm{d}x\mathrm{d}y\mathrm{d}z.
\end{equation}
where
\begin{equation}\label{lambdas3}
\lambda_{1}(t)=\frac{t^{p_{1}}}{\ln^{q_{1}}(t)},\quad  \lambda_{2}(t)=t^{p_{2}}, \quad \lambda_{3}(t)= t^{p_{3}}, \quad \mbox{and}\quad \lambda_{4}(t)=t^{p_{4}}, 
\end{equation}
with $p_{1},p_{2},p_{3},p_{4} > 0$ and $q_{1},q_{2},q_{3},q_{4}> 0$ parameters to be determined.
Also, we consider
\begin{equation}\label{eta3}
\eta(t)=t^{r_{1}}\ln^{r_{2}}(t),
\end{equation}
We will consider $p_{1},r_{1}$ and $q_{1},r_{2}$ satisfying the following conditions:
\begin{equation}\label{e5}
r_{1}=1-p_1,\quad p_{1},r_{1}>0, \quad r_{2}=1+q_{1},\,q_{1}>0.
\end{equation}
The  parameters $p_{1},p_{2},p_{3}$ and $p_{4},$  are chosen in such a way  to  satisfy the following array of conditions
\begin{equation}\label{e40}
p_1,p_2,p_3,p_4>0, \quad p_1<1,
\end{equation}
\begin{equation}\label{e4}
0<2p_{1}+p_{2}+p_{3}+p_{4}< 2,
\end{equation}
\begin{equation}\label{e6}
p_2>p_1,
\end{equation}
\begin{equation}\label{new1}
p_3>p_1,
\end{equation}
\begin{equation}\label{new2}
p_4>p_1,
\end{equation}
\begin{equation}\label{e7}
p_{1}> \frac13(p_3+ p_4),
\end{equation}
\begin{equation}\label{e8}
\frac12p_1+ p_{2}> \frac12 (p_3+p_4),
\end{equation}
\begin{equation}\label{e9}
p_{2}> \frac14(p_1+ p_3 +p_4),
\end{equation}
\begin{equation}\label{e10}
p_{3}>\frac13 (p_1+ p_{4}),
\end{equation}
\begin{equation}\label{e11}
p_{4}>\frac13(p_{1}+p_{3}),
\end{equation}
\begin{equation}\label{e13}
p_2 >\frac15(2p_1 +p_3+p_4),
\end{equation}
\begin{equation}\label{e14}
p_2+3p_3 >p_4+2p_1,
\end{equation}
\begin{equation}\label{e15}
p_2+3p_4 >p_3+2p_1,
\end{equation}
and
\begin{equation}\label{e16}
3p_{1}+ (p_{3}+p_{4})<2.
\end{equation}
Recall \eqref{lambdas3}. Define now
\begin{equation}\label{mP}
\mathcal P:=\{  (p_1,p_2,p_3,p_4) \in (0,\infty)^4 ~  : ~  \eqref{e40}-\eqref{e16} \hbox{ are satisfied} \},
\end{equation}
and
\[
\Omega(t):= \left\{ (x,y,z)\in \mathbb R^3 ~  :  ~  |x|\leq \lambda_1(t), \quad |y|\leq \lambda_3(t),\quad |z|\leq \lambda_4(t), \quad   (p_1,p_2,p_3,p_4) \in\mathcal P \right\}.
\]
This set is precisely the set \eqref{brconditions3D} stated in Theorem \ref{Thmdim3H1}, but with plenty of redundant conditions inside.  Before continuing, we need some simplifications in conditions \eqref{e5}-\eqref{e16}, that will lead to the simpler definition of $\Omega(t)$ in \eqref{brconditions3D}.

\begin{lem}\label{easy_P}
$\mathcal P$ is nonempty. Moreover, it can be reduced to the following simpler set of conditions
\begin{equation*}
\mathcal P = \Big\{  (p_1,p_2,p_3,p_4) \in (0,\infty)^4 ~  : ~   \eqref{e40}-\eqref{e7}, \eqref{e10} \hbox{ and } \eqref{e11} \hbox{ are satisfied. }\Big\}.
\end{equation*}
In particular, the conditions $p_1<\frac12$ and $p_2,p_3,p_4>p_1$ must be satisfied.
\end{lem}

Using this result, it is easy to describe the set $\Omega(t)$ in  \eqref{brconditions3D}, probably taking $b$ slightly smaller if needed. Just redefine $p_1=b$, $p_3=br_1$ and $p_4=br_2$. The parameter $p_2$ can be written as $p_1+\epsilon_0=b+\epsilon_0$, any $\epsilon_0>0$, and it is a free parameter, leading to the last, nonlinear condition in \eqref{brconditions3D}. Also, the two conditions \eqref{new1}-\eqref{new2} are not needed for proving $L^2$ decay, only for proving decay of the $\dot H^1$ norm.

\begin{proof}[Proof of Lemma \ref{easy_P}]
It is easy to check that for any $\epsilon_0>0$ but small, the point $\left(\frac13, \frac13+\epsilon_0,\frac13,\frac13 \right)$ belongs to $\mathcal P$. On the other hand, after defining
\[
\tilde p_2 := \frac{p_2}{p_1}, \quad \tilde p_3 := \frac{p_3}{p_1}, \quad \tilde p_4 := \frac{p_4}{p_1},
\]
equations \eqref{e7}, \eqref{e10} and \eqref{e11} become
\begin{equation}\label{triangle}
1> \frac13( \tilde p_3+ \tilde p_4), \quad \tilde p_{3}>\frac13 (1+ \tilde p_{4}), \quad \tilde p_{4}>\frac13(1+ \tilde p_{3}).
\end{equation}
The solution to this system of inequalities corresponds to the interior of the triangle of vertices $(\frac12,\frac12)$, $(2,1)$ and $(1,2)$ in the $\tilde p_3-\tilde p_4$ plane (see Fig. \ref{fig:5}).
The extrema of the function $a:=\tilde p_3+ \tilde p_4$ in this set satisfies
\[
1<a=\tilde p_3+ \tilde p_4 <3.
\]
Also, it is not difficult to check
\begin{equation}\label{supremos}
\tilde p_4 -3 \tilde p_3 <-1, \quad \tilde p_3 -3 \tilde p_4 <-1.
\end{equation}
Now, equations \eqref{e6}, \eqref{e8}, \eqref{e9} and \eqref{e13} reduce to
\begin{equation}\label{lines}
\tilde p_2 >1, \quad  \tilde p_{2}> \frac12 (a-1), \quad \tilde p_{2}> \frac14(1+ a), \quad \tilde p_2 >\frac15(2 +a).
\end{equation}
A quick checking (see Fig. \ref{fig:5}) reveals that the condition $\tilde p_2 >1$ is the most restrictive one. Also, from \eqref{e14} and \eqref{e15} one has
\[
\tilde p_2  > \tilde p_4 -  3 \tilde p_3 +2, \qquad \tilde p_2 > \tilde p_3 -3\tilde p_4 +2.
\]
From the condition $\tilde p_2>1$ and \eqref{supremos} we have that the two last conditions are redundant.

\medskip

Now we consider the conditions \eqref{e4} and \eqref{e16}. Written in terms of normalized variables, one has
\[
0<2+\tilde p_{2}+a< \frac{2}{p_1}, \qquad 3+ a <\frac{2}{p_1}.
\]
Since $a>0$ and $\tilde p_2>1$, always $2+\tilde p_{2}+a>3+ a$. Therefore, the last condition is redundant. Finally, \eqref{e40} and $2+\tilde p_{2}+a< \frac{2}{p_1}$ imply that $p_1$ must be below $\frac12$.

\medskip

There are two remaining conditions to be considered. These are \eqref{new1} and \eqref{new2}, which become $\tilde p_3>1$ and $\tilde p_4>1$. The representation of these conditions can be found in Fig. \ref{fig:5}, left panel.

\medskip

This ends the proof of the lemma.
\end{proof}

\begin{figure}[h!]
\begin{center}
\begin{tikzpicture}[scale=0.9]
\filldraw[thick, color=lightgray!30] (1,1) -- (4,2) -- (2,4) -- (1,1);
\filldraw[thick, color=lightgray!70] (2,2) -- (4,2) -- (2,4) -- (2,2);
\draw[thick, dashed] (1,1) -- (4,2) -- (2,4) -- (1,1);
\draw[thick,dashed] (4,1) -- (1,4);
\draw[thick,dashed] (2,0) -- (2,4);
\draw[thick,dashed] (0,2) -- (4,2);
\draw[->] (-0.5,0) -- (5,0) node[below] {$\tilde p_3$};
\draw[->] (0,-0.5) -- (0,5) node[right] {$\tilde p_4$};
\node at (1,0){$\bullet$};
\node at (1,-0.5){$\frac 12$};
\node at (4,0){$\bullet$};
\node at (4,-0.5){$2$};
\node at (0,4){$\bullet$};
\node at (-0.5,4){$2$};
\node at (0,2){$\bullet$};
\node at (-0.5,2){$1$};
\node at (0,1){$\bullet$};
\node at (-0.5,1){$\frac12$};
\node at (2,0){$\bullet$};
\node at (2,-0.5){$1$};
\node at (5.3,1){\small $\tilde p_3+\tilde p_4=a$};
\end{tikzpicture}
\quad
\begin{tikzpicture}[scale=0.7]
\filldraw[thick, color=lightgray!90] (2,4)--(6,4) -- (6,5) --(2,5);
\filldraw[thick, color=lightgray!50] (2,12/5)--(6,4) -- (2,4)--(2,12/5);
\filldraw[thick, color=lightgray!30] (2,2)--(6,4) -- (2,12/5)--(2,2);
\filldraw[thick, color=lightgray!10] (2,0)--(6,4) -- (2,2)--(2,0);
\draw[thick,dashed] (0,4) -- (6,4);
\draw[thick,dashed] (0,1.6) -- (6,4);
\draw[thick,dashed] (0,1) -- (6,4);
\draw[thick,dashed] (0,-2) -- (6,4);
\draw[thick,dashed] (2,0) -- (2,5);
\draw[thick,dashed] (6,0) -- (6,5);
\draw[->] (-0.5,0) -- (7,0) node[below] {$a$};
\draw[->] (0,-2.5) -- (0,5.5) node[right] {$\tilde p_2$};
\node at (6,0){$\bullet$};
\node at (6,4){$\bullet$};
\node at (6,-0.5){$3$};
\node at (2,0){$\bullet$};
\node at (2,-0.5){$1$};
\node at (0,4){$\bullet$};
\node at (-0.5,4){$1$};
\node at (0,1.6){$\bullet$};
\node at (-0.5,1.6){$\frac25$};
\node at (0,1){$\bullet$};
\node at (0.5,0.7){$\frac14$};
\node at (0,-2){$\bullet$};
\node at (-0.5,-2){$-\frac12$};
\end{tikzpicture}
\end{center}
\caption{(\emph{Left}). Representation of the factible set described in \eqref{triangle} (light and dark shadowed region). If one includes \eqref{new1} and \eqref{new2}, then only the region $\tilde p_3,\tilde p_4> 1$ must be considered (dark shadowed region). (\emph{Right}). The intersection of the regions described in \eqref{lines} (dark shadowed region). }\label{fig:5}
\end{figure}
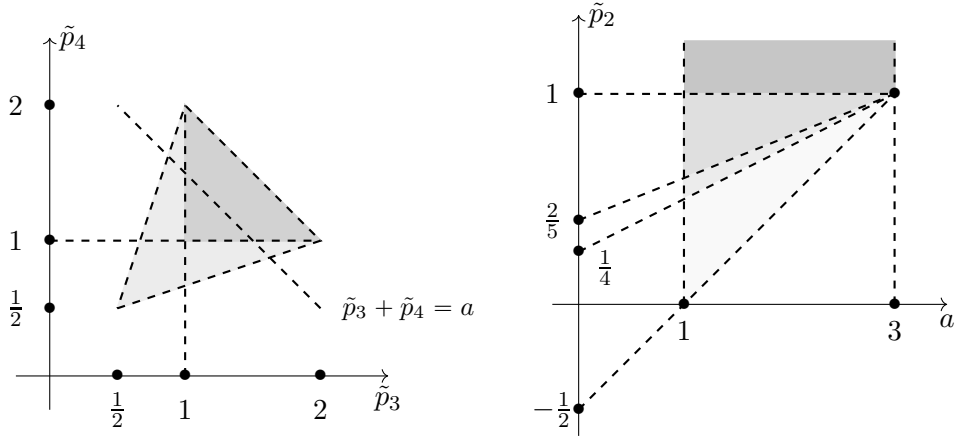

\bigskip

Now we continue with the estimate of $\Xi(t)$.

\medskip

\noindent
{\sc Claim:} Under \eqref{e5} and \eqref{e4}, the functional \eqref{e1} is well-defined.
\begin{proof}
Assume \eqref{e5} and \eqref{e4}. Since $u\in L^{2}$ and the mass is conserved, it is clear that
\[
\begin{split}
\left|\Xi(t)\right|&\lesssim_{\|u_{0}\|_{L^{2}}} \frac{\left(\lambda_{2}(t)\lambda_{3}(t)\lambda_{4}(t)\right)^{1/2}}{\eta(t)}.
\end{split}
\]
Therefore, from \eqref{lambdas3} and \eqref{eta3},
\[
\sup_{t\gg1}|\Xi(t)|<\infty,
\]
and the claim is proved.
\end{proof}

\subsubsection{Mass behavior} Now we compute the evolution of $\Xi(t)$. We follow similar estimates as the one performed to prove Lemma \ref{le:dT}. First of all,
\begin{equation}\label{ee3}
\begin{split}
\frac{\mathrm{d}}{\mathrm{d}t}\Xi(t)&=\frac{1}{\eta(t)}\int_{\mathbb{R}^{2}} \partial_{t}u \psi_{\sigma}\left(\frac{x}{\lambda_{1}(t)}\right)\phi_{\delta_{1}}\left(\frac{x}{\lambda_{2}(t)}\right)\phi_{\delta_{2}}\left(\frac{y}{\lambda_{3}(t)}\right)\phi_{\delta_{3}}\left(\frac{z}{\lambda_{4}(t)}\right)\mathrm{d}x\mathrm{d}y\mathrm{d}z\\
&\quad- \frac{\eta'(t)}{\eta^{2}(t)}\int_{\mathbb{R}^{2}}u\psi_{\sigma}\left(\frac{x}{\lambda_{1}(t)}\right)\phi_{\delta_{1}}\left(\frac{x}{\lambda_{2}(t)}\right)\phi_{\delta_{2}}\left(\frac{y}{\lambda_{3}(t)}\right)\phi_{\delta_{3}}\left(\frac{z}{\lambda_{4}(t)}\right)\mathrm{d}x\mathrm{d}y\mathrm{d}z\\
&\quad +\frac{1}{\eta(t)}\int_{\mathbb{R}^{2}} u\partial_{t}\left( \psi_{\sigma}\left(\frac{x}{\lambda_{1}(t)}\right)\phi_{\delta_{1}}\left(\frac{x}{\lambda_{2}(t)}\right)\phi_{\delta_{2}}\left(\frac{y}{\lambda_{3}(t)}\right)\phi_{\delta_{3}}\left(\frac{z}{\lambda_{4}(t)}\right)\right)\mathrm{d}x\mathrm{d}y\mathrm{d}z\\
&=\Xi_{1}(t)+ \Xi_{2}(t)+ \Xi_{3}(t).
\end{split}
\end{equation}
First, we bound  $\Xi_{1}$. Using \eqref{ZK:Eq} in the 3D case, we get after applying integration by parts
\begin{equation}\label{ee4}
\begin{split}
&\Xi_{1}(t)\\
&=\frac{1}{\eta(t)}\int_{\mathbb{R}^{3}}\left(\Delta u+ \frac{u^{2}}{2}\right) \partial_x \left( \psi_{\sigma}\left(\frac{x}{\lambda_{1}(t)}\right)\phi_{\delta_{1}}\left(\frac{x}{\lambda_{2}(t)}\right) \right) \phi_{\delta_{2}}\left(\frac{y}{\lambda_{3}(t)}\right)\phi_{\delta_{3}}\left(\frac{z}{\lambda_{4}(t)}\right)\mathrm{d}x\mathrm{d}y\mathrm{d}z\\
&=\frac{1}{\eta(t)\lambda_{1}(t)}\int_{\mathbb{R}^{3}}\left(\Delta u+ \frac{u^{2}}{2}\right)\phi_{\sigma}\left(\frac{x}{\lambda_{1}(t)}\right)\phi_{\delta_{1}}\left(\frac{x}{\lambda_{2}(t)}\right)\phi_{\delta_{2}}\left(\frac{y}{\lambda_{3}(t)}\right)\phi_{\delta_{3}}\left(\frac{z}{\lambda_{4}(t)}\right)\mathrm{d}x\mathrm{d}y\mathrm{d}z\\
&\quad+ \frac{1}{\eta(t)\lambda_{2}(t)}\int_{\mathbb{R}^{3}}\left(\Delta u +\frac{u^{2}}{2}\right)\psi_{\sigma}\left(\frac{x}{\lambda_{1}(t)}\right)\phi_{\delta_{1}}'\left(\frac{x}{\lambda_{2}(t)}\right)\phi_{\delta_{2}}\left(\frac{y}{\lambda_{3}(t)}\right)\phi_{\delta_{3}}\left(\frac{z}{\lambda_{4}(t)}\right)\mathrm{d}x\mathrm{d}y\mathrm{d}z\\
&= \Xi_{1,1}(t)+\Xi_{1,2}(t).
\end{split}
\end{equation}
The term $\Xi_{1,1}$ will be treated as follows:
\begin{equation}\label{eee1}
\begin{aligned}
& \Xi_{1,1}(t) =\\
& = \frac{1}{\eta(t)\lambda_{1}(t)}\int_{\mathbb{R}^{3}} \Delta u \phi_{\sigma}\left(\frac{x}{\lambda_{1}(t)}\right)\phi_{\delta_{1}}\left(\frac{x}{\lambda_{2}(t)}\right)\phi_{\delta_{2}}\left(\frac{y}{\lambda_{3}(t)}\right)\phi_{\delta_{3}}\left(\frac{z}{\lambda_{4}(t)}\right)\mathrm{d}x\mathrm{d}y\mathrm{d}z\\
& \quad +\frac{1}{\eta(t)\lambda_{1}(t)}\int_{\mathbb{R}^{3}} \frac{u^{2}}{2} \phi_{\sigma}\left(\frac{x}{\lambda_{1}(t)}\right)\phi_{\delta_{1}}\left(\frac{x}{\lambda_{2}(t)}\right)\phi_{\delta_{2}}\left(\frac{y}{\lambda_{3}(t)}\right)\phi_{\delta_{3}}\left(\frac{z}{\lambda_{4}(t)}\right)\mathrm{d}x\mathrm{d}y\mathrm{d}z\\
& = \Xi_{1,1,1}(t) +\Xi_{1,1,2}(t).
\end{aligned}
\end{equation}
The term $\Xi_{1,1,2}(t)$ is precisely the term that we want to estimate, and we save it. For $\Xi_{1,1,1}$, using integration by parts, we obtain
\begin{equation*}
\begin{split}
&\Xi_{1,1,1}(t)\\
& = \frac{1}{\eta(t)\lambda_{1}(t)}\int_{\mathbb{R}^{3}} \Delta u \phi_{\sigma}\left(\frac{x}{\lambda_{1}(t)}\right)\phi_{\delta_{1}}\left(\frac{x}{\lambda_{2}(t)}\right)\phi_{\delta_{2}}\left(\frac{y}{\lambda_{3}(t)}\right)\phi_{\delta_{3}}\left(\frac{z}{\lambda_{4}(t)}\right)\mathrm{d}x\mathrm{d}y\mathrm{d}z\\
&=\frac{1}{\eta(t)(\lambda_{1}(t))^{3}}\int_{\mathbb{R}^{3}}u\phi_{\sigma}''\left(\frac{x}{\lambda_{1}(t)}\right)\phi_{\delta_{1}}\left(\frac{x}{\lambda_{2}(t)}\right)\phi_{\delta_{2}}\left(\frac{y}{\lambda_{3}(t)}\right)\phi_{\delta_{3}}\left(\frac{z}{\lambda_{4}(t)}\right)\mathrm{d}x\mathrm{d}y\mathrm{d}z\\
&\quad +\frac{2}{\eta(t)(\lambda_{1}(t))^{2}\lambda_{2}(t)}\int_{\mathbb{R}^{3}}u\phi_{\sigma}'\left(\frac{x}{\lambda_{1}(t)}\right)\phi_{\delta_{1}}'\left(\frac{x}{\lambda_{2}(t)}\right)\phi_{\delta_{2}}\left(\frac{y}{\lambda_{3}(t)}\right)\phi_{\delta_{3}}\left(\frac{z}{\lambda_{4}(t)}\right)\mathrm{d}x\mathrm{d}y\mathrm{d}z\\
&\quad+ \frac{1}{\eta(t)\lambda_{1}(t)(\lambda_{2}(t))^{2}}\int_{\mathbb{R}^{3}}u\phi_{\sigma}\left(\frac{x}{\lambda_{1}(t)}\right)\phi_{\delta_{1}}''\left(\frac{x}{\lambda_{2}(t)}\right)\phi_{\delta_{2}}\left(\frac{y}{\lambda_{3}(t)}\right)\phi_{\delta_{3}}\left(\frac{z}{\lambda_{4}(t)}\right)\mathrm{d}x\mathrm{d}y\mathrm{d}z\\
&\quad +\frac{1}{\eta(t)\lambda_{1}(t)(\lambda_{3}(t))^{2}}\int_{\mathbb{R}^{3}}u\phi_{\sigma}\left(\frac{x}{\lambda_{1}(t)}\right)\phi_{\delta_{1}}\left(\frac{x}{\lambda_{2}(t)}\right)\phi_{\delta_{2}}''\left(\frac{y}{\lambda_{3}(t)}\right)\phi_{\delta_{3}}\left(\frac{z}{\lambda_{4}(t)}\right)\mathrm{d}x\mathrm{d}y\mathrm{d}z\\
&\quad +\frac{1}{\eta(t)\lambda_{1}(t)(\lambda_{4}(t))^{2}}\int_{\mathbb{R}^{3}}u\phi_{\sigma}\left(\frac{x}{\lambda_{1}(t)}\right)\phi_{\delta_{1}}\left(\frac{x}{\lambda_{2}(t)}\right)\phi_{\delta_{2}}\left(\frac{y}{\lambda_{3}(t)}\right)\phi_{\delta_{3}}''\left(\frac{z}{\lambda_{4}(t)}\right)\mathrm{d}x\mathrm{d}y\mathrm{d}z\\
&=\Xi_{1,1,1,1}(t)+\Xi_{1,1,1,2}(t)+\Xi_{1,1,1,3}(t)+\Xi_{1,1,1,4}(t)+\Xi_{1,1,1,5}(t).
\end{split}
\end{equation*}
Recall that $p_2>p_1$ from \eqref{e6}. Therefore, $\Xi_{1,1,1,1}$ is bounded as
\[
\begin{split}
|\Xi_{1,1,1,1}(t)| &\lesssim \frac{(\lambda_{3}(t)\lambda_{4}(t))^{1/2}}{\eta(t)(\lambda_{1}(t))^{5/2}}\in L^{1}\left(\{t\gg 1\}\right),
\end{split}
\]
since $p_{1}> \frac13(p_3+ p_4)$ in \eqref{e7}. Next,
\[
|\Xi_{1,1,1,2}(t)| \lesssim \frac{(\lambda_{3}(t)\lambda_{4}(t))^{1/2}}{\eta(t)(\lambda_{1}(t))^{3/2}\lambda_{2}(t)}\in  L^{1}\left(\{t\gg 1\}\right),
\]
whenever $\frac{p_{1}}{2} + p_{2}>\frac{p_{3}+p_{4}}{2}$ in \eqref{e8}.

\medskip

The estimates for $\Xi_{1,1,1,3}(t)$, $\Xi_{1,1,1,4}(t)$, and $\Xi_{1,1,1,5}(t)$ are very similar in nature. For $\Xi_{1,1,1,3}$  we have that
\[
|\Xi_{1,1,1,3}(t)|\lesssim \frac{(\lambda_{3}(t))^{1/2}(\lambda_{4}(t))^{1/2}}{\eta(t)(\lambda_{1}(t))^{1/2}(\lambda_{2}(t))^{2}}\in  L^{1}\left(\{t\gg 1\}\right),
\]
since $p_{2}> \frac14(p_1+ p_3 +p_4)$ in \eqref{e9}. Next, for $\Xi_{1,1,1,4}$  we have that
\[
|\Xi_{1,1,1,4}(t)|\lesssim \frac{(\lambda_{4}(t))^{1/2}}{\eta(t)(\lambda_{1}(t))^{1/2}(\lambda_{3}(t))^{3/2}}\in L^{1}\left(\{t\gg 1\}\right),
\]
since $p_{3}>\frac13 (p_1+ p_{4})$ in \eqref{e10}. Finally, we focus our attention on $\Xi_{1,1,1,5},$ to get the bound
\[
|\Xi_{1,1,1,5}(t)|\lesssim \frac{(\lambda_{3}(t))^{1/2}}{\eta(t)(\lambda_{1}(t))^{1/2}(\lambda_{4}(t))^{3/2}}\in L^{1}\left(\{t\gg 1\}\right),
\]
valid for $p_{4}>\frac13(p_{1}+p_{3})$, thanks to \eqref{e11}. We conclude that
\[
|\Xi_{1,1,1}(t)| \lesssim \frac{1}{t^{1+\kappa_0} \log^{\epsilon_0}(t)} \in L^{1}\left(\{t\gg 1\}\right),
\]
for some fixed positive constants  $\kappa_0$ and $\epsilon_0$ depending on the parameters $p_i$, $q_i$, $i=1,2,3,4$.

\medskip

Concerning the term $\Xi_{1,2}(t)$ from \eqref{ee4}, we first set
\begin{equation}\label{e_12}
\begin{split}
&\Xi_{1,2}(t) \\
&= \frac{1}{\eta(t)\lambda_{2}(t)}\int_{\mathbb{R}^{3}}\Delta u \psi_{\sigma}\left(\frac{x}{\lambda_{1}(t)}\right)\phi_{\delta_{1}}'\left(\frac{x}{\lambda_{2}(t)}\right)\phi_{\delta_{2}}\left(\frac{y}{\lambda_{3}(t)}\right)\phi_{\delta_{3}}\left(\frac{z}{\lambda_{4}(t)}\right)\mathrm{d}x\mathrm{d}y\mathrm{d}z\\
&\quad +\frac{1}{2\eta(t)\lambda_{2}(t)}\int_{\mathbb{R}^{3}}u^{2}\psi_{\sigma}\left(\frac{x}{\lambda_{1}(t)}\right)\phi_{\delta_{1}}'\left(\frac{x}{\lambda_{2}(t)}\right)\phi_{\delta_{2}}\left(\frac{y}{\lambda_{3}(t)}\right)\phi_{\delta_{3}}\left(\frac{z}{\lambda_{4}(t)}\right)\mathrm{d}x\mathrm{d}y\mathrm{d}z\\
&=\Xi_{1,2,1}(t)+\Xi_{1,2,2}(t).
\end{split}
\end{equation}
Thus, after applying integration by parts,
\begin{equation}\label{e121}
\begin{split}
&\Xi_{1,2,1}(t)\\
&=\frac{1}{\eta(t)(\lambda_{1}(t))^{2}\lambda_{2}(t)}\int_{\mathbb{R}^{3}}u\phi_{\sigma}' \left(\frac{x}{\lambda_{1}(t)}\right)\phi_{\delta_{1}}'\left(\frac{x}{\lambda_{2}(t)}\right)\phi_{\delta_{2}}\left(\frac{y}{\lambda_{3}(t)}\right)\phi_{\delta_{3}}\left(\frac{z}{\lambda_{4}(t)}\right)\mathrm{d}x\mathrm{d}y\mathrm{d}z\\
&\quad +\frac{2}{\eta(t)\lambda_{1}(t)(\lambda_{2}(t))^{2}}\int_{\mathbb{R}^{3}}u\phi_{\sigma}\left(\frac{x}{\lambda_{1}(t)}\right)\phi_{\delta_{1}}''\left(\frac{x}{\lambda_{2}(t)}\right)\phi_{\delta_{2}}\left(\frac{y}{\lambda_{3}(t)}\right)\phi_{\delta_{3}} \left(\frac{z}{\lambda_{4}(t)}\right)\mathrm{d}x\mathrm{d}y\mathrm{d}z\\
&\quad +\frac{1}{\eta(t)(\lambda_{2}(t))^{3}}\int_{\mathbb{R}^{3}}u\psi_{\sigma}\left(\frac{x}{\lambda_{1}(t)}\right)\phi_{\delta_{1}}'''\left(\frac{x}{\lambda_{2}(t)}\right)\phi_{\delta_{2}}\left(\frac{y}{\lambda_{3}(t)}\right)\phi_{\delta_{3}} \left(\frac{z}{\lambda_{4}(t)}\right)\mathrm{d}x\mathrm{d}y\mathrm{d}z\\
&\quad +\frac{1}{\eta(t)\lambda_{2}(t) (\lambda_{3}(t))^{2}}\int_{\mathbb{R}^{3}}u\psi_{\sigma}\left(\frac{x}{\lambda_{1}(t)}\right)\phi_{\delta_{1}}'\left(\frac{x}{\lambda_{2}(t)}\right)\phi_{\delta_{2}}''\left(\frac{y}{\lambda_{3}(t)}\right)\phi_{\delta_{3}} \left(\frac{z}{\lambda_{4}(t)}\right)\mathrm{d}x\mathrm{d}y\mathrm{d}z\\
&\quad +\frac{1}{\eta(t)\lambda_{2}(t)(\lambda_{4}(t))^2}\int_{\mathbb{R}^{3}}u\psi_{\sigma}\left(\frac{x}{\lambda_{1}(t)}\right)\phi_{\delta_{1}}'\left(\frac{x}{\lambda_{2}(t)}\right)\phi_{\delta_{2}}\left(\frac{y}{\lambda_{3}(t)}\right)\phi_{\delta_{3}}'' \left(\frac{z}{\lambda_{4}(t)}\right)\mathrm{d}x\mathrm{d}y\mathrm{d}z\\
&=\Xi_{1,2,1,1}(t)+\Xi_{1,2,1,2}(t)+\Xi_{1,2,1,3}(t)++\Xi_{1,2,1,4}(t)++\Xi_{1,2,1,5}(t).
\end{split}
\end{equation}
In the first place,
\[
\begin{split}
| \Xi_{1,2,1,1}(t)| \lesssim \frac{(\lambda_{3}(t)\lambda_{4}(t))^{1/2}}{\eta(t)(\lambda_{1}(t))^{3/2}\lambda_{2}(t)}\in L^{1}\left(\left\{t\gg 1\right\}\right),
\end{split}
\]
for $\frac{p_{1}}{2}+p_{2}>\frac{p_{3}+p_{4}}{2}$ as in \eqref{e8}. Next,
\[
|\Xi_{1,2,1,2}(t) | \lesssim \frac{(\lambda_{3}(t)\lambda_{4}(t))^{1/2}}{\eta(t)(\lambda_{1}(t))^{1/2}(\lambda_{2}(t))^{2}}\in L^{1}\left(\left\{t\gg 1\right\}\right),
\]
since $p_{2}>\frac{p_{1}+p_{3}+p_{4}}{4}$ from \eqref{e9}. Similarly,
\[
|\Xi_{1,2,1,3}(t) | \lesssim \frac{(\lambda_{3}(t)\lambda_{4}(t))^{1/2}}{\eta(t)(\lambda_{2}(t))^{5/2}}\in L^{1}\left(\left\{t\gg 1\right\}\right),
\]
for $p_2 >\frac15(2p_1 +p_3+p_4)$, see \eqref{e13}.  Again,
\[
|\Xi_{1,2,1,4}(t) | \lesssim \frac{(\lambda_{4}(t))^{1/2}}{\eta(t)(\lambda_{2}(t))^{1/2}(\lambda_{3}(t))^{3/2}}\in L^{1}\left(\left\{t\gg 1\right\}\right),
\]
for $p_2+3p_3 >p_4+2p_1$ as in \eqref{e14}. We conclude with
\[
|\Xi_{1,2,1,5}(t) | \lesssim \frac{(\lambda_{3}(t))^{1/2}}{\eta(t)(\lambda_{2}(t))^{1/2} (\lambda_{4}(t))^{3/2}}\in L^{1}\left(\left\{t\gg 1\right\}\right),
\]
for $p_2+3p_4 >p_3+2p_1$ as in \eqref{e15}. We conclude that $\Xi_{1,2,1}(t)$ satisfies
\[
\Xi_{1,2,1}(t) \in L^{1}\left(\left\{t\gg 1\right\}\right).
\]
Finally, from \eqref{e_12},
\[
\begin{split}
&|\Xi_{1,2,2}(t)|\\
&=\left| \frac{1}{2\eta(t)\lambda_{2}(t)}\int_{\mathbb{R}^{3}}u^{2}\psi_{\sigma}\left(\frac{x}{\lambda_{1}(t)}\right)\phi_{\delta_{1}}'\left(\frac{x}{\lambda_{2}(t)}\right)\phi_{\delta_{2}}\left(\frac{y}{\lambda_{3}(t)}\right)\phi_{\delta_{3}} \left(\frac{z}{\lambda_{4}(t)}\right)\mathrm{d}x\mathrm{d}y\mathrm{d}z \right|\\
&\lesssim \frac{1}{\eta(t)\lambda_{2}(t)}\in L^{1}\left(\left\{t\gg 1\right\}\right),
\end{split}
\]
whenever $p_{2}>p_{1}$, which is just \eqref{e6}.

\medskip

Second, we bound  $\Xi_{2}(t)$ from \eqref{ee3}. Following a similar analysis as in the 2D case \eqref{eq1},
\begin{equation}\label{ee2}
\begin{split}
|\Xi_{2}(t)|&\leq\left|\frac{\eta'(t)}{\eta^{2}(t)}\int_{\mathbb{R}^{2}}u\psi_{\sigma}\left(\frac{x}{\lambda_{1}(t)}\right)\phi_{\delta_{1}}\left(\frac{x}{\lambda_{2}(t)}\right)\phi_{\delta_{2}}\left(\frac{y}{\lambda_{3}(t)}\right)\phi_{\delta_{3}}\left(\frac{z}{\lambda_{4}(t)}\right)\,\mathrm{d}x\mathrm{d}y\right|\\
&\lesssim_{\sigma,\delta_{1},\delta_{2}}\|u_{0}\|_{L^{2}}\frac{\left(\lambda_{2}(t)\lambda_{3}(t)\lambda_{4}(t)\right)^{1/2}|\eta'(t)|}{\eta^{2}(t)}\in L^{1}\left(\left\{t\gg 1\right\}\right),
\end{split}
\end{equation}
which is valid if $2p_1+p_2+p_3+p_4<2$, that is, \eqref{e4}. Next, from  \eqref{ee3},
\[
\begin{split}
&\Xi_{3}(t)\\
&=-\frac{\lambda_{1}'(t)}{\eta(t)\lambda_{1}(t)}\int_{\mathbb{R}^{3}}u\phi_{\sigma}\left(\frac{x}{\lambda_{1}(t)}\right)\left(\frac{x}{\lambda_{1}(t)}\right)\phi_{\delta_{1}}\left(\frac{x}{\lambda_{2}(t)}\right)\phi_{\delta_{2}}\left(\frac{y}{\lambda_{3}(t)}\right)\phi_{\delta_{3}}\left(\frac{z}{\lambda_{4}(t)}\right)\mathrm{d}x\mathrm{d}y\mathrm{d}z\\
&\quad -\frac{\lambda_{2}'(t)}{\eta(t)\lambda_{2}(t)}\int_{\mathbb{R}^{3}}u\psi_{\sigma}\left(\frac{x}{\lambda_{1}(t)}\right)\phi_{\delta_{1}}'\left(\frac{x}{\lambda_{2}(t)}\right)\left(\frac{x}{\lambda_{2}(t)}\right)\phi_{\delta_{2}}\left(\frac{y}{\lambda_{3}(t)}\right)\phi_{\delta_{3}}\left(\frac{z}{\lambda_{4}(t)}\right)\mathrm{d}x\mathrm{d}y\mathrm{d}z\\
&\quad -\frac{\lambda_{3}'(t)}{\eta(t)\lambda_{3}(t)}\int_{\mathbb{R}^{3}}u\psi_{\sigma}\left(\frac{x}{\lambda_{1}(t)}\right)\phi_{\delta_{1}}\left(\frac{x}{\lambda_{2}(t)}\right)\phi_{\delta_{2}}'\left(\frac{y}{\lambda_{3}(t)}\right)\left(\frac{y}{\lambda_{3}(t)}\right)\phi_{\delta_{3}}\left(\frac{z}{\lambda_{4}(t)}\right)\mathrm{d}x\mathrm{d}y\mathrm{d}z\\
&\quad -\frac{\lambda_{4}'(t)}{\eta(t)\lambda_{4}(t)}\int_{\mathbb{R}^{3}}u\psi_{\sigma}\left(\frac{x}{\lambda_{1}(t)}\right)\phi_{\delta_{1}}\left(\frac{x}{\lambda_{2}(t)}\right)\phi_{\delta_{2}}\left(\frac{y}{\lambda_{3}(t)}\right)\phi_{\delta_{3}}'\left(\frac{z}{\lambda_{4}(t)}\right)\left(\frac{z}{\lambda_{4}(t)}\right)\mathrm{d}x\mathrm{d}y\mathrm{d}z\\
&=\Xi_{3,1}(t)+\Xi_{3,2}(t)+\Xi_{3,3}(t)+\Xi_{3,4}(t).
\end{split}
\]
Concerning the term $\Xi_{3,1}$, we have for all $\epsilon>0,$
\[
\begin{split}
&|\Xi_{3,1}(t)|\\
&\leq \left|\frac{\lambda_{1}'(t)}{4\eta(t)\epsilon\lambda_{1}(t)}\right|\int_{\mathbb{R}^{3}}u^{2}\phi_{\sigma}\left(\frac{x}{\lambda_{1}(t)}\right)\phi_{\delta_{1}}\left(\frac{x}{\lambda_{2}(t)}\right)\phi_{\delta_{2}}\left(\frac{y}{\lambda_{3}(t)}\right)\phi_{\delta_{3}}\left(\frac{z}{\lambda_{4}(t)}\right)\mathrm{d}x\mathrm{d}y\mathrm{d}z\\
&\quad +\left|\frac{\epsilon\lambda_{1}'(t)}{\eta(t)\lambda_{1}(t)}\right| \int_{\mathbb{R}^{3}} \phi_{\sigma}\left(\frac{x}{\lambda_{1}(t)}\right)\left(\frac{x}{\lambda_{1}(t)}\right)^2\phi_{\delta_{1}}\left(\frac{x}{\lambda_{2}(t)}\right)\phi_{\delta_{2}}\left(\frac{y}{\lambda_{3}(t)}\right)\phi_{\delta_{3}}\left(\frac{z}{\lambda_{4}(t)}\right)\mathrm{d}x\mathrm{d}y\mathrm{d}z.
\end{split}
\]
Hence we get, after  choosing $\epsilon=|\lambda_{1}'(t)|>0,$
\[
\begin{split}
|\Xi_{3,1}(t)|&\lesssim \frac{1}{4\eta(t)\lambda_{1}(t)}\int_{\mathbb{R}^{3}}u^{2}\phi_{\sigma}\left(\frac{x}{\lambda_{1}(t)}\right)\phi_{\delta_{1}}\left(\frac{x}{\lambda_{2}(t)}\right)\phi_{\delta_{2}}\left(\frac{y}{\lambda_{3}(t)}\right)\phi_{\delta_{3}}\left(\frac{z}{\lambda_{4}(t)}\right)\mathrm{d}x\mathrm{d}y\mathrm{d}z\\
&\quad+\frac{\left(\lambda_{1}'(t)\right)^{2}\lambda_{3}(t)\lambda_{4}(t)}{\eta(t)}.
\end{split}
\]
Note that the first term in the r.hs. above is half the quantity $\Xi_{1,1,2}(t)$ to be estimated, see \eqref{eee1}. Therefore, that term is absorbed properly. Instead, the  second quantity  satisfies
\[
\frac{\left(\lambda_{1}'(t)\right)^{2}\lambda_{3}(t)\lambda_{4}(t)}{\eta(t)}\in L^{1}\left(\{t\gg1 \}\right),
\]
whenever  $3p_{1}+(p_{3}+p_{4})<2,$ giving by \eqref{e16}. Next,  we have
\[
|\Xi_{3,2}(t)|\lesssim \left|\frac{\lambda_{2}'(t)\left(\lambda_{3}(t)\lambda_{4}(t)\right)^{1/2}}{\eta(t)(\lambda_{2}(t))^{1/2}}\right|\in L^{1}\left(\{t\gg1 \}\right),
\]
for $2p_{1}+p_{2}+p_{3}+p_{4}<2$, which is \eqref{e4}. Next, for $\Xi_{3,3}$    we obtain the bound
\[
\begin{split}
|\Xi_{3,3}(t)|&\lesssim \frac{\lambda_{3}'(t)\left(\lambda_{2}(t)\lambda_{4}(t)\right)^{1/2}}{\eta(t)(\lambda_{3}(t))^{1/2}}\in L^{1}\left(\{t\gg1 \}\right),
\end{split}
\]
 since $2p_{1}+p_{2}+p_{3}+p_{4}<2$ again.  Finally, we bound   $\Xi_{3,4}$ as follows:
 \[
 |\Xi_{3,4}(t)|\lesssim \frac{\lambda_{4}'(t)\left(\lambda_{2}(t)\lambda_{3}(t)\right)^{1/2}}{\eta(t)(\lambda_{4}(t))^{1/2}}\in L^{1}\left(\{t\gg1 \}\right),
 \]
once again by \eqref{e4}. We conclude 
\begin{equation}\label{main2.2}
\begin{split}
&\frac{1}{\eta(t)\lambda_{1}(t)}\int_{\mathbb{R}^{3}} \frac{u^{2}}{2} \phi_{\sigma}\left(\frac{x}{\lambda_{1}(t)}\right)\phi_{\delta_{1}}\left(\frac{x}{\lambda_{2}(t)}\right)\phi_{\delta_{2}}\left(\frac{y}{\lambda_{3}(t)}\right)\phi_{\delta_{3}}\left(\frac{z}{\lambda_{4}(t)}\right)\mathrm{d}x\mathrm{d}y\mathrm{d}z  \\
&= \Xi_{1,1,2}(t)\\
& \leq 2\frac{d\Xi(t)}{dt} + \Xi_{aux}(t),\\
\end{split}
\end{equation}
where $\Xi_{aux}(t) \in L^{1}\left(\{t\gg1 \}\right)$, with
\[
|\Xi_{aux}(t)| \lesssim \frac{1}{t^{1+\kappa_0} \log^{\epsilon_0}(t)},\quad \kappa_0,\epsilon_0>0.
\]
After this, we conclude essentially in the same form as in the 2D case.

\subsection{$L^2$ virial and $\dot H^1$ local decay} We follow the lines of Section \ref{Sect:4}, devoted this time to the 3D case. Some additional estimates are needed, and some care will be put in some particular parts of the proof.

\medskip

As in the 2D case, we consider the functional
\begin{equation*}
\mathcal{Q}(t):=\frac{1}{\eta(t)}\int_{\mathbb{R}^{3}} u^{2}(x,y,z,t)\psi_{\sigma'}\left(\frac{x}{\lambda_{1}(t)}\right)\phi_{\delta_{2}}\left(\frac{y}{\lambda_{3}(t)}\right)\phi_{\delta_{3}}\left(\frac{z}{\lambda_{4}(t)}\right)\mathrm{d}x\,\mathrm{d}y\,\mathrm{d}z,
\end{equation*}
that is clearly well defined for  solutions of the IVP  \eqref{ZK:Eq} with $d=3.$ Moreover, it is a bounded in time functional.

\medskip

In what follows, we consider the evolution of $\mathcal{Q}(t):$
\begin{equation}\label{main2.1}
\begin{split}
\frac{\mathrm{d}}{\mathrm{d}t}\mathcal{Q}(t)&=\underbrace{\frac{2}{\eta(t)}\int_{\mathbb{R}^{3}}u\partial_{t}u\psi_{\sigma'}\left(\frac{x}{\lambda_{1}(t)}\right)\phi_{\delta_{2}}\left(\frac{y}{\lambda_{3}(t)}\right)\phi_{\delta_{3}}\left(\frac{z}{\lambda_{4}(t)}\right)\,\mathrm{d}x\,\mathrm{d}y\,\mathrm{d}z}_{A_{1}(t)}\\
&\quad \underbrace{-\frac{\lambda_{3}'(t)}{\lambda_{3}(t)\eta(t)}\int_{\mathbb{R}^{3}}u^{2}\psi_{\sigma'}\left(\frac{x}{\lambda_{1}(t)}\right)\phi_{\delta_{2}}'\left(\frac{y}{\lambda_{3}(t)}\right)\left(\frac{y}{\lambda_{3}(t)}\right)\phi_{\delta_{3}}\left(\frac{z}{\lambda_{4}(t)}\right)\,\mathrm{d}x\,\mathrm{d}y\,\mathrm{d}z}_{A_{2}(t)}\\
&\quad \underbrace{-\frac{\lambda_{1}'(t)}{
		\lambda_{1}(t)\eta(t)}\int_{\mathbb{R}^{3}}u^{2}\phi_{\sigma'}\left(\frac{x}{\lambda_{1}(t)}\right)\left(\frac{x}{\lambda_{1}(t)}\right)\phi_{\delta_{2}}\left(\frac{y}{\lambda_{3}(t)}\right)\phi_{\delta_{3}}\left(\frac{z}{\lambda_{4}(t)}\right)\,\mathrm{d}x\,\mathrm{d}y\,\mathrm{d}z}_{A_{3}(t)}\\
	&\quad \underbrace{-\frac{\lambda_{4}'(t)}{\lambda_{4}(t)\eta(t)}\int_{\mathbb{R}^{3}}u^{2}\psi_{\sigma'}\left(\frac{x}{\lambda_{1}(t)}\right)\phi_{\delta_{2}}\left(\frac{y}{\lambda_{3}(t)}\right)\phi_{\delta_{3}}'\left(\frac{z}{\lambda_{4}(t)}\right)\left(\frac{z}{\lambda_{4}(t)}\right)\,\mathrm{d}x\,\mathrm{d}y\,\mathrm{d}z}_{A_{4}(t)}\\
&\quad \underbrace{-\frac{\eta'(t)}{\eta^{2}(t)}\int_{\mathbb{R}^{3}}u^{2}\psi_{\sigma'}\left(\frac{x}{\lambda_{1}(t)}\right)\phi_{\delta_{2}}\left(\frac{y}{\lambda_{3}(t)}\right)\phi_{\delta_{3}}\left(\frac{z}{\lambda_{4}(t)}\right)\,\mathrm{d}x\,\mathrm{d}y\,\mathrm{d}z}_{A_{5}(t)}.
\end{split}
\end{equation}
Compared with the 2D case in \eqref{main2}, the term $A_4(t)$ is a new contribution. First,
\begin{equation*}
\begin{split}
&A_{1}(t)\\
&=\frac{2}{\eta(t)}\int_{\mathbb{R}^{3}}\partial_{x}u\Delta u\psi_{\sigma'}\left(\frac{x}{\lambda_{1}(t)}\right)\phi_{\delta_{2}}\left(\frac{y}{\lambda_{3}(t)}\right)\phi_{\delta_{3}}\left(\frac{z}{\lambda_{4}(t)}\right)\,\mathrm{d}x\mathrm{d}y\mathrm{d}z\\
&\quad + \frac{1}{\eta(t)}\int_{\mathbb{R}^{3}}\partial_{x}u u^{2}\psi_{\sigma'}\left(\frac{x}{\lambda_{1}(t)}\right)\phi_{\delta_{2}}\left(\frac{y}{\lambda_{3}(t)}\right)\phi_{\delta_{3}}\left(\frac{z}{\lambda_{4}(t)}\right)\,\mathrm{d}x\mathrm{d}y\mathrm{d}z\\
&\quad +\frac{2}{\eta(t)\lambda_{1}(t)}\int_{\mathbb{R}^{3}}u\left(\Delta u+\frac{u^{2}}{2}\right)\phi_{\sigma'}\left(\frac{x}{\lambda_{1}(t)}\right)\phi_{\delta_{2}}\left(\frac{y}{\lambda_{3}(t)}\right)\phi_{\delta_{3}}\left(\frac{z}{\lambda_{4}(t)}\right)\,\mathrm{d}x\mathrm{d}y\mathrm{d}z\\
&= A_{1,1}(t)+A_{1,2}(t)+A_{1,3}(t).
\end{split}
\end{equation*}
The term $A_{1,2}(t)$ will produce a local cubic term which will be the most difficult one to be controlled. For the moment, we concentrate ourselves in the term $A_{1,1}(t)$. We have
\begin{equation*}
\begin{split}
A_{1,1}(t)&= -\frac{1}{\eta(t)\lambda_{1}(t)}\int_{\mathbb{R}^{3}}\ \left(\partial_{x}u\right)^{2} \phi_{\sigma'}\left(\frac{x}{\lambda_{1}(t)}\right)\phi_{\delta_{2}}\left(\frac{y}{\lambda_{3}(t)}\right)\phi_{\delta_{3}}\left(\frac{z}{\lambda_{4}(t)}\right)\,\mathrm{d}x\mathrm{d}y\mathrm{d}z\\
&\, -\frac{2}{\eta(t)\lambda_{3}(t)}\int_{\mathbb{R}^{3}}\partial_{x}u\partial_{y}u\psi_{\sigma'}\left(\frac{x}{\lambda_{1}(t)}\right)\phi_{\delta_{2}}'\left(\frac{y}{\lambda_{3}(t)}\right)\phi_{\delta_{3}}\left(\frac{z}{\lambda_{4}(t)}\right)\,\mathrm{d}x\mathrm{d}y\mathrm{d}z\\
&\, +\frac{1}{\eta(t)\lambda_{1}(t)}\int_{\mathbb{R}^{3}}\ \left(\partial_{y}u\right)^{2} \phi_{\sigma'}\left(\frac{x}{\lambda_{1}(t)}\right)\phi_{\delta_{2}}\left(\frac{y}{\lambda_{3}(t)}\right)\phi_{\delta_{3}}\left(\frac{z}{\lambda_{4}(t)}\right)\,\mathrm{d}x\mathrm{d}y\mathrm{d}z\\
&\,+\frac{1}{\eta(t)\lambda_{1}(t)}\int_{\mathbb{R}^{3}}\ \left(\partial_{z}u\right)^{2} \phi_{\sigma'}\left(\frac{x}{\lambda_{1}(t)}\right)\phi_{\delta_{2}}\left(\frac{y}{\lambda_{3}(t)}\right)\phi_{\delta_{3}}\left(\frac{z}{\lambda_{4}(t)}\right)\,\mathrm{d}x\mathrm{d}y\mathrm{d}z \\
&\, -\frac{2}{\eta(t)\lambda_{4}(t)}\int_{\mathbb{R}^{3}}\ \partial_{x}u\partial_{z}u \psi_{\sigma'}\left(\frac{x}{\lambda_{1}(t)}\right)\phi_{\delta_{2}}\left(\frac{y}{\lambda_{3}(t)}\right)\phi_{\delta_{3}}'\left(\frac{z}{\lambda_{4}(t)}\right)\,\mathrm{d}x\mathrm{d}y\mathrm{d}z\\
&=A_{1,1,1}(t)+A_{1,1,2}(t)+A_{1,1,3}(t)+ A_{1,1,4}(t)+A_{1,1,5}(t).
\end{split}
\end{equation*}
We have $A_{1,1,5}(t)$ as the new contribution in 3D, but the reminders also need some care because one does not obtain the same conditions on the parameters as in the previous subsection.

\medskip

From this expression we will only focus on estimate $A_{1,1,2}$ and $A_{1,1,5},$   since $A_{1,1,1}, A_{1,1,3}$ and $A_{1,1,4}$ are part of the quantities to be estimated. So that,
\begin{equation*}
\begin{split}
A_{1,1,2}(t)&\lesssim \|u\|_{H^{1}_{xyz}}^{2}\frac{1}{\eta(t)\lambda_{3}(t)}\in L^{1}\left(\{t\gg 1\}\right),
\end{split}
\end{equation*}
whenever $ p_{3}+r_{1}= p_3 +1-p_1>1$, valid thanks to \eqref{new1}.  Also,
\begin{equation*}
\begin{split}
A_{1,1,5}(t)&\lesssim \|u\|_{H^{1}_{xyz}}^{2}\frac{1}{\eta(t)\lambda_{4}(t)}\in L^{1}\left(\{t\gg 1\}\right),
\end{split}
\end{equation*}
whenever $p_{4}+1-p_1>1$, which is \eqref{new2}. Next,
\begin{equation*}
\begin{split}
A_{1,3}(t) &= -\frac{2}{\eta(t)\lambda_{1}(t)}\int_{\mathbb{R}^{3}}\left(\partial_{x}u\right)^{2}\phi_{\sigma'}\left(\frac{x}{\lambda_{1}(t)}\right)\phi_{\delta_{2}}\left(\frac{y}{\lambda_{3}(t)}\right)\phi_{\delta_{3}}\left(\frac{z}{\lambda_{4}(t)}\right)\,\mathrm{d}x\,\mathrm{d}y\,\mathrm{d}z\\
&\quad  +\frac{1}{\eta(t)\lambda_{1}^{3}(t)}\int_{\mathbb{R}^{3}}u^{2}\phi_{\sigma'}''\left(\frac{x}{\lambda_{1}(t)}\right)\phi_{\delta_{2}}\left(\frac{y}{\lambda_{3}(t)}\right)\phi_{\delta_{3}}\left(\frac{z}{\lambda_{4}(t)}\right)\,\mathrm{d}x\,\mathrm{d}y\,\mathrm{d}z\\
&\quad  -\frac{2}{\eta(t)\lambda_{1}(t)}\int_{\mathbb{R}^{3}}\left(\partial_{y}u\right)^{2}\phi_{\sigma'}\left(\frac{x}{\lambda_{1}(t)}\right)\phi_{\delta_{2}}\left(\frac{y}{\lambda_{3}(t)}\right)\phi_{\delta_{3}}\left(\frac{z}{\lambda_{4}(t)}\right)\,\mathrm{d}x\,\mathrm{d}y\,\mathrm{d}z\\
&\quad + \frac{1}{\eta(t)\lambda_{1}(t)\lambda_{3}^{2}(t)}\int_{\mathbb{R}^{3}}u^{2}\phi_{\sigma'}\left(\frac{x}{\lambda_{1}(t)}\right)\phi_{\delta_{2}}''\left(\frac{y}{\lambda_{3}(t)}\right)\phi_{\delta_{3}}\left(\frac{z}{\lambda_{4}(t)}\right)\,\mathrm{d}x\,\mathrm{d}y\,\mathrm{d}z\\
&\quad - \frac{2}{\eta(t)\lambda_{1}(t)}\int_{\mathbb{R}^{3}}\left(\partial_{z}u\right)^{2}\phi_{\sigma'}\left(\frac{x}{\lambda_{1}(t)}\right)\phi_{\delta_{2}}\left(\frac{y}{\lambda_{3}(t)}\right)\phi_{\delta_{3}}\left(\frac{z}{\lambda_{4}(t)}\right)\,\mathrm{d}x\,\mathrm{d}y\,\mathrm{d}z\\
&\quad +\frac{1}{\eta(t)\lambda_{1}(t)\lambda_{4}^{2}(t)}\int_{\mathbb{R}^{3}}u^{2}\phi_{\sigma'}\left(\frac{x}{\lambda_{1}(t)}\right)\phi_{\delta_{2}}\left(\frac{y}{\lambda_{3}(t)}\right)\phi_{\delta_{3}}''\left(\frac{z}{\lambda_{4}(t)}\right)\,\mathrm{d}x\,\mathrm{d}y\,\mathrm{d}z\\
&\quad + \frac{1}{\eta(t)\lambda_{1}(t)}\int_{\mathbb{R}^{3}}u^{3}\phi_{\sigma'}\left(\frac{x}{\lambda_{1}(t)}\right)\phi_{\delta_{2}}\left(\frac{y}{\lambda_{3}(t)}\right)\phi_{\delta_{3}}\left(\frac{z}{\lambda_{4}(t)}\right)\,\mathrm{d}x\,\mathrm{d}y\,\mathrm{d}z\\
&= A_{1,3,1}(t)+A_{1,3,2}(t)+A_{1,3,3}(t)+A_{1,3,4}(t)+A_{1,3,5}(t)+A_{1,3,6}(t)  +A_{1,3,7}(t).
\end{split}
\end{equation*}
We will provide upper bounds for the terms $A_{1,3,2},A_{1,3,4}$, and $A_{1,3,6}$ since the terms $A_{1,3,1},A_{1,3,3}$ and $A_{1,3,5}$ are part of the quantities to be estimated.  The term $A_{1,3,7}$  will be described   below  since it represents the most harder term to estimate at this step.

\medskip

\noindent
First, similar to the 2D case,
\[
|A_{1,3,2}(t)|\lesssim \frac{1}{\eta(t)\lambda_{1}^{3}(t)}\in L^{1}\left(\{t\gg 1\}\right),
\]
since $p_{1}>0.$ Second,
\[
|A_{1,3,4}(t)|\lesssim \frac{1}{\eta(t)\lambda_{1}(t)\lambda_{3}^{2}(t)}\in L^{1}\left(\{t\gg 1\}\right),
\]
since $p_{3}>0.$ Third,
\[
|A_{1,3,6}(t)| \lesssim \frac{1}{\eta(t)\lambda_{1}(t)\lambda_{4}^{2}(t)}\in L^{1}\left(\{t\gg 1\}\right),
\]
since $p_{4}>0.$ Only the terms $A_{1,2}$ and $A_{1,3,7}$ remain to be bounded, but both are similar in nature and it is only necessary to estimate one of them.

\medskip

Now, we come back to \eqref{main2.1}. Concerning $A_{2}, A_{3}$, $A_{4}$ and $A_5$, we have in the first place the following  bounds: for $A_{2}$  it verifies that
\begin{equation*}
\begin{split}
|A_{2}(t)| &\lesssim \left|\frac{\lambda_{3}'(t)}{\eta(t)\lambda_{3}(t)}\right| \lesssim \frac{1}{t\eta(t)}\in L^{1}\left(\{t\gg 1\}\right),
\end{split}
\end{equation*}
since $r_{1}>0.$ Also, for $A_{3}$ a quite similar  results holds, that is,
\begin{equation*}
\begin{split}
| A_{3}(t)| &\lesssim \left|\frac{\lambda_{1}'(t)}{\eta(t)\lambda_{1}(t)}\right| \lesssim \frac{1}{t^{1+r_{1}}\ln ^{r_{2}}t}\in L^{1}\left(\{t\gg 1\}\right),
\end{split}
\end{equation*}
since $r_{1}>0.$ Next, for $A_{4}$ and $A_5$, the bound is exactly the same, so we skip it. This ends the estimation of all terms, except $A_{1,2}$ and $A_{1,3,7}$ which contain cubic powers.

\medskip

Finally, we  show how to handle the terms with a cubic power. We proceed as follows: let $\epsilon>0$ to be fixed later on. Write $|u|^3=(\epsilon^{-1/4} |u|^{5/2}) (\epsilon^{1/4} |u|^{1/2})$. Using Young's inequality with $p=4/3$ and $p'=4$, it holds
\begin{equation}\label{auxiliar_1}
\begin{split} &\int_{\mathbb{R}^{3}}|u|^{3}\phi_{\sigma'}\left(\frac{x}{\lambda_{1}(t)}\right)\phi_{\delta_{2}}\left(\frac{y}{\lambda_{3}(t)}\right)\phi_{\delta_{3}}\left(\frac{z}{\lambda_{4}(t)}\right)\,\mathrm{d}x\,\mathrm{d}y\,\mathrm{d}z\\
&\leq \frac{3}{4\epsilon ^{1/3}} \int_{\mathbb{R}^{3}}|u|^{10/3}\phi_{\sigma'}\left(\frac{x}{\lambda_{1}(t)}\right)\phi_{\delta_{2}}\left(\frac{y}{\lambda_{3}(t)}\right)\phi_{\delta_{3}}\left(\frac{z}{\lambda_{4}(t)}\right)\,\mathrm{d}x\,\mathrm{d}y\,\mathrm{d}z\\
&\quad +  \frac{\epsilon}{4}\int_{\mathbb{R}^{3}}u^{2}\phi_{\sigma'}\left(\frac{x}{\lambda_{1}(t)}\right)\phi_{\delta_{2}}\left(\frac{y}{\lambda_{3}(t)}\right)\phi_{\delta_{3}}\left(\frac{z}{\lambda_{4}(t)}\right)\,\mathrm{d}x\,\mathrm{d}y\,\mathrm{d}z.
\end{split}
\end{equation}
From  now on we will focus our attention to provide an adequate upper bound for the first term in the last line of \eqref{auxiliar_1}. In this sense, following the same procedure as in \eqref{chi_chi},
\begin{equation*}
\begin{split}
& \int_{\mathbb{R}^{3}}|u|^{\frac{10}{3}}\phi_{\sigma'}\left(\frac{x}{\lambda_{1}(t)}\right)\phi_{\delta_{2}}\left(\frac{y}{\lambda_{3}(t)}\right)\phi_{\delta_{3}}\left(\frac{z}{\lambda_{4}(t)}\right)\,\mathrm{d}x\,\mathrm{d}y\,\mathrm{d}z\\
&=\sum_{(m_{1},m_{2},m_{3})\in\mathbb{Z}^{3}}\int_{m_{3}}^{m_{3}+1}\int_{m_{2}}^{m_{2}+1}\int_{m_{1}}^{m_{1}+1}|u|^{\frac{10}{3}}\phi_{\sigma'}\left(\frac{x}{\lambda_{1}(t)}\right)\phi_{\delta_{2}}\left(\frac{y}{\lambda_{3}(t)}\right)\phi_{\delta_{3}}\left(\frac{z}{\lambda_{4}(t)}\right)\,\mathrm{d}x\,\mathrm{d}y\,\mathrm{d}z\\
&\leq \sum_{(m_{1},m_{2},m_{3})\in\mathbb{Z}^{3}}\int_{\mathbb{R}^{3}}|u|^{\frac{10}{3}}\chi_{m_{1}}(x)\chi_{m_{2}}(y)\chi_{m_{3}}(z)\phi_{\sigma'}\left(\frac{x}{\lambda_{1}(t)}\right)\phi_{\delta_{2}}\left(\frac{y}{\lambda_{3}(t)}\right)\phi_{\delta_{3}}\left(\frac{z}{\lambda_{4}(t)}\right)\,\mathrm{d}x\,\mathrm{d}y\,\mathrm{d}z\\
&\leq \sum_{(m_{1},m_{2},m_{3})\in\mathbb{Z}^{3}}\int_{\mathbb{R}^{3}}\left(|u|\chi_{m_{1}}(x)\chi_{m_{2}}(y)\chi_{m_{3}}(z)\right)^{\frac{10}{3}}\phi_{\sigma'}\left(\frac{x}{\lambda_{1}(t)}\right)\phi_{\delta_{2}}\left(\frac{y}{\lambda_{3}(t)}\right)\phi_{\delta_{3}}\left(\frac{z}{\lambda_{4}(t)}\right)\,\mathrm{d}x\,\mathrm{d}y\,\mathrm{d}z\\
&\leq \sum_{(m_{1},m_{2},m_{3})\in\mathbb{Z}^{3}}\left\|u\chi_{m_{1}}\chi_{m_{2}}\chi_{m_{3}}\right\|_{L^{\frac{10}{3}}_{xyz}}^{\frac{10}{3}}\Lambda_{m_{1}}\Lambda_{m_{2}}\Lambda_{m_{3}},
\end{split}
\end{equation*}
where
\begin{equation*}
\Lambda_{m_{1}}:=\max_{x\in[m_{1},m_{1}+1]}\phi_{\sigma'}\left(\frac{x}{\lambda_{1}(t)}\right),\qquad \Lambda_{m_{2}}:=\max_{y\in[m_{2},m_{2}+1]}\phi_{\delta_{1}}\left(\frac{y}{\lambda_{3}(t)}\right),
\end{equation*}
and
\begin{equation*}
\Lambda_{m_{3}}:=\max_{z\in[m_{3},m_{3}+1]}\phi_{\delta_{3}}\left(\frac{z}{\lambda_{4}(t)}\right).
\end{equation*}
Now, by the Gagliardo-Nirenberg-Sobolev  inequality  in its optimal form, we get
\begin{equation}\label{nlp1}
\begin{split}
&\sum_{(m_{1},m_{2},m_{3})\in\mathbb{Z}^{3}}\left\|u\chi_{m_{1}}\chi_{m_{2}}\chi_{m_{3}}\right\|_{L^{\frac{10}{3}}_{xyz}}^{\frac{10}{3}}\Lambda_{m_{1}}\Lambda_{m_{2}}\Lambda_{m_{3}}\\
&\leq  \sum_{(m_{1},m_{2},m_{3})\in\mathbb{Z}^{3}}c^{\frac{10}{3}}_{\mathrm{opt}}\left\|\nabla\left(u\chi_{m_{1}}\chi_{m_{2}}\chi_{m_{3}}\right)\right\|_{L^{2}_{xyz}}^{2}\left\|u\chi_{m_{1}}\chi_{m_{2}}\chi_{m_{3}}\right\|_{L^{2}_{xyz}}^{\frac{4}{3}}\Lambda_{m_{1}}\Lambda_{m_{2}}\Lambda_{m_{3}}\\
&\leq \left\|u_{0}\right\|_{L^{2}_{xyz}}^{\frac{4}{3}} \sum_{(m_{1},m_{2},m_{3})\in\mathbb{Z}^{3}}c^{\frac{10}{3}}_{\mathrm{opt}}\left\|\nabla\left(u\chi_{m_{1}}\chi_{m_{2}}\chi_{m_{3}}\right)\right\|_{L^{2}_{xyz}}^{2}\Lambda_{m_{1}}\Lambda_{m_{2}}\Lambda_{m_{3}},
\end{split}
\end{equation}
where $c_{\mathrm{opt}}$ denotes the optimal constant for dimension $d=3$. Also,
\begin{equation*}
\begin{split}
\left\|\nabla\left(u\chi_{m_{1}}\chi_{m_{2}}\chi_{m_{3}}\right)\right\|_{L^{2}_{xyz}}&=\left\|\chi_{m_{1}}\chi_{m_{2}}\chi_{m_{3}}\nabla u+u\nabla(\chi_{m_{1}}\chi_{m_{2}}\chi_{m_{3}})\right\|_{L^{2}_{xyz}}\\
&\leq \left\|\chi_{m_{1}}\chi_{m_{2}}\chi_{m_{3}}\nabla u\right\|_{L^{2}_{xyz}}+\left\|u\nabla(\chi_{m_{1}}\chi_{m_{2}}\chi_{m_{3}})\right\|_{L^{2}_{xyz}}.
\end{split}
\end{equation*}
Aditionally, an analysis similar to the employed in \eqref{e1.1}-\eqref{e1.7} allow us to obtain the following bounds:
	\begin{equation*}
\Lambda_{m_{1}} \leq  C\min_{x\in[m_{1},m_{1}+1]}\phi_{\sigma'}\left(\frac{x}{\lambda_{1}(t)}\right),\quad\mbox{for all} \; x\in\mathbb{R},
\end{equation*}
	\begin{equation*}
\Lambda_{m_{2}}\leq C \min_{y\in[m_{2},m_{2}+1]}\phi_{\delta_{2}}\left(\frac{y}{\lambda_{3}(t)}\right),\quad\mbox{for all} \; y\in\mathbb{R},
\end{equation*}
and
	\begin{equation*}
\Lambda_{m_{3}}\leq  C \min_{z\in[m_{3},m_{3}+1]}\phi_{\delta_{3}}\left(\frac{z}{\lambda_{4}(t)}\right),\quad\mbox{for all} \; z\in\mathbb{R}.
\end{equation*}
Next, we go back to \eqref{nlp1}. Incorporating the previous computations,
\begin{equation*}
\begin{split}
&\left\|u_{0}\right\|_{L^{2}_{xyz}}^{\frac{4}{3}} \sum_{(m_{1},m_{2},m_{3})\in\mathbb{Z}^{3}}c^{\frac{10}{3}}_{\mathrm{opt}}\left\|\nabla\left(u\chi_{m_{1}}\chi_{m_{2}}\chi_{m_{3}}\right)\right\|_{L^{2}_{xyz}}^{2}\Lambda_{m_{1}}\Lambda_{m_{2}}\Lambda_{m_{3}}\\
&\leq 2\left\|u_{0}\right\|_{L^{2}_{xyz}}^{\frac{4}{3}} \sum_{(m_{1},m_{2},m_{3})\in\mathbb{Z}^{3}}c^{\frac{10}{3}}_{\mathrm{opt}}\left(           \left\|\chi_{m_{1}}\chi_{m_{2}}\chi_{m_{3}}\nabla u\right\|_{L^{2}_{xyz}}^2 + \left\|u\nabla(\chi_{m_{1}}\chi_{m_{2}}\chi_{m_{3}})\right\|_{L^{2}_{xyz}}^2 \right)\Lambda_{m_{1}}\Lambda_{m_{2}}\Lambda_{m_{3}}\\
&\leq \widetilde{c} \sum_{(m_{1},m_{2},m_{3})\in\mathbb{Z}^{3}}  \int_{\mathbb{R}^{3}}\left|\nabla u\right|^{2}\chi_{m_{1}}^{2}\chi_{m_{2}}^{2}\chi_{m_{3}}^{2}\phi_{\sigma'}\left(\frac{x}{\lambda_{1}(t)}\right)\phi_{\delta_{2}}\left(\frac{y}{\lambda_{3}(t)}\right)\phi_{\delta_{3}}\left(\frac{z}{\lambda_{4}(t)}\right)\,\mathrm{d}x\mathrm{d}y\mathrm{d}z\\
&\quad + \widetilde{c} \sum_{(m_{1},m_{2},m_{3})\in\mathbb{Z}^{3}} \int_{\mathbb{R}^{3}} u^{2}\left|\nabla \left(\chi_{m_{1}}\chi_{m_{2}}\chi_{m_{3}}\right)\right|^{2}\phi_{\sigma'}\left(\frac{x}{\lambda_{1}(t)}\right)\phi_{\delta_{2}}\left(\frac{y}{\lambda_{3}(t)}\right)\phi_{\delta_{3}}\left(\frac{z}{\lambda_{4}(t)}\right)\,\mathrm{d}x\mathrm{d}y\mathrm{d}z\\
&\leq  \widetilde{c}   \int_{\mathbb{R}^{3}}\left|\nabla u\right|^{2}\phi_{\sigma'}\left(\frac{x}{\lambda_{1}(t)}\right)\phi_{\delta_{2}}\left(\frac{y}{\lambda_{3}(t)}\right)\phi_{\delta_{3}}\left(\frac{z}{\lambda_{4}(t)}\right)\,\mathrm{d}x\mathrm{d}y\mathrm{d}z\\
&\quad + \widetilde{c} \sum_{(m_{1},m_{2},m_{3})\in\mathbb{Z}^{3}} \int_{\mathbb{R}^{3}} u^{2}\left|\nabla \left(\chi_{m_{1}}\chi_{m_{2}}\chi_{m_{3}}\right)\right|^{2}\phi_{\sigma'}\left(\frac{x}{\lambda_{1}(t)}\right)\phi_{\delta_{2}}\left(\frac{y}{\lambda_{3}(t)}\right)\phi_{\delta_{3}}\left(\frac{z}{\lambda_{4}(t)}\right)\,\mathrm{d}x\mathrm{d}y\mathrm{d}z.
\end{split}
\end{equation*}
%
%
To estimate the remainder $L^2$ term above we  consider a  smooth  function $\rho$  satisfying: $\rho\equiv 1$ on $[-1,2]$ and $\rho\equiv 0$  on $(-\infty,-2]\cup[3,\infty).$ Then, for $m_{1},m_{2},m_{3}\in\mathbb{Z},$ we set $\rho_{m_{1}}(x):=\rho(x-m_{1}),\,\,\rho_{m_{2}}(y):=\rho(y-m_{2}) ,$ and $\rho_{m_{3}}(z):=\rho(z-m_{3}).$

\medskip

Hence,
\begin{equation*}
\begin{split}
&\widetilde{c} \sum_{(m_{1},m_{2},m_{3})\in\mathbb{Z}^{3}} \int_{\mathbb{R}^{3}} u^{2}\left|\nabla \left(\chi_{m_{1}}\chi_{m_{2}}\chi_{m_{3}}\right)\right|^{2}\phi_{\sigma'}\left(\frac{x}{\lambda_{1}(t)}\right)\phi_{\delta_{2}}\left(\frac{y}{\lambda_{3}(t)}\right)\phi_{\delta_{3}}\left(\frac{z}{\lambda_{4}(t)}\right)\,\mathrm{d}x\mathrm{d}y\mathrm{d}z\\
&=\widetilde{c} \sum_{(m_{1},m_{2},m_{3})\in\mathbb{Z}^{3}} \int_{\mathbb{R}^{3}} u^{2}\left( \chi_{m_{1}}'\chi_{m_{2}}\chi_{m_{3}}\right)^{2}\phi_{\sigma'}\left(\frac{x}{\lambda_{1}(t)}\right)\phi_{\delta_{2}}\left(\frac{y}{\lambda_{3}(t)}\right)\phi_{\delta_{3}}\left(\frac{z}{\lambda_{4}(t)}\right)\,\mathrm{d}x\mathrm{d}y\mathrm{d}z\\
&\quad +\widetilde{c} \sum_{(m_{1},m_{2},m_{3})\in\mathbb{Z}^{3}} \int_{\mathbb{R}^{3}} u^{2}\left( \chi_{m_{1}}\chi_{m_{2}}'\chi_{m_{3}}\right)^{2}\phi_{\sigma'}\left(\frac{x}{\lambda_{1}(t)}\right)\phi_{\delta_{2}}\left(\frac{y}{\lambda_{3}(t)}\right)\phi_{\delta_{3}}\left(\frac{z}{\lambda_{4}(t)}\right)\,\mathrm{d}x\mathrm{d}y\mathrm{d}z\\
&\quad +\widetilde{c} \sum_{(m_{1},m_{2},m_{3})\in\mathbb{Z}^{3}} \int_{\mathbb{R}^{3}} u^{2}\left( \chi_{m_{1}}\chi_{m_{2}}\chi_{m_{3}}'\right)^{2}\phi_{\sigma'}\left(\frac{x}{\lambda_{1}(t)}\right)\phi_{\delta_{2}}\left(\frac{y}{\lambda_{3}(t)}\right)\phi_{\delta_{3}}\left(\frac{z}{\lambda_{4}(t)}\right)\,\mathrm{d}x\mathrm{d}y\mathrm{d}z\\
&\lesssim \widetilde{c} \sum_{(m_{1},m_{2},m_{3})\in\mathbb{Z}^{3}} \int_{\mathbb{R}^{3}} u^{2}\left( \rho_{m_{1}}\chi_{m_{2}}\chi_{m_{3}}\right)^{2}\phi_{\sigma'}\left(\frac{x}{\lambda_{1}(t)}\right)\phi_{\delta_{2}}\left(\frac{y}{\lambda_{3}(t)}\right)\phi_{\delta_{3}}\left(\frac{z}{\lambda_{4}(t)}\right)\,\mathrm{d}x\mathrm{d}y\mathrm{d}z\\
&\quad +\widetilde{c} \sum_{(m_{1},m_{2},m_{3})\in\mathbb{Z}^{3}} \int_{\mathbb{R}^{3}} u^{2}\left( \chi_{m_{1}}\rho_{m_{2}}\chi_{m_{3}}\right)^{2}\phi_{\sigma'}\left(\frac{x}{\lambda_{1}(t)}\right)\phi_{\delta_{2}}\left(\frac{y}{\lambda_{3}(t)}\right)\phi_{\delta_{3}}\left(\frac{z}{\lambda_{4}(t)}\right)\,\mathrm{d}x\mathrm{d}y\mathrm{d}z\\
&\quad +\widetilde{c} \sum_{(m_{1},m_{2},m_{3})\in\mathbb{Z}^{3}} \int_{\mathbb{R}^{3}} u^{2}\left( \chi_{m_{1}}\chi_{m_{2}}\rho_{m_{3}}\right)^{2}\phi_{\sigma'}\left(\frac{x}{\lambda_{1}(t)}\right)\phi_{\delta_{2}}\left(\frac{y}{\lambda_{3}(t)}\right)\phi_{\delta_{3}}\left(\frac{z}{\lambda_{4}(t)}\right)\,\mathrm{d}x\mathrm{d}y\mathrm{d}z\\
&\lesssim\widetilde{c} \int_{\mathbb{R}^{3}} u^{2}\phi_{\sigma'}\left(\frac{x}{\lambda_{1}(t)}\right)\phi_{\delta_{2}}\left(\frac{y}{\lambda_{3}(t)}\right)\phi_{\delta_{3}}\left(\frac{z}{\lambda_{4}(t)}\right)\,\mathrm{d}x\mathrm{d}y\mathrm{d}z.
\end{split}
\end{equation*}
Summarizing, for the cubic term we have proved that for all $\epsilon>0,$ the following inequality holds:
\begin{equation*}
\begin{split}
&\frac{1}{\eta(t)\lambda_{1}(t)}\int_{\mathbb{R}^{3}}|u|^{3}\phi_{\sigma'}\left(\frac{x}{\lambda_{1}(t)}\right)\phi_{\delta_{2}}\left(\frac{y}{\lambda_{3}(t)}\right)\phi_{\delta_{3}}\left(\frac{z}{\lambda_{4}(t)}\right)\,\mathrm{d}x\,\mathrm{d}y\,\mathrm{d}z\\
&\leq \frac{3 \widetilde{c} }{4\epsilon ^{1/3}\eta(t)\lambda_{1}(t)}  \int_{\mathbb{R}^{3}}\left|\nabla u\right|^{2}\phi_{\sigma'}\left(\frac{x}{\lambda_{1}(t)}\right)\phi_{\delta_{2}}\left(\frac{y}{\lambda_{3}(t)}\right)\phi_{\delta_{3}}\left(\frac{z}{\lambda_{4}(t)}\right)\,\mathrm{d}x\mathrm{d}y\mathrm{d}z \\
&\quad + \frac{C\epsilon}{4\eta(t)\lambda_{1}(t)}\int_{\mathbb{R}^{3}}u^{2}\phi_{\sigma'}\left(\frac{x}{\lambda_{1}(t)}\right)\phi_{\delta_{2}}\left(\frac{y}{\lambda_{3}(t)}\right)\phi_{\delta_{3}}\left(\frac{z}{\lambda_{4}(t)}\right)\,\mathrm{d}x\,\mathrm{d}y\,\mathrm{d}z,
\end{split}
\end{equation*}
where $C$ is a positive constant depending on $\widetilde{c}.$

\medskip

Note that, independent of the constant $\epsilon$ taken, the second term above satisfies:
\begin{equation*}
\frac{C\epsilon}{\eta(t)\lambda_{1}(t)}\int_{\mathbb{R}^{3}}u^{2}\phi_{\sigma'}\left(\frac{x}{\lambda_{1}(t)}\right)\phi_{\delta_{2}}\left(\frac{y}{\lambda_{3}(t)}\right)\phi_{\delta_{3}}\left(\frac{z}{\lambda_{4}(t)}\right)\,\mathrm{d}x\,\mathrm{d}y\,\mathrm{d}z\in L^{1}\left(\left\{
t\gg 1\right\}\right)
\end{equation*}
after choosing properly $\sigma'>0$ and using \eqref{main2.2}.
\medskip

Finally, we  obtain after  choosing $\epsilon>(3\widetilde{c}/2)^{3}$,  we obtain that
\begin{equation*}
\begin{split}
\int_{\{t\gg1 \}}\left(\frac{1}{\eta(t)\lambda_{1}(t)}\int_{\mathbb{R}^{3}}\left|\nabla u\right|^{2}\phi_{\sigma'}\left(\frac{x}{\lambda_{1}(t)}\right)\phi_{\delta_{2}}\left(\frac{y}{\lambda_{3}(t)}\right)\phi_{\delta_{3}}\left(\frac{z}{\lambda_{4}(t)}\right)\,\mathrm{d}x\mathrm{d}y\mathrm{d}z \right)\mathrm{d}t<\infty.
\end{split}
\end{equation*}
This result, together with \eqref{main2.2}, allow us to conclude Theorem \ref{Thmdim3H1}, in the same form as in the 2D case. Therefore,
\begin{equation*}
\liminf_{t\rightarrow \infty} \int_{\mathbb{R}^{3}}  \left(\left|\nabla u\right|^{2} +u^2 \right)(x,y,z,t)\phi_{\sigma'}\left(\frac{x}{\lambda_{1}(t)}\right)\phi_{\delta_{2}}\left(\frac{y}{\lambda_{3}(t)}\right)\phi_{\delta_{3}}\left(\frac{z}{\lambda_{4}(t)}\right)\,\mathrm{d}x\mathrm{d}y\mathrm{d}z =0.
\end{equation*}

\bigskip

\section{Decay in far far regions. Proof of Theorem \ref{Thmdim23L2}}\label{Sect:6}

This last section is devoted to the proof of Theorem \ref{Thmdim23L2}. As stated in the Introduction, the proof differs from the other proofs in this paper. In particular, we will need slightly different weighted functions, as stated in Section \ref{2}, Subsection \ref{2.3}. We will closely follow \cite{MPS}, with some key differences.

\subsection{$2D$ case} Fix any $p\geq 1$ and $\epsilon>0$. The proof consists of two independent decay estimates, one in a band of the form $|x|\sim t^p \log^{1+\epsilon} t$ and the other one for the band $|y|\sim t^p \log^{1+\epsilon} t$. Although both results are similar in nature, the proofs are slightly different, and some care is needed in both cases.

\medskip

\noindent
{\bf Case $|x|\sim t^p \log^{1+\epsilon} t$.} Let $\chi$ be the cut-off function introduced in \eqref{xi_especial}.  Additionally, we will consider  $\theta_{1}(t):=t^p \ln ^{1+\epsilon}t.$ So that,
it is clear that
\[
\frac{\theta_{1}'(t)}{\theta_{1}(t)}\sim\frac{1}{t}\quad \mbox{for}\quad t\gg 1.
\]
Note that unlike the previous analysis, the function $\theta_{1}^{-1}\in L^{1}\left(\{t\gg 1\}\right).$ So that, it suggests that  the proof will be obtained by exploiting  properties of the weighted function $\chi$  as we will see below.

\medskip

Firstly, we will estimate in the portion $x\sim -\theta_{1}(t).$  To estimate in the portion $x\sim\theta_{1}(t)$   the  procedure follows by  using an argument quite similar.

\medskip

Formally  we get after multiplying  the equation in \eqref{ZK:Eq} by $u\chi\left(\frac{x+\theta_{1}(t)}{\theta_{1}(t)}\right)$ that
\begin{equation}\label{partida}
\left(u\partial_{t}u+u\partial_{x}\Delta u + u^{2}\partial_{x}u\right)\chi\left(\frac{x+\theta_{1}(t)}{\theta_{1}(t)}\right)=0.
\end{equation}
Then after integrating in space we obtain the following identity:
\begin{equation}\label{6p2}
\begin{split}
&\frac{1}{2}\frac{\mathrm{d}}{\mathrm{d}t}\int_{\mathbb{R}^{2}} u^{2}\chi\left(\frac{x+\theta_{1}(t)}{\theta_{1}(t)}\right)\,\mathrm{d}x\,\mathrm{d}y \quad \underbrace{-\frac{1}{2}\int_{\mathbb{R}^{2}}u^{2}\chi'\left(\frac{x+\theta_{1}(t)}{\theta_{1}(t)}\right)\left(\frac{\theta_{1}'(t)}{\theta_{1}(t)}\right)\,\mathrm{d}x\,\mathrm{d}y}_{A_{1}(t)}\\
&\underbrace{+\frac{\theta_{1}'(t)}{2\theta_{1}(t)}\int_{\mathbb{R}^{2}}u^{2}\chi'\left(\frac{x+\theta_{1}(t)}{\theta_{1}(t)}\right)\left(\frac{x+\theta_{1}(t)}{\theta_{1}(t)}\right)\,\mathrm{d}x\,\mathrm{d}y}_{A_{2}(t)}\\
&
\underbrace{+\frac{3}{2\theta_{1}(t)}\int_{\mathbb{R}^{2}}\left(\partial_{x}u\right)^{2}\chi'\left(\frac{x+\theta_{1}(t)}{\theta_{1}(t)}\right)\mathrm{d}x\,\mathrm{d}y}_{A_{3}(t)}
\\
&\underbrace{+\frac{1}{2\theta_{1}(t)}\int_{\mathbb{R}^{2}}\left(\partial_{y}u\right)^{2}\chi'\left(\frac{x+\theta_{1}(t)}{\theta_{1}(t)}\right)\mathrm{d}x\,\mathrm{d}y}_{A_{4}(t)}\quad \underbrace{-\frac{1}{2\theta_{1}^{3}(t)}\int_{\mathbb{R}^{2}} u^{2}\chi'''\left(\frac{x+\theta_{1}(t)}{\theta_{1}(t)}\right)\mathrm{d}x\,\mathrm{d}y}_{A_{5}(t)}\\
&\underbrace{-\frac{1}{3\theta_{1}(t)}\int_{\mathbb{R}^{2}}u^{3}\chi'\left(\frac{x+\theta_{1}(t)}{\theta_{1}(t)}\right)\mathrm{d}x\,\mathrm{d}y}_{A_{6}(t)} \quad = \quad 0.
\end{split}
\end{equation}
Since  $\chi$ is monotone  decreasing,  then  $\chi'\left(\frac{x+\theta_{1}(t)}{\theta_{1}(t)}\right)\leq 0,$ and in   particular
$
\left(\frac{x+\theta_{1}(t)}{\theta_{1}(t)}\right)\chi'\left(\frac{x+\theta_{1}(t)}{\theta_{1}(t)}\right)\geq 0,$ so that, for $t\gg 1,$ the terms
$A_{1}(t) $ and $A_{2}(t)$ are positive.

\medskip

Instead, to estimate the terms $A_{3}$ and $A_{4}$ the scenario is quite different due to $\theta_{1}^{-1}\in L^{1}\left( \{t\gg 1\}\right).$ More precisely, we  obtain
\begin{equation*}
|A_{3}(t)|\lesssim_{\|u\|_{L^{\infty}_{t}H^{1}_{xy}}}\frac{1}{\theta_{1}(t)}\in L^{1}\left(\left\{t\gg 1\right\}\right).
\end{equation*}
Also,
\begin{equation*}
|A_{4}(t)|\lesssim_{\|u\|_{L^{\infty}_{t}H^{1}}}\frac{1}{\theta_{1}(t)}\in L^{1}\left(\left\{t\gg 1\right\}\right).
\end{equation*}
Next, 
\begin{equation*}
|A_{5}(t)|\lesssim_{\|u_{0}\|_{L^{2}}}\frac{1}{\theta_{1}^{3}(t)}\in L^{1}\left(\left\{t\gg 1\right\}\right).
\end{equation*}

To estimate $A_{6}$ we use the boundedness of the $H^1$ norm uniform in time and the fact that $\theta_{1}^{-1}\in L^{1}\left( \{t\gg 1\}\right)$; for the sake of brevity we omit the details here. In summary, we get
\begin{equation*}
|A_{6}(t)|\lesssim_{\|u\|_{L^{\infty}_{t}H^{1}}} \frac{1}{\theta_{1}(t)}\in L^{1}\left(\left\{t\gg 1\right\}\right).
\end{equation*}
Finally, we gather the previous estimates, that combined with the fact that $\theta_{1}^{-1}\notin L^{1}\left(\left\{t\gg 1\right\}\right)$  yield us to   conclude
that
\begin{equation}\label{s1}
\begin{split}
\int_{\{t\gg 1\}}\frac{1}{t}\left(\int_{\mathbb{R}^{2}}u^{2} \left| \chi'\left(\frac{x+\theta_{1}(t)}{\theta_{1}(t)}\right) \right| \,\mathrm{d}x\,\mathrm{d}y+
\int_{\mathbb{R}^{2}}u^{2}\chi'\left(\frac{x+\theta_{1}(t)}{\theta_{1}(t)}\right)\left(\frac{x+\theta_{1}(t)}{\theta_{1}(t)}\right)\,\mathrm{d}x\,\mathrm{d}y\right)\mathrm{d}t<\infty.
\end{split}
\end{equation}
Note that the  weighted function $\chi'\left(\frac{x+\theta_{1}(t)}{\theta_{1}(t)}\right)$ in the  first term in the l.h.s above  is supported on the region
$-2\theta_{1}(t)<x<-\theta_{1}(t),$ that is, $x\sim -\theta_{1}(t).$

\medskip

The lack of integrability of the function $t^{-1}$  for $t\gg1, $ implies that there exist a sequence of times for that the function in  \eqref{s1} converges to zero. More precisely, we ensure  that there exist a sequence of  positive times $(t_{n})_{n},$ such that $t_{n}\uparrow\infty$ as $n$ goes to infinity and satisfying
\begin{equation*}
\lim_{n\uparrow\infty}  \int_{\Lambda(t_{n})}u^{2}\,\mathrm{d}x\,\mathrm{d}y=0,
\end{equation*}
where ${\displaystyle \Lambda(t):=\left\{(x,y)\in\mathbb{R}^{2}\,|\, x\sim -\theta_{1}(t)\right\}}.$

\medskip

To prove the decay to zero in the right portion $x\sim \theta_{1}(t)$ is enough  to  apply  an argument similar  to  the  one described above,  but considering this time the  weighted function $\widetilde{\chi}(x):=\chi(-x),\,x\in\mathbb{R}.$

\medskip

Finally, we conclude \eqref{Strong} in this case by performing again an estimate as \eqref{6p2} but with weighted function $\tilde\chi$ nonnegative and supported in the interval $[-\frac34,\frac14]$, and using \eqref{bound_below}. The details are essentially in \cite{MPS}, and we skip them.

\medskip

\noindent
{\bf Case $|y|\sim t^p \log^{1+\epsilon} t$.} This case is similar to the previous one, with some differences. We have from \eqref{partida},
\begin{equation}\label{partida2}
\left(u\partial_{t}u+u\partial_{x}\Delta u + u^{2}\partial_{x}u\right)\chi\left(\frac{y+\theta_{1}(t)}{\theta_{1}(t)}\right)=0.
\end{equation}
Integrating in space and by parts, we get now
\begin{equation}\label{6p3}
\begin{split}
&\frac{1}{2}\frac{\mathrm{d}}{\mathrm{d}t}\int_{\mathbb{R}^{2}} u^{2}\chi\left(\frac{y+\theta_{1}(t)}{\theta_{1}(t)}\right)\,\mathrm{d}x\,\mathrm{d}y \quad \underbrace{-\frac{1}{2}\int_{\mathbb{R}^{2}}u^{2}\chi'\left(\frac{y+\theta_{1}(t)}{\theta_{1}(t)}\right)\left(\frac{\theta_{1}'(t)}{\theta_{1}(t)}\right)\,\mathrm{d}x\,\mathrm{d}y}_{A_{1}(t)}\\
&\underbrace{+\frac{\theta_{1}'(t)}{2\theta_{1}(t)}\int_{\mathbb{R}^{2}}u^{2}\chi'\left(\frac{y+\theta_{1}(t)}{\theta_{1}(t)}\right)\left(\frac{y+\theta_{1}(t)}{\theta_{1}(t)}\right)\,\mathrm{d}x\,\mathrm{d}y}_{A_{2}(t)}\\
&
\underbrace{+\frac{1}{\theta_{1}(t)}\int_{\mathbb{R}^{2}}\left(\partial_{x}u\partial_{y}u \right) \chi'\left(\frac{y+\theta_{1}(t)}{\theta_{1}(t)}\right)\mathrm{d}x\,\mathrm{d}y}_{A_{3}(t)}
\quad = \quad 0.
\end{split}
\end{equation}
The only new term here is $A_3$, for which we immediately have
\begin{equation*}
|A_{3}(t)|\lesssim_{\|u\|_{L^{\infty}_{t}H^{1}_{xy}}}\frac{1}{\theta_{1}(t)}\in L^{1}\left(\left\{t\gg 1\right\}\right).
\end{equation*}
The rest of the proof is the same as in the previous case.

\subsection{$3D$ case}

Similar to the $2D$ case,  we will describe how to obtain the decay to zero in the region $x\sim -\theta_{1}(t)$ and $y\sim -\theta_1(t)$. The decay in the remaining regions (in particular, in the $|z|\sim \theta_1(t)$ can be obtained by using quite similar arguments.

\medskip

As usual our starting point  is based on  energy estimates. So that,   a standard  procedure  allow us  to obtain the identity
\begin{equation*}
\begin{split}
&\frac{1}{2}\frac{\mathrm{d}}{\mathrm{d}t}\int_{\mathbb{R}^{3}} u^{2}\chi\left(\frac{x+\theta_{1}(t)}{\theta_{1}(t)}\right)\,\mathrm{d}x\,\mathrm{d}y\,\mathrm{d}z \quad \underbrace{-\frac{1}{2}\int_{\mathbb{R}^{3}}u^{2}\chi'\left(\frac{x+\theta_{1}(t)}{\theta_{1}(t)}\right)\left(\frac{\theta_{1}'(t)}{\theta_{1}(t)}\right)\,\mathrm{d}x\,\mathrm{d}y\,\mathrm{d}z}_{A_{1}(t)}\\
& \underbrace{+ \frac{\theta_{1}'(t)}{2\theta_{1}(t)}\int_{\mathbb{R}^{3}}u^{2}\chi'\left(\frac{x+\theta_{1}(t)}{\theta_{1}(t)}\right)\left(\frac{x+\theta_{1}(t)}{\theta_{1}(t)}\right)\,\mathrm{d}x\,\mathrm{d}y\mathrm{d}z}_{A_{2}(t)}\\
&  \underbrace{+ \frac{3}{2\theta_{1}(t)}\int_{\mathbb{R}^{3}}\left(\partial_{x}u\right)^{2}\chi'\left(\frac{x+\theta_{1}(t)}{\theta_{1}(t)}\right)\mathrm{d}x\,\mathrm{d}y\mathrm{d}z}_{A_{3}(t)}
\\
&+ \underbrace{\frac{1}{2\theta_{1}(t)}\int_{\mathbb{R}^{3}}\left(\partial_{y}u\right)^{2}\chi'\left(\frac{x+\theta_{1}(t)}{\theta_{1}(t)}\right)\mathrm{d}x\,\mathrm{d}y\,\mathrm{d}z}_{A_{4}(t)} \quad \underbrace{+\frac{1}{2\theta_{1}(t)}\int_{\mathbb{R}^{3}}\left(\partial_{z}u\right)^{2}\chi'\left(\frac{x+\theta_{1}(t)}{\theta_{1}(t)}\right)\mathrm{d}x\,\mathrm{d}y\,\mathrm{d}z}_{A_{5}(t)}\\
&\underbrace{-\frac{1}{2\theta_{1}^{3}(t)}\int_{\mathbb{R}^{3}} u^{2}\chi'''\left(\frac{x+\theta_{1}(t)}{\theta_{1}(t)}\right)\mathrm{d}x\,\mathrm{d}y\,\mathrm{d}z}_{A_{6}(t)} \quad\underbrace{-\frac{1}{3\theta_{1}(t)}\int_{\mathbb{R}^{3}}u^{3}\chi'\left(\frac{x+\theta_{1}(t)}{\theta_{1}(t)}\right)\mathrm{d}x\,\mathrm{d}y\,\mathrm{d}z}_{A_{7}(t)} \quad =\quad 0.
\end{split}
\end{equation*}
The constraints on $\chi$ implies that  for $t\gg1, $ the terms $A_{1}$ and $A_{2}$ are positive. Also, the terms $A_{3},A_{4},$ and $A_{5}$ satisfy the bounds
\begin{equation*}
A_{3}(t)\lesssim_{\|u\|_{L^{\infty}_{t}H^{1}}}\frac{1}{\theta_{1}(t)}\in L^{1}\left(\left\{t\gg 1\right\}\right),
\end{equation*}
\begin{equation*}
A_{4}(t)\lesssim_{\|u\|_{L^{\infty}_{t}H^{1}}}\frac{1}{\theta_{1}(t)}\in L^{1}\left(\left\{t\gg 1\right\}\right),
\end{equation*}
and
\begin{equation*}
A_{5}(t)\lesssim_{\|u\|_{L^{\infty}_{t}H^{1}}}\frac{1}{\theta_{1}(t)}\in L^{1}\left(\left\{t\gg 1\right\}\right).
\end{equation*}
For $A_{6}$ we have that
\begin{equation*}
A_{6}(t)\lesssim_{\|u_{0}\|_{L^{2}}}\frac{1}{\theta_{1}^{3}(t)}\in L^{1}\left(\left\{t\gg 1\right\}\right).
\end{equation*}
To estimate $A_{7}$ we can use the boundedness of the $H^1$ norm in time, exactly as in the 2D case. In summary, we get
\begin{equation*}
A_{7}(t)\lesssim_{\|u\|_{L^{\infty}_{t}H^{1}}} \frac{1}{\theta_{1}(t)}\in L^{1}\left(\left\{t\gg 1\right\}\right).
\end{equation*}
Finally, after gathering  the estimates  in this step that combined with the fact that $\theta_{1}^{-1}\notin L^{1}\left(\left\{t\gg 1\right\}\right)$  yield us to   conclude
that
\begin{equation}\label{e1.1.1}
\begin{split}
& \int_{\{t\gg 1\}}\frac{1}{t}\left(-\int_{\mathbb{R}^{3}}u^{2}\chi'\left(\frac{x+\theta_{1}(t)}{\theta_{1}(t)}\right)\,\mathrm{d}x\,\mathrm{d}y\,\mathrm{d}z\right.\\
& \qquad\qquad \quad\left. +\int_{\mathbb{R}^{3}}u^{2}\chi'\left(\frac{x+\theta_{1}(t)}{\theta_{1}(t)}\right)\left(\frac{x+\theta_{1}(t)}{\theta_{1}(t)}\right)\,\mathrm{d}x\,\mathrm{d}y\,\mathrm{d}z\right)\mathrm{d}t<\infty.
\end{split}
\end{equation}
From this point, the rest of the proof is the same as in the 2D case.

\medskip

Finally, the case $y\sim -\theta_1$ is obtained  in similar terms starting from \eqref{partida2}. We skip the details.

\bigskip

\section{Proof of Theorem \ref{Thmdim2L2} in the gKdV case}\label{A}

In this appendix our goal will be in  provide a  proof of the gKdV version of Theorem \ref{Thmdim2L2}, namely Theorem \ref{Thmdim2L2_KDV}. This new result complements \cite{MR3936126}.

\medskip

\begin{proof}[Proof of Theorem \ref{Thmdim2L2_KDV}]
We sketch the proof, following the proof of Theorem \ref{Thmdim2L2}. Using \eqref{XI2D}, we consider this time the functional
\begin{equation*}
\Xi(t):=\frac{1}{\eta(t)}\int_{\mathbb{R}}u(x,t)\psi_{\sigma}\left(\frac{\tilde x}{\lambda_{1}(t)}\right)\phi_{\delta_{1}}\left(\frac{\tilde x}{\lambda_{1}^q(t)}\right)\,\mathrm{d}x, \quad \tilde x:= x-\rho(t),
\end{equation*}
for $\lambda_1$, $\rho$ and $\eta$ given by
\begin{equation}\label{defns_gKdV}
\begin{aligned}
\lambda_{1}(t)=&~{} \frac{t^{b}}{\ln t},\quad \rho(t) =\pm t^n ,\quad \mbox{and} \quad \eta(t)=t^{m}\ln ^{2} t, \quad m +b=1,\quad b,m>0,\\
0<b \leq &~{} \min\left\{ \frac{p}{p+q(p-1)}, \frac{2}{2+q} , \frac{p}{2p-1}\right\}, \quad q>1, \quad 0\leq n \leq 1-\frac{b}2.
 \end{aligned}
\end{equation}
Also, \eqref{Ap7} and \eqref{Ap8} are satisfied.

\medskip

Here we have two cases. If $p=2$, then $b<\frac23$, by taking $q=1+\epsilon_0$, $\epsilon_0$ arbitrarily small. If $p=4$, one has $\frac{p}{p+q(p-1)} < \frac{2}{2+q}$ and we conclude $b<\frac47$ performing the same trick as before. We conclude $b<\frac{p}{2p-1}$ as the condition for the validity of Theorem \ref{Thmdim2L2_KDV}, exactly as in \eqref{final_gKdV}.

\medskip

Now, we estimate $\Xi$ and its derivative in time. First of all, by Cauchy-Schwarz inequality and \eqref{defns_gKdV} we obtain
\[
\sup_{t\gg 1} |\Xi(t)| \lesssim \sup_{t\gg 1} \frac{\left(\lambda_{1}^q(t)\right)^{1/2}}{\eta(t)} \lesssim  \sup_{t\gg 1} \frac1{t^{1-b-\frac12 bq} \ln^{2+\frac{q}2} t} <\infty. 
\]
In what follows, we compute and estimate the dynamics of $\Xi(t)$ in the long time regime. We will prove for $p=2,4$, and $C_0>0$,
\begin{equation}\label{dT_ppal_gKdV}
\begin{split}
\frac{1}{\lambda_{1}(t)\eta(t)}\int_{\mathbb{R}} u^{p} \psi_{\sigma}'\left(\frac{\tilde x}{\lambda_{1}(t)}\right)\phi_{\delta_{1}}\left(\frac{\tilde x}{\lambda_{1}^q(t)}\right)\,\mathrm{d}x &\le C_0 \frac{d\Xi}{dt}(t) + \Xi_{int}(t),
\end{split}
\end{equation}
where $\Xi_{int}(t)$ are terms that belong to $L^1\left(\{t\gg 1\}\right
)$. Once this result is proved, the rest of the proof is direct. We have
\begin{equation}\label{separation}
\begin{split}
& \frac{\mathrm{d}}{\mathrm{d}t}\Xi(t)\\
&=\frac{1}{\eta(t)}\int_{\mathbb{R}}\partial_{t}\left(u\psi_{\sigma}\left(\frac{\tilde x}{\lambda_{1}(t)}\right)\phi_{\delta_{1}}\left(\frac{\tilde x}{\lambda_{1}^q(t)}\right)\right)\,\mathrm{d}x -\frac{\eta'(t)}{\eta^{2}(t)}\int_{\mathbb{R}}u\psi_{\sigma}\left(\frac{\tilde{x}}{\lambda_{1}(t)}\right)\phi_{\delta_{1}}\left(\frac{\tilde x}{\lambda_{1}^q(t)}\right)\,\mathrm{d}x\\
&=: \Xi_{1}(t)+ \Xi_{2}(t).
\end{split}
\end{equation}
First, we bound  $\Xi_{2}$. In virtue of \eqref{eq1} the same analysis applied there  yields
\begin{equation}\label{2}
\begin{split}
|\Xi_{2}(t)|&\leq\left|\frac{\eta'(t)}{\eta^{2}(t)}\int_{\mathbb{R}}u\psi_{\sigma}\left(\frac{\tilde x}{\lambda_{1}(t)}\right)\phi_{\delta_{1}}\left(\frac{\tilde x}{\lambda^q_{1}(t)}\right)\,\mathrm{d}x\right| \lesssim \frac{(\lambda_1(t))^{q/2}}{t \eta(t)} =\frac{1}{t^{2-b -\frac{1}{2}bq} \ln^{2 +\frac{q}2} t}.
\end{split}
\end{equation}
We need $2-b -\frac{1}{2}bq \geq 1$, which is $b\leq \frac{2}{2+q}.$ Since \eqref{defns_gKdV} hold, the last term integrates. Thus,  $\Xi_{2}\in L^1(\{t\gg1\})$. Now,
\begin{equation}\label{separation_Xi_gKdV}
\begin{split}
\Xi_{1}(t)&=\frac{1}{\eta(t)}\int_{\mathbb{R}}\partial_{t}u\psi_{\sigma}\left(\frac{\tilde x}{\lambda_{1}(t)}\right)\phi_{\delta_{1}}\left(\frac{\tilde x}{\lambda_{1}^q(t)}\right)\,\mathrm{d}x \\
&\quad -\frac{\lambda_{1}'(t)}{\lambda_{1}(t)\eta(t)}\int_{\mathbb{R}}u \psi_{\sigma}'\left(\frac{\tilde x}{\lambda_{1}(t)}\right)\left(\frac{\tilde x}{\lambda_{1}(t)}\right)\phi_{\delta_{1}}\left(\frac{\tilde x}{\lambda_{1}^q(t)}\right)\,\mathrm{d}x\\
&\quad -q\frac{\lambda_{1}'(t)}{\lambda_{1}(t)\eta(t)}\int_{\mathbb{R}}u \psi_{\sigma}\left(\frac{\tilde x}{\lambda_{1}(t)}\right)\left(\frac{\tilde x}{\lambda_{1}^q(t)}\right)\phi_{\delta_{1}}'\left(\frac{\tilde x}{\lambda_{1}^q(t)}\right)\,\mathrm{d}x\\
&\quad -\frac{\rho'(t)}{\lambda_{1}(t)\eta(t)}\int_{\mathbb{R}}u \psi_{\sigma}'\left(\frac{\tilde x}{\lambda_{1}(t)}\right)\phi_{\delta_{1}}\left(\frac{\tilde x}{\lambda_{1}^q(t)}\right)\,\mathrm{d}x\\
&\quad -\frac{\rho'(t)}{\lambda_{1}^q(t)\eta(t)}\int_{\mathbb{R}}u \psi_{\sigma}\left(\frac{\tilde x}{\lambda_{1}(t)}\right)\phi_{\delta_{1}}'\left(\frac{\tilde x}{\lambda_{1}^q(t)}\right)\,\mathrm{d}x\\
&=: \Xi_{1,1}(t)+\Xi_{1,2}(t)+\Xi_{1,3}(t) +\Xi_{1,4}(t) +\Xi_{1,5}(t).
\end{split}
\end{equation}
Concerning to $\Xi_{1,1}$ we have by \eqref{gKdV}  and integration by  parts
\begin{equation}\label{separation_Xi1_gKdV}
\begin{split}
\Xi_{1,1}(t)&=-\frac{1}{\eta(t)}\int_{\mathbb{R}}\partial_{x}\left(\partial_x^2 u+u^p \right)\psi_{\sigma}\left(\frac{\tilde x}{\lambda_{1}(t)}\right)\phi_{\delta_{1}}\left(\frac{\tilde x}{\lambda^q_{1}(t)}\right)\,\mathrm{d}x\\
&=\frac{1}{\eta(t)\lambda_{1}(t)}\int_{\mathbb{R}}\partial_x^2 u\psi_{\sigma}'\left(\frac{\tilde x}{\lambda_{1}(t)}\right)\phi_{\delta_{1}}\left(\frac{\tilde x}{\lambda^q_{1}(t)}\right)\,\mathrm{d}x\\
& \quad +\frac{1}{\eta(t)\lambda^q_{1}(t)}\int_{\mathbb{R}} \partial_x^2 u\psi_{\sigma}\left(\frac{\tilde x}{\lambda_{1}(t)}\right)\phi_{\delta_{1}}'\left(\frac{\tilde x}{\lambda^q_{1}(t)}\right)\,\mathrm{d}x\\
&\quad +\frac{1}{\eta(t)\lambda_{1}(t)}\int_{\mathbb{R}}  u^{p}\psi_{\sigma}'\left(\frac{\tilde x}{\lambda_{1}(t)}\right)\phi_{\delta_{1}}\left(\frac{\tilde x}{\lambda_{1}^q(t)}\right)\,\mathrm{d}x \\
& \quad +\frac{1}{\eta(t)\lambda^q_{1}(t)}\int_{\mathbb{R}}  u^{p}\psi_{\sigma}\left(\frac{\tilde x}{\lambda_{1}(t)}\right)\phi_{\delta_{1}}'\left(\frac{\tilde x}{\lambda^q_{1}(t)}\right)\,\mathrm{d}x\\
&=: \Xi_{1,1,1}(t)+\Xi_{1,1,2}(t)+\Xi_{1,1,3}(t)+\Xi_{1,1,4}(t).
\end{split}
\end{equation}
For  $\Xi_{1,1,1}$ we have  after combining  integration by parts
\begin{equation*}
\begin{split}
\Xi_{1,1,1}(t)&=\frac{1}{\eta(t)\lambda_{1}^{3}(t)}\int_{\mathbb{R}}u\psi'''_{\sigma}
\left(\frac{\tilde x}{\lambda_{1}(t)}\right)\phi_{\delta_{1}}\left(\frac{\tilde x}{\lambda^q_{1}(t)}\right)\,\mathrm{d}x \\
& \quad +\frac{2}{\eta(t)\lambda_{1}^{2+q}(t)}\int_{\mathbb{R}}u\psi''_{\sigma}\left(\frac{\tilde x}{\lambda_{1}(t)}\right)\phi_{\delta_{1}}'\left(\frac{\tilde x}{\lambda^q_{1}(t)}\right)\mathrm{d}x\\
&\quad +\frac{1}{\eta(t)\lambda_{1}^{1+2q}(t)}\int_{\mathbb{R}}u\psi'_{\sigma}\left(\frac{\tilde x}{\lambda_{1}(t)}\right)\phi_{\delta_{1}}''\left(\frac{\tilde x}{\lambda^q_{1}(t)}\right)\mathrm{d}x.
\end{split}
\end{equation*}
First, we bound each term using Cauchy-Schwarz inequality, as follows: Since $q>1$,
\[
    |\Xi_{1,1,1}(t)|\lesssim  \frac{1}{\eta(t)\lambda_1^{5/2}(t)}+\frac{1}{\eta(t)\lambda_1^{3/2+q}(t)}+\frac{1}{\eta(t)\lambda_1^{1/2+2q}(t)} \lesssim  \frac{1}{\eta(t)\lambda_1^{5/2}(t)},
\]
which clearly integrates. Next, applying integration by parts,
\[
\begin{aligned}
 \Xi_{1,1,2}(t) = &~{} \frac{1}{\eta(t)\lambda_{1}^q(t)}\int_{\mathbb{R}} \partial_x^2 u\psi_{\sigma}\left(\frac{\tilde x}{\lambda_{1}(t)}\right)\phi_{\delta_{1}}'\left(\frac{\tilde x}{\lambda_{1}^q(t)}\right)\,\mathrm{d}x\\
 =&~{} \frac{1}{\eta(t)\lambda_{1}^{2+q}(t)}\int_{\mathbb{R}}u\psi_{\sigma}''\left(\frac{\tilde x}{\lambda_{1}(t)}\right)\phi_{\delta_{1}}'\left(\frac{\tilde x}{\lambda_{1}^q(t)}\right)\,\mathrm{d}x \\
&~{}   +\frac{2}{\eta(t)\lambda_{1}^{1+2q}(t)}\int_{\mathbb{R}}u\psi_{\sigma}'\left(\frac{\tilde x}{\lambda_{1}(t)}\right)\phi_{\delta_{1}}''\left(\frac{\tilde x}{\lambda_{1}^q(t)}\right)\,\mathrm{d}x \\
&~{} +\frac{1}{\eta(t)\lambda_{1}^{3q}(t)}\int_{\mathbb{R}}u\psi_{\sigma}\left(\frac{\tilde x}{\lambda_{1}(t)}\right)\phi_{\delta_{1}}'''\left(\frac{\tilde x}{\lambda_{1}^q(t)}\right)\,\mathrm{d}x,
\end{aligned}
\]
and the Cauchy-Schwarz inequality yields
\[
  \begin{aligned}
    |\Xi_{1,1,2}(t)|     \lesssim& {}~ \frac{1}{\eta(t)(\lambda_1(t))^{3/2+q}}+\frac{1}{\eta(t)(\lambda_1(t))^{1/2+2q}}+\frac{1}{\eta(t)(\lambda_1(t))^{5q/2}}.
  \end{aligned}
\]
Each term above integrates since $q>1$.
%
%
%
We emphasize that the term $\Xi_{1,1,3}$ in \eqref{separation_Xi1_gKdV}
\[
\Xi_{1,1,3}(t)= \frac{1}{2\eta(t)\lambda_{1}(t)}\int_{\mathbb{R}}  u^{p}\psi_{\sigma}'\left(\frac{\tilde x}{\lambda_{1}(t)}\right)\phi_{\delta_{1}}\left(\frac{\tilde x}{\lambda_{1}^q(t)}\right)\,\mathrm{d}x,
\]
is the term to be estimated after integrating in time. Therefore, it will be taken until the end of the proof.

\medskip

The therm $\Xi_{1,1,4}$ in \eqref{separation_Xi1_gKdV} satisfies de following estimate
\[
\begin{split}
|\Xi_{1,1,4}(t)|&\leq\left|\frac{1}{2\eta(t)\lambda_{1}^q(t)}\int_{\mathbb{R}}  u^{p}\psi_{\sigma}\left(\frac{\tilde x}{\lambda_{1}(t)}\right)\phi_{\delta_{1}}'\left(\frac{\tilde x}{\lambda_{1}^q(t)}\right)\,\mathrm{d}x\right| \lesssim \frac{1}{2\eta(t)\lambda_1^q(t)}.
\end{split}
\]
Since $q>1$, $\Xi_{1,1,4}\in L^1(\{t\gg1\})$. 

\medskip

Now, we focus our attention in the remaining terms in \eqref{separation_Xi_gKdV}. First, by means of Young's inequality,  we have for   $\epsilon>0$, $p=2,4$ and $p'=\frac{p}{p-1}$,
\begin{equation*}
\begin{split}
\left| \Xi_{1,2}(t)  \right| &= \left| \frac{\lambda_{1}'(t)}{\lambda_{1}(t)\eta(t)}\int_{\mathbb{R}}u \psi_{\sigma}'\left(\frac{\tilde x}{\lambda_{1}(t)}\right)\left(\frac{\tilde x}{\lambda_{1}(t)}\right)\phi_{\delta_{1}}\left(\frac{\tilde x}{\lambda_{1}^q(t)}\right)\,\mathrm{d}x \right| \\
&\leq \frac{1}{4\epsilon^{p/2}}\left|\frac{\lambda_{1}'(t)}{\lambda_{1}(t)\eta(t)}\right|\int_{\mathbb{R}}u^{p}\psi_{\sigma}'\left(\frac{\tilde x}{\lambda_{1}(t)}\right)
\phi_{\delta_{1}}\left(\frac{\tilde x}{\lambda_{1}^q(t)}\right)\,\mathrm{d}x\\
&\quad +c\epsilon^{p'/2}\left|\frac{\lambda_{1}'(t)}{\lambda_{1}(t)\eta(t)}\right|\int_{\mathbb{R}}
\psi_{\sigma}'\left(\frac{\tilde x}{\lambda_{1}(t)}\right)\left| \frac{\tilde x}{\lambda_{1}(t)}\right|^{p'}\phi_{\delta_{1}}\left(\frac{\tilde x}{\lambda_{1}^q(t)}\right)\,\mathrm{d}x\\
&\leq\frac{1}{4\epsilon^{p/2}}\left|\frac{\lambda_{1}'(t)}{\lambda_{1}(t)\eta(t)}\right|\int_{\mathbb{R}}u^{p}\psi_{\sigma}'\left(\frac{\tilde x}{\lambda_{1}(t)}\right)
\phi_{\delta_{1}}\left(\frac{\tilde x}{\lambda_{1}^q(t)}\right)\,\mathrm{d}x+c\epsilon^{p'/2}\left|\frac{\lambda_{1}'(t)}{\eta(t)}\right|,\\
\end{split}
\end{equation*}
so that, taking $\epsilon^{p/2}=\lambda'(t)>0$ for $t\gg1;$ it is clear that
\begin{equation*}
\begin{split}
\left|\Xi_{1,2}(t)\right|&\le \frac{1}{4\lambda_{1}(t)\eta(t)}\int_{\mathbb{R}}u^{2}\psi_{\sigma}'\left(\frac{\tilde x}{\lambda_{1}(t)}\right)
\phi_{\delta_{1}}\left(\frac{\tilde x}{\lambda_{1}^q(t)}\right)\,\mathrm{d}x +c\frac{\left(\lambda_{1}'(t)\right)^{1+\frac1{p-1}}}{\eta(t)} \\
&=:\frac{1}{4}\Xi_{1,1,3}(t)+\Xi^*_{1,2}(t).
\end{split}
\end{equation*}
Note that the first term in the r.h.s. is the quantity to be estimated.  The remaining term $\Xi^*_{1,2}$ integrates since $b\leq \frac{p}{2p-1}$, see \eqref{defns_gKdV}.

\medskip

Now, we consider the term $\Xi_{1,3}(t)$. Combining the properties attribute to $\phi_{\delta_{1}}$, the fact that $u\in L^2$ for $p=2$ and $u\in L^4$ for $p=4$, and Young's inequality we get for $\theta(t)=t^{2m/p}$,
\[
\begin{split}
\left| \Xi_{1,3}(t) \right|   &= q\left|\frac{\lambda_{1}'(t)}{\lambda_{1}(t)\eta(t)}\int_{\mathbb{R}}u \psi_{\sigma}\left(\frac{\tilde x}{\lambda_{1}(t)}\right)\left(\frac{\tilde x}{\lambda_{1}^q(t)}\right)\phi_{\delta_{1}}'\left(\frac{\tilde x}{\lambda^q_{1}(t)}\right)\,\mathrm{d}x \right| \\
&\leq q\left|\frac{\lambda_{1}'(t)}{\lambda_{1}(t)\eta(t)}\right|\int_{\mathbb{R}}|u| \psi_{\sigma}\left(\frac{\tilde x}{\lambda_{1}(t)}\right)\left|\frac{\tilde x}{\lambda_{1}^q(t)}\right|\left|\phi^\prime_{\delta_{1}}\left(\frac{\tilde x}{\lambda_{1}^q(t)}\right)\right| \,\mathrm{d}x \\
&\leq \left|\frac{\lambda_{1}'(t) \theta^{p/2}(t)}{\lambda_{1}(t)\eta(t)}\right|\int_{\mathbb{R}}u^{p} \psi_{\sigma}^{p}\left(\frac{\tilde x}{\lambda_{1}(t)}\right) \,\mathrm{d}x \\
&\quad +\left|\frac{\lambda_{1}'(t)}{\lambda_{1}(t)\eta(t)\theta^{p'/2}(t)}\right|\int_{\mathbb{R}}\left(\frac{\tilde x}{\lambda_{1}^q(t)}\right)^{p'}\left|\phi^\prime_{\delta_{1}}\left(\frac{\tilde x}{\lambda_{1}^q(t)}\right)\right|^{p'}\mathrm{d}x \\
&\lesssim \left|\frac{\lambda_{1}'(t)\theta^{p/2}(t)}{\lambda_{1}(t)\eta(t)}\right|   + \left|\frac{\lambda_{1}'(t)\lambda_1^{q}(t)}{\lambda_1(t)\eta(t)\theta^{p'/2}(t)}\right|.
\end{split}
\]
Hence,
\begin{equation}\label{13}
\begin{split}
\left| \Xi_{1,3}(t) \right|   &\lesssim \frac{1}{t\ln^2t}+\frac{1}{t^{2+\frac1{p-1} -b(1+q+\frac1{p-1})} \ln^{2+q} t},
\end{split}
\end{equation}
which integrates if $b\leq \frac{p}{p+q(p-1)}$, which is just \eqref{defns_gKdV}. Thus, $\Xi_{1,3}\in L^1(\{t\gg1\})$.

\medskip

Now,
\[
|\Xi_{1,4}(t)| \lesssim \frac{|\rho'(t)|}{\eta(t) \lambda_1^{1/2}(t)} \lesssim \frac1{t^{2-\frac b2-n} \ln^{\frac32} t},
\]
which integrates since $n\leq 1- \frac b2$. Finally,
\[
|\Xi_{1,5}(t)| \lesssim \frac{|\rho_1'(t)|}{\eta(t) \lambda_1^{q/2}(t)} \ll  \frac1{t^{2-\frac b2-n} \ln^{\frac32} t}\in L^1(\{t\gg1\}),
\]
since $n\leq 1-\frac{b}{2}.$

Therefore, once again we conclude the integrability in time. The rest of the proof is the same as in \eqref{ZK:Eq}, and we omit the details.

%
%
\end{proof}

\bigskip

{\footnotesize
\bibliographystyle{plain}
\bibliography{Ref-ZK.bib}
}

\end{document}